\documentclass[a4paper,11pt]{article}

\usepackage[margin=3.2cm]{geometry}

\usepackage{graphics}

\usepackage{docmute}

\usepackage[utf8]{inputenc}

\usepackage{mathtools}

\usepackage{amsmath}

\usepackage{amsfonts}

\usepackage{amssymb}

\usepackage{amsthm}
\usepackage{aliascnt}

\usepackage{makeidx}
\makeindex

\usepackage{textcomp}
\usepackage[sb]{libertine}
\usepackage[varqu,varl]{zi4}%
\usepackage[libertine,bigdelims,vvarbb]{newtxmath}
\usepackage[supstfm=libertinesups,supscaled=1.2,raised=-.13em]{superiors}
\usepackage{bm}
\useosf

\usepackage[breaklinks,linktocpage]{hyperref}
\usepackage[hyperpageref]{backref}

\usepackage{url}
\usepackage{breakurl}

\usepackage{tikz-cd}

\usepackage{enumitem}
\usepackage[UKenglish]{babel}
\usepackage[makeroom]{cancel}
\usepackage{pstricks-add}

\definecolor{shadecolor}{rgb}{0.88,0.91,0.95}       
\usepackage{tcolorbox}

\usepackage{dynkin-diagrams}
\usetikzlibrary{backgrounds}

\usepackage{booktabs}

\usepackage{fancyhdr}
\newcommand{\Z}{\mathbb{Z}}
\newcommand{\N}{\mathbb{N}}
\newcommand{\R}{\mathbb{R}}
\newcommand{\C}{\mathbb{C}}

\newcommand{\Ker}{\operatorname{Ker}}

\newcommand{\diag}{\operatorname{diag}}

\renewcommand{\Im}{\operatorname{Im}}

\newcommand{\End}{\operatorname{End}}
\newcommand{\Hom}{\operatorname{Hom}}
\newcommand{\Aff}{\operatorname{Aff}}
\newcommand{\Iso}{\operatorname{Iso}}
\newcommand{\Sym}{\operatorname{Sym}}
\newcommand{\Aut}{\operatorname{Aut}}

\newcommand{\Hol}{\operatorname{Hol}}

\newcommand{\Rm}{\operatorname{Rm}}

\newcommand{\GL}{\operatorname{GL}}

\newcommand{\SU}{\operatorname{SU}}
\newcommand{\SO}{\operatorname{SO}}
\renewcommand{\O}{\operatorname{O}}

\newcommand{\so}{\mathfrak{so}}
\newcommand{\U}{\operatorname{U}}
\newcommand{\PU}{\operatorname{PU}}

\newcommand{\Lie}{\operatorname{Lie}}

\newcommand{\CP}{\mathbb{CP}}
\newcommand{\RP}{\mathbb{RP}}

\newcommand{\Ad}{\operatorname{Ad}}

\newcommand{\tr}{\operatorname{tr}}

\newcommand{\Id}{\operatorname{Id}}

\renewcommand{\|}[1]{\left| \left| #1 \right| \right|}

\newcommand{\<}{\langle}
\renewcommand{\>}{\rangle}

\newcommand{\mmod}{{/ \!\! / \!\! / \!}}

\newcommand{\tensor}{\otimes}

\newcommand{\ddt}{\frac{d}{dt}}

\newcommand{\vol}{\operatorname{vol}}

\newcommand{\inj}{\operatorname{inj}}
\newcommand{\Isom}{\operatorname{Isom}}

\newcommand{\del}{\partial}
\newcommand{\delbar}{\overline{\partial}}

\newcommand{\loc}{\operatorname{loc}}

\newcommand{\fix}{\operatorname{fix}}
\newcommand{\orb}{\operatorname{orb}}
\newcommand{\asd}{\operatorname{asd}}

\newcommand{\ind}{\operatorname{ind}}

\renewcommand{\d}{\operatorname{d} \!}

\renewcommand{\tilde}[1]{\widetilde{#1}}

\newcommand{\Fr}{\operatorname{Fr}}

\newcommand{\XEH}{X_{\text{EH}}}
\newcommand{\XEHh}{\hat{X}_{\text{EH}}}

\newcommand\Item[1][]{%
  \ifx\relax#1\relax  \item \else \item[#1] \fi
  \abovedisplayskip=0pt\abovedisplayshortskip=0pt~\vspace*{-\baselineskip}}

\author{Daniel Platt}
\date{\today}
\title{$G_2$-instantons on resolutions of $G_2$-orbifolds}

\usepackage[capitalise]{cleveref}

\usepackage{etoolbox,xcolor}
\makeatletter
\newcommand{\undefinedlabel}{\color{red}[?]}
\patchcmd{\@@setcref}         {??}{\undefinedlabel}{}{}
\patchcmd{\@@setcref}         {??}{\undefinedlabel}{}{}
\patchcmd{\@@setcrefrange}    {??}{\undefinedlabel}{}{}
\patchcmd{\@@setcrefrange}    {??}{\undefinedlabel}{}{}
\patchcmd{\@@setcrefrange}    {??}{\undefinedlabel}{}{}
\patchcmd{\@@setcrefrange}    {??}{\undefinedlabel}{}{}
\patchcmd{\@@setcrefrange}    {??}{\undefinedlabel}{}{}
\patchcmd{\@@setcrefrange}    {??}{\undefinedlabel}{}{}
\patchcmd{\@@setnamecref}     {??}{\undefinedlabel}{}{}
\patchcmd{\@@setnamecref}     {??}{\undefinedlabel}{}{}
\patchcmd{\@@setcpageref}     {??}{\undefinedlabel}{}{}
\patchcmd{\@@setcpageref}     {??}{\undefinedlabel}{}{}
\patchcmd{\@@setcpagerefrange}{??}{\undefinedlabel}{}{}
\patchcmd{\@@setcpagerefrange}{??}{\undefinedlabel}{}{}
\patchcmd{\@@setcpagerefrange}{??}{\undefinedlabel}{}{}
\patchcmd{\@@setcpagerefrange}{??}{\undefinedlabel}{}{}
\patchcmd{\@@setcpagerefrange}{??}{\undefinedlabel}{}{}
\patchcmd{\@@cref}            {??}{\undefinedlabel}{}{}
\makeatother

\numberwithin{equation}{section}

\newtheorem{proposition}[equation]{Proposition}
\crefname{proposition}{Proposition}{Propositions}

\newtheorem{lemma}[equation]{Lemma}
\crefname{lemma}{Lemma}{Lemmas}

\newtheorem{corollary}[equation]{Corollary}
\crefname{corollary}{Corollary}{Corollaries}

\newtheorem*{corollary*}{Corollary}
\crefname{corollary*}{Corollary}{Corollaries}

\newtheorem{theorem}[equation]{Theorem}
\crefname{theorem}{Theorem}{Theorems}

\newtheorem*{theorem*}{Theorem}
\crefname{theorem*}{Theorem}{Theorems}

\crefname{claim}{Claim}{Claims}

\theoremstyle{remark}

\crefname{question}{Question}{Questions}

\newtheorem{definition}[equation]{Definition}
\crefname{definition}{Definition}{Definitions}

\newtheorem{example}[equation]{Example}
\crefname{example}{Example}{Examples}

\newtheorem{remark}[equation]{Remark}
\crefname{remark}{Remark}{Remarks}

\newtheorem{assumption}[equation]{Assumption}
\crefname{assumption}{Assumption}{Assumptions}

\begin{document}
\maketitle

\begin{abstract}
We explain a construction of $G_2$-instantons on manifolds obtained by resolving $G_2$-orbifolds.
This includes the case of $G_2$-instantons on resolutions of $T^7/\Gamma$ as a special case.
The ingredients needed are a $G_2$-instanton on the orbifold and a Fueter section over the singular set of the orbifold which are used in a gluing construction.
In the general case, we make the very restrictive assumption that the Fueter section is pointwise rigid.
In the special case of resolutions of $T^7/\Gamma$, improved control over the torsion-free $G_2$-structure allows to remove this assumption.
As an application, we construct a large number of $G_2$-instantons on the simplest example of a resolution of $T^7/\Gamma$ with hundreds of distinct ones among them.
We also construct one new example of a $G_2$-instanton on the resolution of $(T^3 \times \text{K3})/\Z^2_2$.
\end{abstract}

\vspace{5cm}

\begin{figure}[bhpt]
\centering
\begin{minipage}{0.32 \textwidth}
\centering
\includegraphics[width=4.5cm]{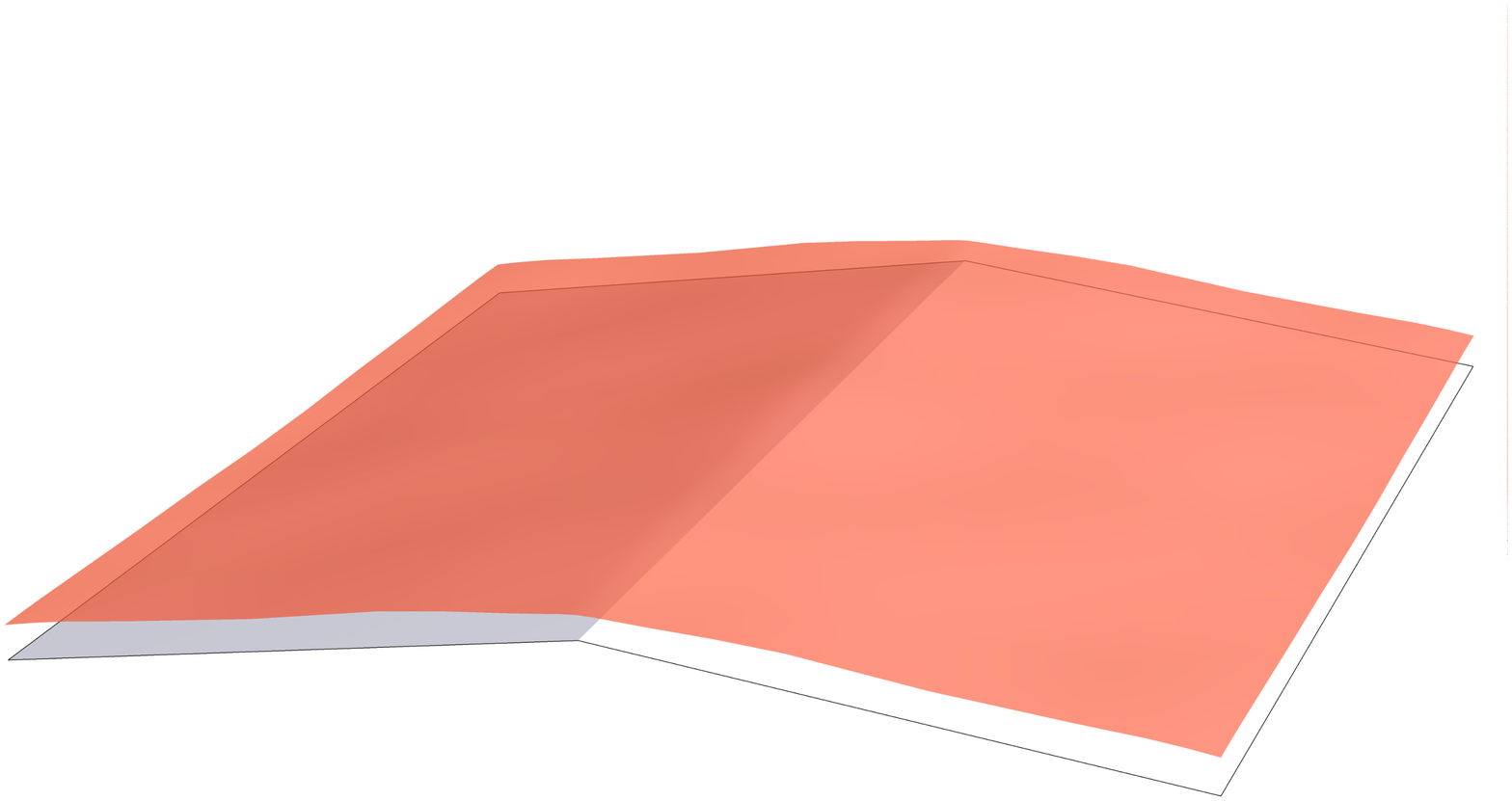}
\end{minipage}
\begin{minipage}{0.32 \textwidth}
\centering
\includegraphics[width=4.5cm]{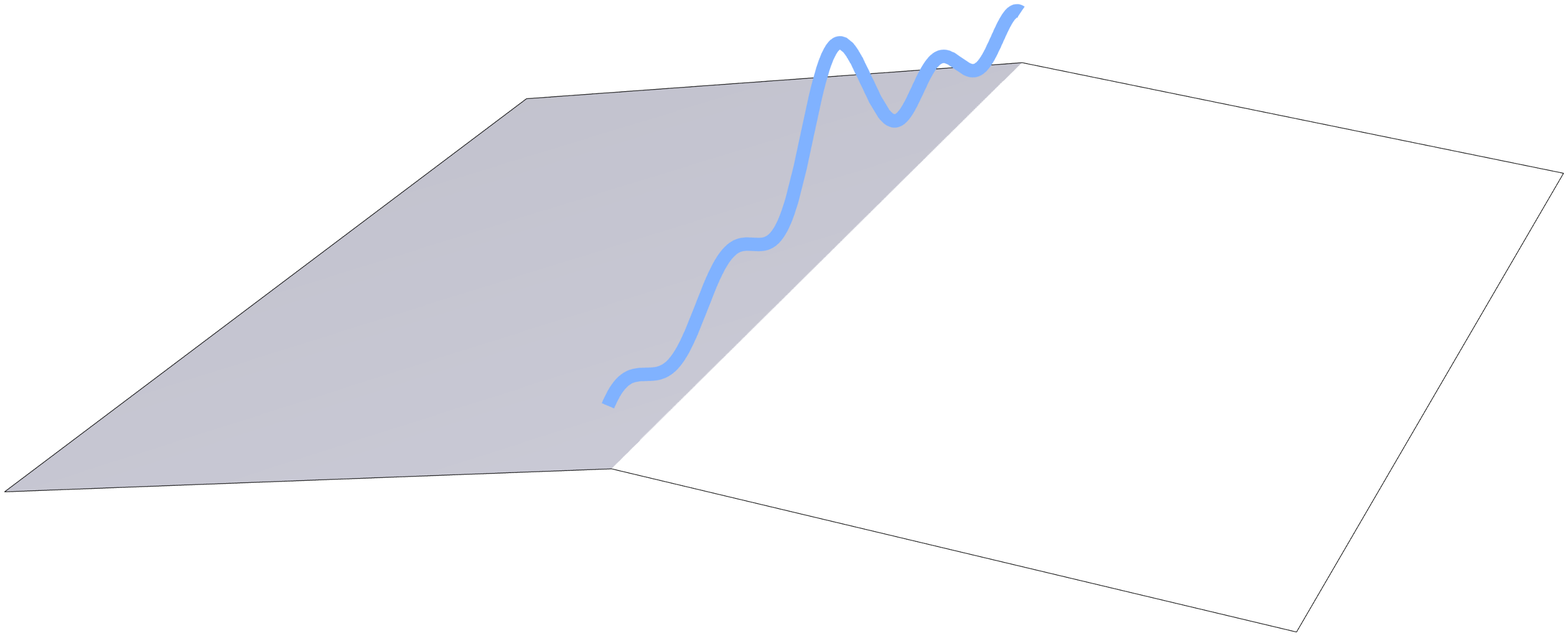}
\end{minipage}
\begin{minipage}{0.32 \textwidth}
\centering
\includegraphics[width=4.5cm]{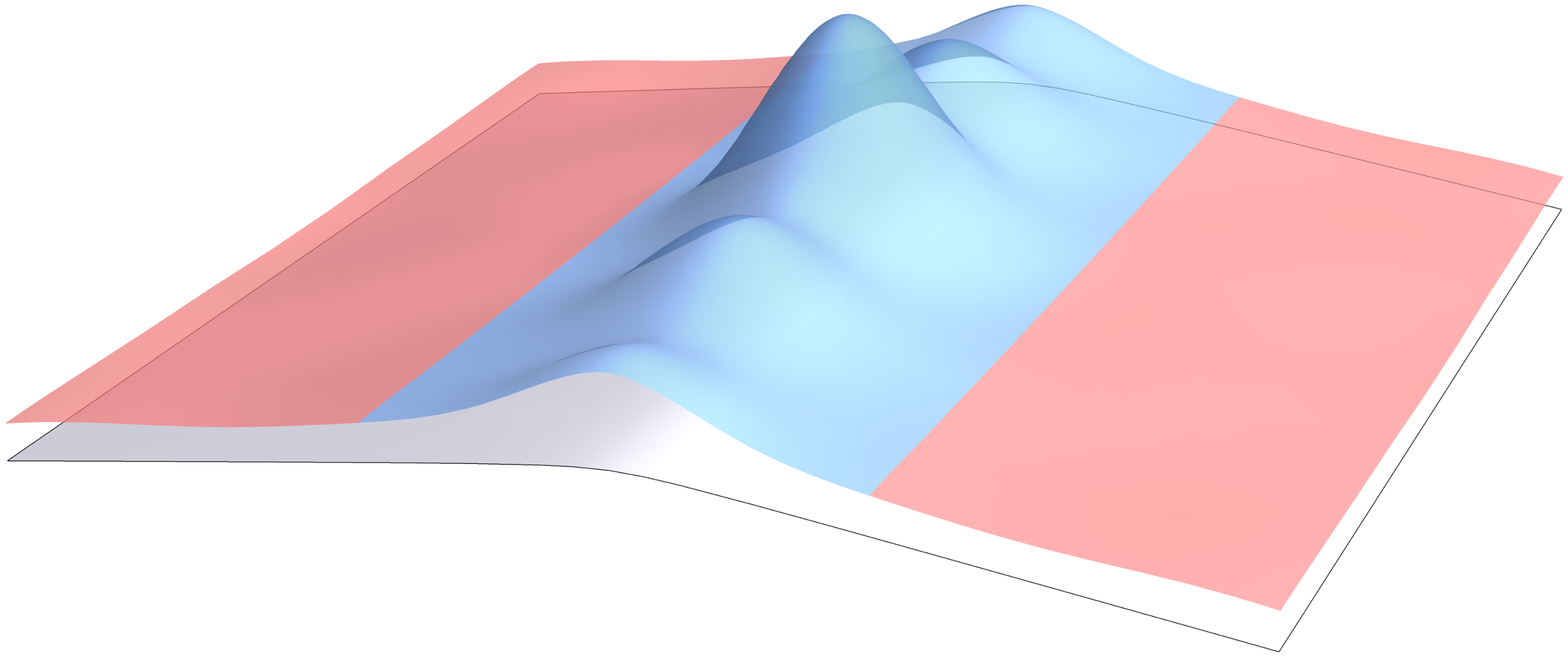}
\end{minipage}
\caption{Illustration of the gluing construction. Left: a $G_2$-instanton $\theta$ (in red) over a $G_2$-orbifold (in gray). Middle: a Fueter section $s$ (in blue) over the singular set of the orbifold. Right: the underlying smooth manifold is a resolution of the orbifold from the beginning. The gluing construction provides a connection over this smooth manifold which looks like $s$ close to the resolution locus, and looks like $\theta$ far away from the resolution locus.}
\label{figure:gluing-construction-title-graphic}
\end{figure}

\pagebreak

\setcounter{tocdepth}{2}
\tableofcontents

\setlength{\parskip}{0.3cm}

\section{Introduction}

In \cite{Berger1955}, Berger presented a list of groups which can possibly occur as the holonomy groups of Riemannian manifolds.
However, constructing manifolds which realise these holonomy groups remained a wide-open problem for decades.
A milestone in this direction was the formulation and proof of the Calabi conjecture in \cite{Calabi1954,Calabi1957} and \cite{Yau1977,Yau1978} respectively.
Among other things, the proof of this conjecture gives a powerful characterisation of manifolds admitting a metric with holonomy $\SU(n)$, giving rise to a wealth of examples of such manifolds.
Another entry on Berger's list is the exceptional holonomy group $G_2$.
The first compact examples of Riemannian manifolds with holonomy equal to $G_2$ were constructed in \cite{Joyce1996} by resolving an orbifold of the form $T^7/\Gamma$, where $\Gamma$ is a finite group of isometries of $T^7$.
In \cite{Joyce2017}, this construction was extended to resolutions of orbifolds of the form $Y/\Gamma$, where $Y$ is a manifold with holonomy contained in $G_2$, but not necessarily flat, and $\Gamma$ is a finite group of $G_2$-involutions.
However, an analogue of the Calabi conjecture for the holonomy group $G_2$ remains out of reach, and not much is known about which $7$-manifolds admit torsion-free $G_2$-structures, and if they do, how many.

In the seminal article \cite{Donaldson1983}, the moduli space of anti-self-dual connections was used to define invariants of smooth $4$-manifolds.
Following this, a rich theory of gauge theoretical invariants and their relations to other manifold invariants in $4$ dimensions was developed.
The article \cite{Donaldson1998} then recognised some of the $4$-dimensional phenomena in dimension $7$, for example the existence of a functional whose critical points are instantons.
With great optimism, one may hope to recreate the four-dimensional success story in dimension $7$, and use the moduli space of $G_2$-instantons to define invariants of $G_2$-manifolds that do not change when the $G_2$-structure is deformed.
This may shed some light on how many $G_2$-structures a $7$-manifold admits. 
For example, if two $G_2$-structures on the same manifold with different gauge theoretical invariants exist, one cannot be deformed into the other.

There are analytic difficulties present in dimension $7$ that were not there in dimension $4$, and therefore the study of $G_2$-instantons has mainly focused on the construction of examples.
The examples that have appeared in the literature so far are \cite{Walpuski2013a}, using a gluing construction on Joyce's Generalised Kummer construction, \cite{SaEarp2015,Menet2015,Walpuski2016} using a gluing construction on the Twisted Connected Sum construction, and \cite{Guenaydin1995,Lotay2018,Lotay2020} using cohomogeneity one methods.
We add to this by generalising the results from \cite{Walpuski2013a}:
we prove a gluing theorem that can be used to construct $G_2$-instantons on the $G_2$-manifolds from \cite{Joyce2017}.
Such a manifold is a resolution of a $G_2$-orbifold, obtained by taking the quotient of a $G_2$-manifold $Y$ by a finite group of $G_2$-involutions $\Gamma$.
The resolution $N$ is obtained by gluing Eguchi-Hanson spaces over the singular set of $Y/\Gamma$.
Given a $G_2$-instanton $\theta$ on $Y/\Gamma$ one may be able to construct from it a $G_2$-instanton on $N$.
To do this, one needs a connection over the glued in part.
One way to get such a connection is by taking a suitable family of anti-self-dual instantons over Eguchi-Hanson space, say $s$.
If $\theta$ and $s$ satisfy a simple topological compatibility condition, one can glue them together to a one-parameter family of connections $A_t$.
This is in general not a $G_2$-instanton, but our main result is that one can perturb $A_t$ and obtain a genuine $G_2$-instanton if $s$ consists of rigid instantons (cf. \cref{theorem:instanton-existence}):

\begin{theorem*}
 Assume that the section $s$ is given by a rigid ASD-instanton in every point $x \in L$, and assume that the connection $\theta$ used to define the approximate $G_2$-instanton $A_t$ is infinitesimally rigid.
 
 There exists $c>0$ such that for small $t$ there exists $\underline{a}_t=(a_t, \xi_t) \in C^{1,\alpha}(\Omega^0 \oplus \Omega^1(\Ad E_t))$ such that $\tilde{A}_t := A_t + a_t$ is a $G_2$-instanton.
 Furthermore, $\underline{a}_t$ satisfies $\|{ \underline{a}_t }_{C^{1,\alpha}_{-1,\delta;t}} \leq ct^{1/18}$.
\end{theorem*}

Here, $\alpha \in (0,1)$ must be a small number and $\|{ \; \cdot \; }_{C^{1,\alpha}_{-1,\delta;t}}$ denotes a weighted Hölder norm.
We use this theorem to construct two types of examples of $G_2$-instantons.
First, on the Generalised Kummer Construction (cf. \cref{corollary:g2-instantons-on-resolution-of-T7}):

 \begin{corollary*}
  Let $\Gamma$ act on $T^7$ as defined in \cref{equation:alpha-beta-gamma-t7} and let $N'_t$ denote the one parameter family of resolutions of $T^7/\Gamma$ from \cref{section:torsion-free-on-resolution-of-T7}.
  Then, for $t$ small enough, there are at least $246$ non-flat, irreducible $G_2$-instantons with structure group $\SO(3)$ over $N'$ with the property that no one of them is mapped onto another one of them under a smooth bijective bundle map covering an isometry $N' \rightarrow N'$ that preserves the $G_2$-structure on $N'$.
 \end{corollary*}
 
Second, on the resolution of $(T^3 \times \text{K3})/\Z_2^2$ (cf. \cref{corollary:non-trivial-g2-instanton}):

\begin{corollary*}
 Let $N_t$ denote the one parameter family of resolutions of $(T^3 \times X)/\Gamma$ from \cref{subsubsection:branched-double-cover}.
 Then, for $t$ small enough, there exists an irreducible $G_2$-instanton with structure group $\SO(3)$ over the resolution $N_t$.
\end{corollary*}

Thanks to the improved control over the torsion-free $G_2$-structure on resolutions of $T^7/\Gamma$ from \cite[Theorem 4.58]{Platt2020} we have an even stronger gluing theorem on these manifolds.
In this case, we need not require that the section $s$ is given by rigid instantons, only that it is a rigid solution of the Fueter equation (cf. \cref{theorem:instanton-existence3}):

\begin{theorem*}
 Let $N \rightarrow Y'$ be the resolution of the orbifold $Y'=T^7/\Gamma$ from before.
 Assume that the connection $\theta$ used to define the approximate $G_2$-instanton $A_t$ is infinitesimally rigid and that $s$ is an infinitesimally rigid Fueter section.
 
 There exists $c>0$ such that for small $t$ there exists an $\underline{a}_t=(a_t, \xi_t) \in C^{1,\alpha}(\Omega^0 \oplus \Omega^1(\Ad E_t))$ such that $\tilde{A}_t := A_t + a_t$ is a $G_2$-instanton.
 Furthermore, $\underline{a}_t$ satisfies $\|{ \underline{a}_t }_{\mathfrak{X}_t} \leq ct^{2-2\alpha}$.
\end{theorem*}

Here, $\|{ \; \cdot \; }_{\mathfrak{X}_t}$ denotes a complicated composite norm.
It consists of a part that is harmonic in the Eguchi-Hanson directions in the gluing region and a rest, and the two parts are scaled differently.

Unfortunately, no genuine examples of these more general ingredients are known.
That is:
all known rigid Fueter sections are actually sections of rigid instantons.

The article is structured as follows:
in \cref{section:background} we prepare some facts that are used later on.
Notably, in \cref{proposition:moduli-space-bijection} we give a proof of the folklore result that the moduli spaces of framed instantons over an ALE space and over its compactification are in bijection.
Following this, in \cref{section:review-of-the-manifold-construction}, we review the two construction methods for torsion-free $G_2$-structures from \cite{Joyce1996} and \cite{Joyce2017}.
In \cref{section:the-gluing-construction-for-instantons} we prove our gluing theorem for $G_2$-instantons:
we first construct the approximate solution $A_t$, then construct the perturbation to a genuine $G_2$-instanton.
Last, in \cref{section:examples}, we apply this construction method to the construction new $G_2$-instantons.

\textbf{Acknowledgment.} The author is indebted to his PhD supervisor Jason Lotay and co-supervisor Simon Donaldson for their encouragement and advice during the work on this article.
The author thanks Simon Salamon for helpful discussions and thanks the London School of Geometry and Number Theory for creating a stimulating learning and research environment.
This work was supported by the Engineering and Physical Sciences Research Council [EP/L015234/1], the EPSRC Centre for Doctoral Training in Geometry and Number Theory (The London School of Geometry and Number Theory), University College London. 
The author was also supported by Imperial College London and King's College London.

\section{Background}
\label{section:background}

\subsection{$G_2$-structures}
\label{section:g2-structures}

We now introduce $G_2$-structures and their torsion, following the treatment in \cite{Joyce2000}.

\begin{definition}[Definition 10.1.1 in \cite{Joyce2000}]
\label{definition:g2}
 Let $(x_1,\dots,x_7)$ be coordinates on $\R^7$.
 Write $\d x_{ij \dots l}$ for the exterior form $\d x_i \wedge \d x_j \wedge \dots \wedge \d x_l$.
 Define $\varphi_0 \in \Omega^3(\R^7)$ by
 \begin{align}
 \label{equation:standard-g2-form}
  \varphi_0
  =
  \d x_{123}+\d x_{145}+\d x_{167}+\d x_{246}-\d x_{257}-\d x_{347}-\d x_{356}.
 \end{align}
 The subgroup of $\GL(7,\R)$ preserving $\varphi_0$ is the exceptional Lie group $G_2$.
 It also fixes the Euclidean metric $g_0=\d x_1^2+\dots + \d x_7^2$, the orientation on $\R^7$, and $* \varphi_0 \in \Omega^4(\R^7)$.
\end{definition}

\begin{definition}
\label{definition:cross-product}
 The skew-symmetric bilinear map $\times : \R^7 \rightarrow \R^7$ defined by
 \[
  \varphi_0(u,v,w)
  =
  g_0(u \times v, w)
 \]
 for $u,v,w \in \R^7$ is called the \emph{cross product induced by $\varphi$}.
\end{definition}

\begin{theorem}[Theorem 8.5 in \cite{Salamon2017}]
\label{theorem:g2-decomposition-of-forms}
 Let $\psi = * \varphi_0$.
 Then $\Lambda^* (\R^7)^*$ splits into irreducible representations of $G_2$ as follows:
 \begin{align*}
  \Lambda^1 V^* &= \Lambda^1_7,
  \\
  \Lambda^2 V^* &= \Lambda^2_7 \oplus \Lambda^2_{14},
  \\
  \Lambda^3 V^* &= \Lambda^3_1 \oplus \Lambda^3_7 \oplus \Lambda^3_{27}
 \end{align*}
 and correspondingly for $\Lambda^k (\R^7)^* \simeq \Lambda^{7-k}(\R^7)^*$ with $k=4,5,6$.
 Here, $\dim \Lambda^k_d=d$ and
 \begin{align*}
  \Lambda^2_7
  &:=
  \{
   \alpha: *(\alpha \wedge \varphi_0)=2\alpha
  \}
  =
  \{
   i(u) \varphi_0: u \in \R^7
  \} \simeq \Lambda^1_7,
  \\
  \Lambda^2_{14}
  &:=
  \{
   \alpha: *(\alpha \wedge \varphi_0)=-\alpha
  \}
  =
  \{
   \alpha: \alpha \wedge \psi=0
  \}
  \simeq \mathfrak{g}_2,
  \\
  \Lambda^3_1
  &:=
  \< \varphi_0 \>,
  \\
  \Lambda^3_7
  &:=
  \{ i(u) \psi : u \in \R^7 \}
  \simeq \Lambda^1_7,
  \text{ and}
  \\
  \Lambda^3_{27}
  &:=
  \{
   \alpha : \alpha \wedge \varphi_0=0
   \text{ and }
   \alpha \wedge \psi =0
  \}
  \simeq \Sym_0(\R^7)
 \end{align*}
\end{theorem}

\begin{definition}
 Let $M$ be an oriented $7$-manifold.
 A principal subbundle $Q$ of the bundle of oriented frames with structure group $G_2$ is called a \emph{$G_2$-structure}.
 Viewing $Q$ as a set of linear maps from tangent spaces of $M$ to $\R^7$, there exists a unique $\varphi \in \Omega^3(M)$ such that $Q$ identifies $\varphi$ with $\varphi_0 \in \Omega^3(\R^7)$ at every point.
 
 Such $G_2$-structures are in $1$-$1$ correspondence with $3$-forms on $M$ for which there exists an oriented isomorphism mapping them to $\varphi_0$ at every point.
 We will therefore also refer to such $3$-forms as $G_2$-structures.
\end{definition}

Let $M$ be a manifold with $G_2$-structure $\varphi$.
We call $\nabla \varphi$ the \emph{torsion} of a $G_2$-structure $\varphi \in \Omega^3(M)$.
Here, $\nabla$ denotes the Levi-Civita induced by $\varphi$ in the following sense:
we have $G_2 \subset \SO(7)$, so $\varphi$ defines a Riemannian metric $g$ on $M$, which in turn defines a Levi-Civita connection.
As a shorthand, we also use the following notation:
write $\Theta(\varphi)=*\varphi$, where ``$*$'' denotes the Hodge star defined by $g$.
Using this, the following theorem gives a characterisation of torsion-free $G_2$-manifolds:

\begin{theorem}[Propositions 10.1.3 and 10.1.5 in \cite{Joyce2000}]
\label{theorem:torsion-free-g2-structures-characterisation}
 Let $M$ be an oriented $7$-manifold with $G_2$-structure $\varphi$ with induced metric $g$.
 The following are equivalent:
 \begin{enumerate}[label=(\roman*)]
  \item 
  $\Hol(g) \subseteq G_2$,
  
  \item
  $\nabla \varphi=0$ on $M$, where $\nabla$ is the Levi-Civita connection of $g$, and
  
  \item
  $\d \varphi =0$ and $\d \Theta(\varphi)=0$ on $M$.
 \end{enumerate}
 If these hold then $g$ is Ricci-flat.
\end{theorem}

The goal of \cref{section:review-of-the-manifold-construction} will be to construct $G_2$-structures that induce metrics with holonomy \emph{equal} to $G_2$.
A torsion-free $G_2$-structure alone only guarantees holonomy \emph{contained} in $G_2$, but in the compact setting a characterisation of manifolds with holonomy equal to $G_2$ is available:

\begin{theorem}[Proposition 10.2.2 and Theorem 10.4.4 in \cite{Joyce2000}]
 Let $M$ be a compact oriented manifold with torsion-free $G_2$-structure $\varphi$ and induced metric $g$.
 Then $\Hol(g)=G_2$ if and only if $\pi_1(M)$ is finite.
 In this case the moduli space of metrics with holonomy $G_2$ on $M$, up to diffeomorphisms isotopic to the identity, is a smooth manifold of dimension $b^3(M)$.
\end{theorem}

Note that this theorem makes no statement about the existence of a torsion-free $G_2$-structure in the first place.
Finding a characterisation of manifolds which admit a torsion-free $G_2$-structure and even the construction of examples remain challenging problems in the field.

Later on, we will investigate perturbations of $G_2$-structures and analyse how that changes their torsion.
To this end, we will use the following estimates for the map $\Theta$ defined before:

\begin{proposition}[Proposition 10.3.5 in \cite{Joyce2000} and eqn. (21) of part I in \cite{Joyce1996}]
\label{proposition:Theta-estimates}
 There exists $\epsilon > 0$ and $c>0$ such that whenever $M$ is a $7$-manifold with $G_2$-structure $\varphi$ satisfying $\d \varphi=0$, then the following is true.
 Suppose $\chi \in C^\infty(\Lambda^3 T^*M)$ and $| \chi | \leq \epsilon$.
 Then $\varphi+\chi$ is a $G_2$-structure, and
 \begin{align}
  \label{equation:theta-expansion}
  \Theta(\varphi+\chi)
  =*\varphi-T(\chi)-F(\chi),
 \end{align}
 where ``$*$'' denotes the Hodge star with respect to the metric induced by $\varphi$, $T: \Omega^3(M) \rightarrow \Omega^4(M)$ is a linear map (depending on $\varphi$), and $F$ is a smooth function from the closed ball of radius $\epsilon$ in $\Lambda^3T^*M$ to $\Lambda^4 T^*M$ with $F(0)=0$.
 Furthermore,
  \begin{align*}
   \left|
    F(\chi)
   \right|
   &\leq
   c
   \left| \chi \right|^2,
   \\
   \left|
    \d \,(F(\chi))
   \right|
   &\leq
   c
   \left\{
    \left| \chi \right|^2 \left| \d^* \varphi \right|
    +
    \left| \nabla \chi \right| \left| \chi \right|
   \right\},
   \\
   [\d \, (F(\chi))]_\alpha
   &\leq
   c
   \left\{
    [\chi]_\alpha \|{ \chi }_{L^\infty} \|{ \d^* \varphi }_{L^\infty}
    +
    \|{ \chi }_{L^\infty}^2 [ \d ^* \varphi ]_\alpha
    +
    [\nabla \chi ]_\alpha \|{ \chi }_{L^\infty}
    +
    \|{ \nabla \chi }_{L^\infty} [\chi]_\alpha
   \right\},
  \end{align*}
  as well as
  \begin{align*}
   \left|
    \nabla (F(\chi))
   \right|
   &\leq
   c
   \left\{
    \left| \chi \right|^2 \left| \nabla \varphi \right|
    +
    \left| \nabla \chi \right| \left| \chi \right|
   \right\},
   \\
   [\nabla (F(\chi))]_{C^{0,\alpha}}
   &\leq
   c
   \left\{
    [\chi]_\alpha \|{ \chi }_{L^\infty} \|{ \nabla \varphi }_{L^\infty}
    +
    \|{ \chi }_{L^\infty}^2 [ \nabla \varphi ]_\alpha
    +
    [\nabla \chi ]_\alpha \|{ \chi }_{L^\infty}
    +
    \|{ \nabla \chi }_{L^\infty} [\chi]_\alpha
   \right\}.
  \end{align*}
  Here, $\left| \cdot \right|$ denotes the norm induced by $\varphi$, $\nabla$ denotes the Levi-Civita connection of the metric induced by $\varphi$, and $[ \cdot ]_{C^{0,\alpha}}$ denotes the unweighted Hölder semi-norm induced by this metric.
\end{proposition}

Finally, the landmark result on the existence of torsion-free $G_2$-structures is the following theorem.
It first appeared in \cite[part I, Theorem A]{Joyce1996}, and we present a rewritten version in analogy with \cite[Theorem 2.7]{Joyce2017}:

\begin{theorem}
\label{theorem:original-torsion-free-existence-theorem}
 Let $\alpha, K_1, K_2, K_3$ be any positive constants.
 Then there exist $\epsilon \in (0,1]$ and $K_4 >0$, such that whenever $0<t \leq \epsilon$, the following holds.
 
 Let $M$ be a compact oriented $7$-manifold, with $G_2$-structure $\varphi$ with induced metric $g$ satisfying $\d \varphi=0$.
 Suppose there is a closed $3$-form $\psi$ on $M$ such that $\d^* \varphi=\d^* \psi$ and
 \begin{enumerate}[label=(\roman*)]
  \item 
  $\|{ \psi }_{C^0} \leq K_1 t^\alpha$,
  $\|{ \psi }_{L^2} \leq K_1 t^{7/2+\alpha}$,
  and
  $\|{ \psi }_{L^{14}} \leq K_1 t^{-1/2+\alpha}$.
  
  \item
  The injectivity radius $\inj$ of $g$ satisfies $\inj \geq K_2 t$.
  
  \item
  The Riemann curvature tensor $\Rm$ of $g$ satisfies $\|{ \Rm }_{C^0} \leq K_3 t^{-2}$.
 \end{enumerate}
 Then there exists a smooth, torsion-free $G_2$-structure $\tilde{\varphi}$ on $M$ such that $\|{ \tilde{\varphi}-\varphi}_{C^0} \leq K_4 t^\alpha$ and $[\tilde{\varphi}]=[\varphi]$ in $H^3(M,\R)$.
 Here all norms are computed using the original metric $g$.
\end{theorem}

On $\mathbb{H}$ with coordinates $(y_0,y_1,y_2,y_3)$ we have the three symplectic forms $\omega_1,\omega_2,\omega_3$ given as
\begin{align*}
 \omega_0&=
 \d y_0 \wedge \d y_1 + \d y_2 \wedge \d y_3,
 &
 \omega_1&=
 \d y_0 \wedge \d y_2 - \d y_1 \wedge \d y_3,
 &
 \omega_2&=
 \d y_0 \wedge \d y_3 + \d y_1 \wedge \d y_2.
\end{align*}
Identify $\R^7$ with coordinates $(x_1,\dots,x_7)$ with $\R^3 \oplus \mathbb{H}$ with coordinates $((x_1,x_2,x_3),(y_1,y_2,y_3,y_4))$.
Then we have for $\varphi_0, * \varphi_0$ from \cref{definition:g2}:
\begin{align}
\label{equation:product-g2-structure}
 \varphi_0
 &=
 \d x_{123} - 
 \sum_{i=1}^3
 \d x_i \wedge \omega_i
 ,&
 *\varphi_0
 &=
 \vol_{\mathbb{H}}-
 \sum_
 {\substack{(i,j,k) = (1,2,3)\\
 \text{and cyclic permutation}}}
 \omega_i \wedge \d x_{jk}.
\end{align}

This linear algebra statement easily extends to product manifolds in the following sense:
if $X$ is a Hyperkähler $4$-manifold, and $\R^3$ is endowed with the Euclidean metric, then $\R^3 \times X$ has a $G_2$-structure.
The $G_2$-structure is given by the same formula as in the flat case, namely \cref{equation:product-g2-structure}, after replacing $(\omega_1,\omega_2,\omega_3)$ with the triple of parallel symplectic forms defining the Hyperkähler structure on $X$.
This \emph{product $G_2$-structure} will be glued into $G_2$-orbifolds \cref{section:review-of-the-manifold-construction}.

\subsection{Gauge Theory in Dimension $4$}
\label{subsection:gauge-theory-dimension-4}
\label{subsubsection:gauge-theory-on-ale}

In this part we briefly review the theory of ASD instantons on ALE spaces.
We follow the treatment of \cite{Nakajima1990}.
A treatment of compact $4$-manifolds can be found in \cite{Donaldson1990}.

Let $\Gamma \subset \SU(2)$ be a finite subgroup and let $X$ be an ALE $4$-manifold asymptotic to $\C^2 /\Gamma$.
Even though $X$ is non-compact, some of the results from gauge theory on compact manifolds carry over to this setting.
First, we explain a correspondence between gauge equivalence classes of connections on $X$ and on its one point compactification $\hat{X}=X \cup \{ \infty \}$.
For $\hat{X}$, we have:

\begin{proposition}[p.687 in \cite{Kronheimer1989} and Proposition 2.36 in \cite{Walpuski2013}]
\label{proposition:one-point-compactification-of-eguchi-hanson}
\label{equation:orbifold-chart}
Let $(X,g)$ be an ALE manifold asymptotic to $\C^2 / \Gamma$ and let $\hat{X}=X \cup \{ \infty \}$ be the one point compactification of $X$.
\begin{enumerate}
 \item
 The topological space $\hat{X}$ is an orbifold and there exist a neighbourhood $V$ of $\infty$ and an orbifold chart $f: B^4/ \Gamma \rightarrow V$, where $B^4$ is the unit ball in $\R^4$.
 
 \item
 The orbifold $\hat{X}$ carries an orbifold metric $\hat{g}$ of regularity $C^{3,\alpha}$ for any $\alpha \in (0,1)$ such that the restriction of $\hat{g}$ to $X \subset \hat{X}$ is conformally equivalent to $g$.
\end{enumerate}
\end{proposition}

Let $G$ be a compact connected Lie group with a faithful representation $G \rightarrow \GL(V)$.
Let $\hat{P}$ be an orbifold $G$-bundle over $\hat{X}$ and denote its restriction to $X$ by $P$, i.e. $P=\hat{P}|_{X}$.
That is, $\hat{P}$ restricted to $V \simeq B^4/\Gamma$ from \cref{proposition:one-point-compactification-of-eguchi-hanson} is the trivial bundle $B^4 \times G$ together with a fixed lift of the action of $\Gamma$ on $B^4$ to $B^4 \times G$.
Over the point $0 \in B^4$, this defines a homomorphism $\rho : \Gamma \rightarrow G$.
The following proposition states that this homomorphism essentially characterises the orbifold bundle over $B^4$ completely.

\begin{proposition}
\label{assumption:trivialisation-at-infinity}
 There exists a trivialisation $\kappa: \hat{P}|_{B^4} \rightarrow B^4 \times G$ such that $\Gamma$ acts through left multiplication by $\rho$:
\begin{align}
 \label{equation:asymptotic-at-infinity-condition2}
 \gamma \cdot \kappa^{-1}(b,g)
 =
 \kappa^{-1}(\gamma \cdot b,\rho(\gamma)g)
 \text{ for }
 \gamma \in \Gamma,
 (b,g) \in B^4 \times G.
\end{align}
\end{proposition}

\begin{proof}
 The lift of the action of $\Gamma$ to $B^4 \times G$ can be viewed as an element $w \in C^\infty(B^4, \Hom(\Gamma,G))$ via $\gamma \cdot (b,g)=(\gamma \cdot b, w(b)(\gamma)\cdot g)$.
 The space $B^4$ is connected, so by \cref{corollary:rep-variety-components-come-from-conjugation} the conjugacy class of $w$ does not change over $B^4$.
 That is, there exists $\sigma \in C^\infty(B^4, G)$ such that $l_\sigma r_{\sigma^{-1}} w \in C^\infty(B^4, \Hom(\Gamma,G))$ is constant and $l_\sigma r_{\sigma^{-1}} w(0)=\rho$.
 Thus $\sigma$ defines a trivialisation of $B^4 \times G$ in which $\Gamma$ acts through left multiplication via $\rho$.
\end{proof}

Because of \cref{assumption:trivialisation-at-infinity} we can fix a trivialisation of $\hat{P}$ over $B^4$ such that $\Gamma$ acts through left multiplication by $\rho$.
Then denote by $A_0$ any extension of the product connection with respect to this trivialisation to all of $\hat{P}$.
Different choices of extension will give rise to the very same spaces in \cref{equation:ALE-moduli-definitions}.
We identify $[R, \infty) \times S^3 / \Gamma \simeq X \setminus K$ for some $R>0$ big enough and a compact set $K \subset X$.
Then the monodromy representation of $A_0$ restricted to $\{ t \} \times S^3 / \Gamma$, say $h: \pi_1(\{ t \} \times S^3 / \Gamma) \rightarrow G$, satisfies
\begin{align}
\label{equation:asymptotic-at-infinity-condition}
 h = \rho
\end{align}
under the canonical identification $\Gamma \simeq \pi_1(\{ t \} \times S^3 / \Gamma)$.
Extend the projection onto the first component $X \setminus K \simeq [R, \infty) \times S^3 \rightarrow [R, \infty)$ to a smooth positive function $r$ on all of $X$.
For a non-negative integer $l$, a weight $\delta \in \R$, and $p \geq 1$ define the weighted Sobolev norm on the $k$-forms with values in the adjoint bundle with compact support $\Omega^k_0(\Ad P)$ via
\begin{align}
\label{equation:weighted-sobolev-norm}
 \|{
  \alpha
 }_{L^p_{l,\delta}}
 =
 \sum_{j=0}^l
 \left(
  \int_X
  | \nabla_{A_0}^j \alpha |^p
  r^{-(\delta-j)p-4}
  \d V
 \right)^{1/p},
\end{align}
and denote by $L^p_{l,\delta}(\Lambda^k(\Ad P))$ the completion of $\Omega^k_0(\Ad P)$ with respect to the norm 
$\|{
  \alpha
 }_{L^p_{l,\delta}}$.
 
As before, set $E=P \times_G V$ and for $l \geq 3$ define
\begin{align}
\label{equation:ALE-moduli-definitions}
\begin{split}
 \mathscr{A}^{l, \delta}
 &=
 \{
  A_0 + \alpha : \alpha \in L^2_{l,\delta}(\Lambda^1(\Ad P))
 \},
 \\
 \mathscr{G}_0^{l+1, \delta+1}
 &=
 \{
  s \in L^2_{l+1, \text{loc}}(\Lambda^0(\End(E)) :
  s(x) \in G \text{ for all } x \in G,
  \|{ s - \Id }_{L^2_{l+1,\delta+1}} < \infty
 \},
 \\
 G_\rho
 &=
 \{
  s \in G: s \rho s^{-1} = \rho
 \},
 \\
 \mathscr{G}^{l+1, \delta+1}
 &=
 \{
  s \in L^2_{l+1, \text{loc}}(\Lambda^0(\End(E)) :
  s(x) \in G \text{ for all } x \in G,
  \\
  &\;\;\;\;\;\;\;\;\;\;\;\;\;\;\;\;\;\;\;\;\;\;\;\;\;\;\;\;\;\;\;\;\;\;\;\;\;\;\;\;\;\;\;\;\;\;\;
  \|{ s - s_\infty }_{L^2_{l+1,\delta+1}} < \infty
  \text{ for some }
  s_\infty \in G_\rho
 \}.
\end{split}
\end{align}
 In the definition of $\mathscr{G}^{l+1, \delta+1}$ we regarded $s_\infty \in G_\rho$ as an element in $C^\infty(\Lambda^0(\End(E))$ as follows:
 consider $\hat{P}$ over $B^4$ defined by the orbifold chart around $\infty$.
Using the trivialisation from \cref{assumption:trivialisation-at-infinity}, this canonically defines a gauge transformation over $B^4$.
(It is the same to say that we obtain a gauge transformation by parallel transport with respect to $A_0$.)
 This gauge transformation is $\Gamma$-equivariant by definition of $G_\rho$ and \cref{assumption:trivialisation-at-infinity}.
 We then extend it arbitrarily on the rest of $\hat{X}$ to an element in $C^\infty(\Lambda^0(\End(E))$.
 The choice of the extension does not matter for the condition $\|{ s - s_\infty }_{L^2_{l+1,\delta+1}} < \infty$.

The gauge groups $\mathscr{G}_0^{l+1, \delta+1}$ and $\mathscr{G}^{l+1, \delta+1}$ both act on $\mathscr{A}^{l, \delta}$, and the quotient spaces $\mathscr{A}^{l, \delta}/\mathscr{G}_0^{l+1, \delta+1}$ and $\mathscr{A}^{l, \delta}/\mathscr{G}^{l+1, \delta+1}$ are called the moduli space of framed connections and the moduli space of unframed connections, respectively.
We can restrict to anti-self-dual connections:
\begin{align*}
 \mathscr{A}^{l, \delta}_{\text{asd}}
 =
 \{
  A \in \mathscr{A}^{l, \delta} : A \text{ is anti-self-dual}
 \}
\end{align*}
and obtain the \emph{moduli space of framed ASD connections $M^{l, \delta}:=\mathscr{A}^{l, \delta}_{\text{asd}}/\mathscr{G}_0^{l+1, \delta+1}$} and the \emph{moduli space of ASD connections $\mathscr{A}^{l, \delta}_{\text{asd}}/\mathscr{G}^{l+1, \delta+1}$}.

The four quotient spaces $\mathscr{A}^{l, \delta}/\mathscr{G}_0^{l+1, \delta+1}$, $\mathscr{A}^{l, \delta}/\mathscr{G}^{l+1, \delta+1}$, $M^{l, \delta}$, and $\mathscr{A}^{l, \delta}_{\text{asd}}/\mathscr{G}^{l+1, \delta+1}$ are topological spaces.
For $M^{l, \delta}$ we will observe explicitly (cf. \cref{theorem:tangent-space-of-moduli}) that it is metrisable and therefore Hausdorff, and the same argument works for the other three quotient spaces, cf. \cite[Lemma 4.2.4]{Donaldson1990}.

Moving on to the orbifold, we define:
\begin{definition}
 For $l \geq 3$ let
 \begin{align*}
  \mathscr{A}^{l, \text{orb}}_{\text{asd}}
  &=
  \{
   A_0 + \alpha : \alpha \in L^2_{l}(\Lambda^1(\Ad \hat{P}))
  \},
  \\
  \mathscr{G}^{l+1, \text{orb}}
  &=
  \{
   s \in L^2_{l+1}(\Lambda^0(\End V)):
   s(x) \in G \text{ for all } x \in \hat{X},
   s(\infty) \in G_\rho
  \},
  \\
  \mathscr{G}^{l+1, \text{orb}}_0
  &=
  \{
   s \in \mathscr{G}^{l+1, \text{orb}}:
   s(\infty) = \Id
  \}.
 \end{align*}
 Then $\mathscr{G}^{l+1, \text{orb}}$ and $\mathscr{G}^{l+1, \text{orb}}_0$ both act on $\mathscr{A}^{l, \text{orb}}_{\text{asd}}$ and we can form the quotient spaces $\mathscr{A}^{l, \text{orb}}_{\text{asd}}/\mathscr{G}^{l+1, \text{orb}}$ and $M^{l, \text{orb}}=\mathscr{A}^{l, \text{orb}}_{\text{asd}}/\mathscr{G}^{l+1, \text{orb}}_0$.
 Here, $M^{l, \text{orb}}$ is called the \emph{moduli space of framed ASD connections} on $\hat{X}$.
\end{definition}

We have that these definitions are essentially independent of the chosen regularity $l$:

\begin{proposition}
\label{proposition:moduli-independent-of-regularity}
 For $3 \leq l_1 < l_2$, the inclusion maps
 \begin{align*}
  M^{l_1, \text{orb}} \hookrightarrow M^{l_2, \text{orb}},
  &&
  M^{l_1, -2} \hookrightarrow M^{l_2, -2}
 \end{align*}
 are homeomorphisms.
\end{proposition}

The proof of \cref{proposition:moduli-independent-of-regularity} works the same as in the compact case, i.e. the proof of \cite[Proposition 4.2.16]{Donaldson1990}.
The only difference is that in the non-compact case, i.e. for the claim $M^{l_1, -2} \hookrightarrow M^{l_2, -2}$, one has to take the weighted Sobolev norms from \cref{equation:weighted-sobolev-norm}.
These have their own versions of the Sobolev embedding theorem and, if the weight is non-positive, the multiplication theorem for Sobolev norms also holds.
These properties of weighted Sobolev norms are proved in \cite[Corollary 6.8]{Pacini2013}.

\begin{proposition}
\label{proposition:bundle-extension-is-isomorphic}
 For any $A \in \mathscr{A}^{l,-2}_{\asd}$ there exists a connection $\hat{A} \in \mathscr{A}(\hat{P})$ satisfying $\hat{A}|_P=A$.
\end{proposition}

\begin{proof}
 \Cref{corollary:removable-sing-over-ale-limit-at-infinity} gives a bundle $P'$ over $\hat{X}$ with connection $A'$ together with an injective bundle homomorphism $\xi:P \rightarrow P'$.
 After fixing a trivialisation of $\hat{P}$ around $\infty$, this canonically defines an isomorphism of orbifold $G$-bundles $h:\hat{P} \rightarrow P'$, and $\hat{A}:=h^*(A')$ satisfies $\hat{A}|_P=A$.
\end{proof}

\begin{definition}
\label{definition:psi-function-moduli-space-iso}
 Define the map
 \[
  \Psi:
  M^{3,-2} \rightarrow M^{3,\orb}
 \]
 as follows:
 for $[A_0+a] \in M^{3,-2}$ let $\hat{A} \in \mathscr{A}(\hat{P})$ be the induced connection from \cref{proposition:bundle-extension-is-isomorphic} and set $\Psi([A_0+a]) := [\hat{A}]$.
\end{definition}

\begin{proposition}
\label{proposition:moduli-space-bijection}
 The function $\Psi$ from \cref{definition:psi-function-moduli-space-iso} is bijective.
\end{proposition}

\begin{proof}
\label{proposition:orbifold-moduli-space-homeo} 
 \textbf{$\Psi$ is injective:}
 let $[A_0 + a],[A_0 + \tilde{a}] \in M^{3,-2}$ such that $\Psi([A_0 + a])=[\hat{A}]$ as well as $\Psi([A_0 + \tilde{a}])=[\hat{A}']$.
 If $[\hat{A}]=[\hat{A}']$, then $\hat{A}'=s\hat{A}$ for some $s \in \mathscr{G}^{4,\orb}_0$.
 We have $s(\infty)=\Id$, so $(s-\Id)=\mathcal{O}(|x|)$ and $\nabla_{A_0}^k(s-\Id)=\mathcal{O}(1)$ for $k \in \{ 1,2,3,4\}$.
 Here, $\nabla_{A_0}^k$ includes terms containing the Levi-Civita connection for the orbifold metric $\hat{g}$ on $\hat{X}$ for $k>1$, and $|x|$ denotes the distance from $\infty \in \hat{X}$ in this metric.
 In particular, $\nabla_{A_0}^k(s-\Id)=\mathcal{O}(|x|^{1-k})$.
 We have
 \[
  \left| \nabla^k_{A_0} (s-\Id) \right|_g
  =
  (1+r^2)^{-k}
  \left| \nabla^k_{A_0} (s-\Id) \right|_{\hat{g}}
  =
  \mathcal{O}(r^{-2k} \left| x \right|^{1-k})
  =
  \mathcal{O}(r^{-1-k}),
 \]
 where $g$ denotes the ALE metric, in the first step we used the definition of $\hat{g}$ from the proof of \cref{proposition:one-point-compactification-of-eguchi-hanson} and the fact that we are measuring a tensor with $k$ covariant indices and $0$ contravariant indices.
 Thus, $s \in \mathscr{G}_0^{4,-1}$. 
 Therefore, $[A_0 + a]=[A_0+\tilde{a}]$ as elements in $M^{3,-2}$, which shows the claim.
 
 \textbf{$\Psi$ is surjective:}
 Let $[A_0+a] \in M^{3,\orb}$, i.e. $A_0+a \in \mathscr{A}^{3,\orb}_{\asd}$.
 Similar to the previous point we find that $\nabla_{A_0}^k a = \mathcal{O}(r^{-2-k})$.
 By construction $\Psi([(A_0+a)|_X])=[A_0+a]$, which proves the claim.
\end{proof}

Because of \cref{proposition:moduli-independent-of-regularity} we will drop the regularity and decay from the notation of our moduli spaces most of the time.
That is, we will often write $M$ for $M^{l, \delta}$ with any $l \geq 3$ and $\delta = -2$.
Likewise for $\mathscr{A}, \mathscr{G}, \mathscr{G}_0, \mathscr{A}^{\orb}, M^{\text{orb}}, \mathscr{G}^{\orb}$, and $\mathscr{G}_0^{\orb}$. 

The important results about the local structure of $M$ are the following:

\begin{theorem}[Theorem 2.4 and Proposition 5.1 in \cite{Nakajima1990}]
\label{theorem:tangent-space-of-moduli}
 $M$ is a nonsingular smooth manifold and for $[A] \in M$ its tangent space is isomorphic to
 \[
  H^1_{A,-2}
  :=
  \{
   \alpha \in L^2_{l,-2}(\Lambda^1(\Ad P)) : \delta_A(\alpha)=0
  \},
 \]
 where
 \begin{align}
  \label{equation:asd-instanton-linearisation}
  \begin{split}
  \delta_A: \Omega^1(\Ad P) & \rightarrow
  (\Omega^0 \oplus \Omega^2_+)(\Ad P)
  \\
  \alpha & \mapsto 
  (
  \d^*_A \alpha,
  \d_A^+ \alpha
  ).
  \end{split}
 \end{align}
\end{theorem}

For the linear operator $\delta_A$ we have the following analytic result:

\begin{proposition}[Proposition 5.10 in \cite{Walpuski2013a}]
\label{proposition:ale-asd-kernel-cokernel}
 Let $A \in \mathscr{A}(E)$ be a finite energy ASD instanton on $E$.
 Then the following holds:
 \begin{enumerate}
  \item
  If $a \in \Ker \delta_A$ decays to zero at infinity, i.e., $\lim_{ r \rightarrow \infty} \sup_{\rho(x)=r} |a|(x) =0$, then $\nabla_A^k a= \mathcal{O}(|\pi|^{-3-k})$ for all $k \geq 0$.
  
  \item
  If $(\xi, \omega) \in \Ker \delta_A^*$ decays to zero at infinity, then $(\xi, \omega)=0$.
 \end{enumerate}
\end{proposition}

The Hyperkähler triple of $X$ acts on the $1$-form part of $\Omega^1(\Ad P)$.
It is checked in \cite[Section 4]{Itoh1988} together with \cite[Proposition 2.4]{Itoh1985} that this action restricts to $H^1_{A,-2}$ for all $[A]\in M$.
We thus have a triple of complex structures on $M$.
The following theorem states that this defines a Hyperkähler structure with respect to the standard metric on $M$:

\begin{theorem}[Theorem 2.6 and Proposition 5.1 in \cite{Nakajima1990}]
\label{theorem:moduli-hyperkaehler-structure}
 The metric $g_M$ defined by
 \[
  g_{M}(\alpha,\beta)
  =
  \int_X
  g(\alpha, \beta) \vol _X
  \;\;\;\;
  \text{ for }
  \alpha,\beta \in H^1_{A,-2}
 \]
 and the Hyperkähler triple defined by acting with the Hyperkähler triple of $X$ on the $1$-form part of $\Omega^1(\Ad P)$ is well-defined on $M$ and defines a Hyperkähler structure on $M$.
\end{theorem}

\begin{theorem}[Theorem 2.47 in \cite{Walpuski2013}]
\label{theorem:moduli-space-dimension}
 Let $\rho: \Gamma \rightarrow G$ be a homomorphism, $A_0$ a connection on a bundle $P$ that is flat at infinity as in \cref{assumption:trivialisation-at-infinity} whose holonomy representation is equal to $\rho$ in the sense of \cref{equation:asymptotic-at-infinity-condition}.
 Let $\delta \in (-3,-1)$ and $A=A_0 + \alpha$ for some $\alpha \in L^2_{1,\delta}(\Lambda^1(\Ad P))$.
 Then the $L^2$ index of $\delta_A$, defined as 
 \begin{align*}
  &\dim \{ a \in L^2(\Lambda^1(\Ad P)) \cap C^\infty(\Lambda^1(\Ad P)): \delta_A(a)=0 \}
  \\
 -&\dim \{ \underline{a} \in L^2(\Lambda^0\oplus \Lambda^2_+(\Ad P)) \cap C^\infty(\Lambda^0\oplus \Lambda^2_+(\Ad P)): \delta_A^*(\underline{a})=0 \},
 \end{align*}
  is given by
 \begin{align}
  \ind \delta_A
  =
  -2 \int_X p_1(\Ad P)
  +\frac{2}{|\Gamma |}
  \sum_{g \in \Gamma \setminus \{ e \}}
  \frac{\chi_{\mathfrak{g}}(g)-\dim \mathfrak{g}}{2-\tr g}.
 \end{align}
 Here $p_1(\Ad P)$ is the Chern-Weil representative of the first Pontrjagin class of $P$ and $\chi_{\mathfrak{g}}$ is the character of $g$ acting on $\mathfrak{g}$, the Lie algebra associated with $G$, via $\rho$, and $\tr g$ is the trace of $g$ acting on $\mathfrak{g}$.
 Moreover, if $A$ is an ASD instanton, then $\ind \delta_A=\dim \Ker \delta_A = \dim M$.
\end{theorem}

We will now explain one example of an ASD-instanton on Eguchi-Hanson space that will be needed later.
To this end, we recall the construction of it as a Hyperkähler quotient as explained in \cite{Gibbons1997}.
Let $\mathcal{M}=\mathbb{H}^2$ with quaternionic coordinates $q_a$, $a \in \{1,2\}$, and let $\U(1)$ act on $\mathcal{M}$ via
\begin{align}
\label{equation:hk-quotient-action}
 q_a \mapsto q_a e^{it}, \quad t \in (0, 2 \pi].
\end{align}
A Hyperkähler moment map for this action is given by
\begin{align}
\label{equation:hk-moment-map}
\begin{split}
 \mu : \mathcal{M} & \rightarrow \Im (\mathbb{H}) \simeq \R^3 \tensor \mathfrak{u}(1)
 \\
 (q_1,q_2) & \mapsto
 \frac{1}{2}
 \sum_{a \in \{1,2\}}
 q_a i \overline{q}_a.
\end{split}
\end{align}
Let $\zeta = \frac{i}{2} \in \Im (\mathbb{H})$.
The group $\U(1)$ acts freely on $\mu^{-1}(\zeta)$ and the general theory of Hyperkähler reduction gives rise to a Hyperkähler structure on the four-dimensional manifold $\mathcal{M} \mmod \U(1) := \mu^{-1}(\zeta)/\U(1)$.

\begin{definition}
\label{definition:eguchi-hanson-hyperkaehler}
 The Hyperkähler space $\XEH=\mathcal{M} \mmod \U(1)$ is called \emph{Eguchi-Hanson space}.
\end{definition}

From this, we get our first example of an ASD-instanton on Eguchi-Hanson space:

\begin{proposition}[Section 2 in \cite{Gocho1992}]
\label{proposition:gocho-asd-instanton}
 The $\U(1)$-bundle $\mathcal{R} := \mu^{-1}(i/2) \rightarrow \XEH=\mu^{-1}(i/2)/\U(1)$ admits a non-flat finite energy ASD instanton $A$ asymptotic to the representation $\rho: \Z_2 \rightarrow \U(1)$ determined by $\rho(-1)=-1$ in the sense of \cref{equation:asymptotic-at-infinity-condition}.
\end{proposition}

There exists an ADHM-type construction of ASD-instantons on ALE spaces which generalises this example, cf. \cite{Kronheimer1990}, but we will not need this here.

An additional property of $\mathcal{R}$ that we will need later is the following:

\begin{proposition}
 There exists a lift of the action of the holomorphic isometry group $\U(2)/\{\pm 1\}$ of $\XEH$ to $\mathcal{R}$.
\end{proposition}

\begin{proof}
 We show in \cref{proposition:eguchi-hanson-isometry-group} that the holomorphic isometry group $\U(2)/\{\pm 1\}$ is realised as an action of $\U(2)/\{ \pm 1\}$ on $\mu^{-1}(i/2)$ that commutes with the action of $\U(1)$ on $\mu^{-1}(i/2)$.
 The action of $\U(2)/\{ \pm 1\}$ on $\mu^{-1}(i/2)$ is the desired lift of the action of $\U(2)/\{ \pm 1\}$ on $\XEH$.
\end{proof}

\begin{remark}
\label{proposition:gocho-asd-infinitesimally-rigid}
 We can apply \cref{theorem:moduli-space-dimension} to the $\U(1)$-bundle over $\XEH$ defined before to find that it is rigid.
 As $\Ad \mathcal{R}$ has rank $1$, we have that $p_1(\Ad \mathcal{R}) = c_2(\Ad \mathcal{R} ^ {\C})=0$, and plugging this into the index formula from \cref{theorem:moduli-space-dimension} proves the claim.
\end{remark}

\begin{remark}
 On simply connected compact manifolds it is the case that any $\U(1)$-bundle admits an ASD-instanton that is unique up to the action of the gauge group.
 This is a consequence of the Hodge theorem.
 On non-compact manifolds a variation of the Hodge theorem for $L^2$-forms holds, see \cite[Example 0.15]{Lockhart1987}, and can be used to give an alternative proof of \cref{proposition:gocho-asd-infinitesimally-rigid} without the use of the index formula.
\end{remark}

Before ending the section we will state two results about universal bundles that will be needed later.
The following proposition is proved in \cite[Proposition 5.2.17]{Donaldson1990} for compact manifolds, but the proof carries over to the ALE setting with small alterations.

\begin{proposition}
\label{proposition:tautological-bundle-properties}
 There exist
 \begin{itemize}
  \item 
  a $G$-bundle $\tilde{\mathbb{P}}$ over $M \times
  \hat{X}$ with a natural action of $G_\rho \simeq \mathscr{G}/\mathscr{G}_0$ on $\tilde{\mathbb{P}}$ covering the action of $G_\rho$ on $M$,
  
  \item
  a connection $\tilde{\mathbb{A}} \in \mathscr{A}(\tilde{\mathbb{P}})$ that is invariant under the action of $G_\rho \simeq \mathscr{G}/\mathscr{G}_0$, and
  
  \item
  for each choice of $\phi \in \Iso_{\Gamma}(G, P_\infty)$ a canonical isomorphism of $G$-bundles with $\Gamma$ left action
 $\underline{\phi}:
  \tilde{\mathbb{P}}|_{M \times \{ \infty \}}
  \rightarrow
  G \times M$
 \end{itemize}
 satisfying:
 \begin{itemize}
  \item
  for any element $[A] \in M$ there exists an isomorphism $\tilde{\mathbb{P}}|_{\{[A]\} \times \hat{X}} \simeq \hat{P}$ such that under this isomorphism $\tilde{\mathbb{A}}|_{\{[A]\} \times X}$ and $A$ agree up to the action of $\mathscr{G}_0$.
  
  \item
  if we decompose the curvature of $\tilde{\mathbb{A}}$ over $M \times X$ according to the bi-grading on $\Lambda^* T^*(M \times X)$ induced by $T^*(M \times X)=\pi_1 ^* T^* M \oplus \pi_2 ^* T^* X$, then its components satisfy the following:
  
  \begin{itemize}
  \item
  $F_{\tilde{\mathbb{A}}}^{1,1} \in \Gamma(\Hom(\pi_1^* T^* M, \pi_2^* T^* X \tensor \Ad P))$ at $([A], x)$ is the evaluation of $a \in T_{[A]}M$ at $x$,
  
  \item
  $F_{\tilde{\mathbb{A}}}^{0,2} \in \Gamma(\pi_2^* \Lambda^- ( X)^* \tensor \Ad P)$, where $\Lambda^-$ is defined using the ALE metric on $X$,
 \end{itemize}
 
 \item
 $\underline{\phi}^*A_{\text{product}}
 =
 \tilde{\mathbb{A}}|_{M \times \{\infty\}}$,
 where $A_{\text{product}} \in \mathscr{A}(G \times M)$ denotes the product connection.
 \end{itemize}
\end{proposition}

By \cref{proposition:eguchi-hanson-isometry-group}, the group of holomorphic isometries acting on $\XEH$ is $\U(2)/\{\pm 1\}$.
This induces a non-effective action of $\U(2)$ on $\XEHh$ by demanding that each group element fixes $\infty \in \XEHh$.
Then $\U(2)$ acts from the left on $M$ (and equally $M^{\text{orb}}$) as follows:
$\U(2)$ is connected, so $(u^{-1})^*E$ and $E$ are homotopic bundles and in particular isomorphic.
Different choices of isomorphism give rise to gauge equivalent connections, so $[(u^{-1})^* A] \in M$ is well-defined.

Later on (cf. \cref{definition:moduli-bundle-cov-derivative}) we will need the following assumption:

\begin{assumption}
\label{assumption:u2-lifts-to-universal-bundle}
 The action of $\U(2)$ on $M \times \XEHh$ can be lifted to an action on $\tilde{\mathbb{P}}$ that preserves $\tilde{\mathbb{A}}$.
\end{assumption}

In the examples constructed in \cref{section:examples} this assumption will be satisfied because of the following proposition:

\begin{proposition}
\label{proposition:lift-to-tautological-bundle}
 Let $\tilde{\mathbb{P}} \rightarrow M \times \XEHh$ be the tautological bundle with tautological connection $\tilde{\mathbb{A}}$ from \cref{proposition:tautological-bundle-properties}.

 If the action of $\U(2)$ on $\XEHh$ can be lifted to an action on $\hat{P}$, then the action of $\U(2)$ on $M \times \XEHh$ can be lifted to an action on $\tilde{\mathbb{P}}$.
 If it exists, this lift can be chosen to preserve $\tilde{\mathbb{A}}$.
\end{proposition}

\begin{proof}
 First, assume that the action of $\U(2)$ on $\XEHh$ can be lifted to an action on $\hat{P}$.
 This is equivalent to saying that for all $g \in G$ there exists a bundle isomorphism $\xi_g: \hat{P} \rightarrow \hat{P}$ covering $g: \XEHh \rightarrow \XEHh$.
 The bundle $\tilde{\mathbb{P}}$ is defined as $\tilde{\mathbb{P}} = \pi_2^* \hat{P}/\mathscr{G}^{\orb}_0$, where $\pi_2:\mathscr{A}^{\orb}_{\asd} \times \XEHh \rightarrow \XEHh$ is the projection onto the second factor.
 Let $([A], x) \in M \times \XEHh$ and $[u] \in \tilde{\mathbb{P}}_{([A], x)}$ where $u \in \left( \pi_2^* \hat{P} \right) _{(A,x)} \simeq \hat{P}_x$.
 We define $\kappa_g: \tilde{\mathbb{P}} \rightarrow \tilde{\mathbb{P}}$ covering $g: M \times \XEHh \rightarrow M \times \XEHh$ via $\kappa_g [u] := [\xi_g (u)]$.
 The check that this is well-defined and that this lift preserves $\tilde{\mathbb{A}}$ is straightforward, using the definition of $\tilde{\mathbb{A}}$ from \cite[Section 5.2.3]{Donaldson1990}.
\end{proof}

\subsection{Gauge Theory on $G_2$-manifolds}

\begin{definition}
\label{definition:g2-instanton}
 Let $(Y,\varphi)$ be a $G_2$-manifold, $\psi=*_\varphi \varphi$, and $E$ be a principal bundle over $Y$.
 A connection $A \in \mathscr{A}(E)$ is called a \emph{$G_2$-instanton}, if
 $F_A \in \Gamma(\Lambda^2_{14} \tensor \Ad E)$, i.e. (by \cref{theorem:g2-decomposition-of-forms})
 \begin{align}
 \label{equation:vanilla-instanton-equation}
  F_A \wedge \psi =0,
 \end{align}
 where the wedge product is taken in the $2$-form part of $\Lambda^2 \tensor \Ad E$.
\end{definition}

\begin{example}
\label{example:pullback-of-asd-is-g2-instanton}
 Let $A$ be an ASD instanton on a bundle $E$ over a Hyperkähler $4$-fold $X$.
 Denote by $p_X: \R^3 \times X \rightarrow X$ the projection onto the second factor.
 Then $\R^3 \times X$ carries the torsion-free $G_2$-structure $\varphi$ from \cref{equation:product-g2-structure}, and $p_X^*A$ is a $G_2$-instanton on the bundle $p_X^*E$ with respect to this $G_2$-structure.
 To see this, let $\omega_1,\omega_2,\omega_3 \in \Omega^2(X)$ denote a Hyperkähler triple on $X$.
 These $2$-forms are self-dual, thus $A$ being ASD is equivalent to $F_A \wedge \omega_i=0$ for $i \in \{1,2,3\}$.
 Recall that for the product $G_2$-structure, we have that
 \begin{align*}
  * \varphi
  =
  \psi
  =
  \frac{1}{2}
  \omega_1^2
  -
  \d x_{12} \wedge \omega_3
  -
  \d x_{23} \wedge \omega_1
  -
  \d x_{31} \wedge \omega_2
 \end{align*}
 and therefore
 \begin{align*}
  F_{p_X^*A} \wedge \psi
  &=
  p_X^*(F_A) \wedge \psi
  =0.
 \end{align*}
\end{example}

The linearisation of the $G_2$-instanton equation \cref{equation:vanilla-instanton-equation} is not elliptic.
This problem is overcome in the following proposition:

\begin{lemma}[Proposition 1.98 in \cite{Walpuski2013}]
\label{lemma:equivalent-instanton-definition}
 Let $(Y,\varphi)$ be a compact $G_2$-manifold, $\psi=*_\varphi \varphi$, and $E$ be a principal bundle over $Y$, and $A \in \mathscr{A}(E)$.
 Then $A$ is a $G_2$-instanton if and only if there exists $\xi \in \Omega^0(Y,\Ad E)$ such that
 \begin{align}
  *(F_A \wedge \psi)+\d _A \xi=0.
 \end{align}
\end{lemma}

Adding the Coulomb gauge condition to this, we consider for a fixed connection $A \in \mathscr{A}(E)$, $\xi \in \Omega^0(Y,\Ad E)$, and $a \in \Omega^1(Y,\Ad E)$ the system
\begin{align}
\label{eqution:g2-instanton-equation-made-elliptic}
 \begin{split}
  *(F_{A+a} \wedge \psi) + \d_{A+a} \xi &= 0
  \\
  \d^*_A a &= 0.
 \end{split}
\end{align}
Here, every solution $(\xi, a)$ defines the $G_2$-instanton $A+a$ which is in Coulomb gauge with respect to $A$.
A computation in coordinates shows that the linearisation of \cref{eqution:g2-instanton-equation-made-elliptic} is an elliptic operator:

\begin{proposition}
\label{proposition:g2-instanton-equation-elliptic}
 The linearisation of \cref{eqution:g2-instanton-equation-made-elliptic} is
 \begin{align}
 \label{equation:g2-instanton-linearisation}
 \begin{split}
  L_A:
  (\Omega^0 \oplus \Omega^1)(Y,\Ad E)
  & \rightarrow 
  (\Omega^0 \oplus \Omega^1)(Y,\Ad E)
  \\
  \begin{pmatrix}
   \xi \\ a
  \end{pmatrix}
  & \mapsto
  \begin{pmatrix}
   0 & \d^*_A \\
   \d_A & *(\psi \wedge \d_A)
  \end{pmatrix}
  \begin{pmatrix}
   \xi \\ a
  \end{pmatrix}
 \end{split}
 \end{align}
 which is a self-adjoint elliptic operator if $\d^* \varphi =0$.
\end{proposition}

\begin{remark}
 A coordinate-free proof for the ellipticity of operator $L_A$ is given in \cite[Section 3, Lemma 4]{ReyesCarrion1998}.
\end{remark}

\section{Resolutions of $G_2$-orbifolds}
\label{section:review-of-the-manifold-construction}

\subsection{Torsion-Free $G_2$-structures on Resolutions of $T^7/\Gamma$}
\label{section:torsion-free-on-resolution-of-T7}
\label{section:torsion-free-structures-on-the-generalised-kummer-construction}

In the two articles \cite{Joyce1996}, Joyce constructed the first compact examples of manifolds with holonomy equal to $G_2$.
One starts with the flat $7$-torus $T^7$, which carries the flat $G_2$-structure $\varphi_0$.
A quotient of the torus by a finite group of maps $\Gamma$ preserving the $G_2$-structure still carries a flat $G_2$-structure, but has \emph{singularities}.
The singularities are modelled on $\R^3 \times (\C^2/\{\pm 1\})$ and admit a resolution $\R^3 \times \XEH$.
Gluing this resolution into $T^7/\Gamma$, we have two different $G_2$-structures near the resolution locus:
the flat $G_2$-structure $\varphi_0$ and the product $G_2$-structure $\varphi^P_t$ which depends on one real parameter that is proportional to the size of the unique minimal $2$-sphere in $\XEH$.
Gluing $\varphi_0$ and $\varphi^P_t$, one obtains a $1$-parameter family of smooth manifolds $N_t$ and $G_2$-structures $\varphi^t \in \Omega^3(N)$ depending on a single real parameter $t \in (0,1)$.
While $\varphi_0$ and $\varphi^P_t$ are torsion-free, the glued structure $\varphi^t$ is not, because of the error introduced by the gluing.
However, one can check that
\begin{align}
 \label{equation:pregluing-estimate-G2-structure-T7}
 \|{
  \varphi^P_t-\varphi^t
 }_{C^k}
 \leq
 ct^4
\end{align}
for any $k$ and a constant $c>0$ independent of $t$.
This implies that for small $t$, \cref{theorem:original-torsion-free-existence-theorem} can be applied and one obtains a torsion-free $G_2$-structure $\tilde{\varphi}^t \in \Omega^3(N)$ satisfying
\[
 \|{
  \tilde{\varphi}^t - \varphi^t
 }_{C^0}
 \leq
 ct^{1/2}.
\]
We need an improved version of this estimate using weighted Hölder norms denoted by
$\|{
  \cdot
 }_{C^{2,\alpha/2}_{\beta;t}}$
from \cite{Platt2020}.
At this point, we will not reproduce the definition of these norms from the reference, because we will define more general norms in \cref{definition:hoelder-norms-for-gauge-with-cases}.

\begin{theorem}[Theorem 4.58 in \cite{Platt2020}]
\label{corollary:kummer-construction-simplified-torsion-free-estimate}
Let $N_t$ be the resolution of $T^7/\Gamma$ and $\varphi^t \in \Omega^3(N_t)$ the $G_2$-structure with small torsion from \cref{section:torsion-free-on-resolution-of-T7}.
There exists $c>0$ independent of $t$ such that the following is true:
for $t$ small enough, there exists $\eta^t \in \Omega^2(N_t)$ such that $\tilde{\varphi}=\varphi^t+\d \eta^t$ is a torsion-free $G_2$-structure, and $\eta^t$ satisfies
\[
 \|{
  \eta^t
 }_{C^{2,\alpha/2}_{\beta;t}}
 \leq
 ct^{7/2-\beta}.
\]
In particular,
\[
\|{\tilde{\varphi}-\varphi^t}_{L^\infty} \leq ct^{5/2} \text{ and }
\|{\tilde{\varphi}-\varphi^t}_{C^{0,\alpha/2}} \leq ct^{5/2-\alpha/2} \text{ as well as }
\|{\tilde{\varphi}-\varphi^t}_{C^{1,\alpha/2}} \leq ct^{3/2-\alpha/2}.
\]
\end{theorem}

\subsection{Torsion-Free $G_2$-Structures on Joyce-Karigiannis Manifolds}
\label{section:torsion-free-on-joyce-karigiannis}

In \cite{Joyce2017}, the authors constructed new examples of compact manifolds with holonomy $G_2$.
They first used a gluing procedure to construct a $G_2$-structure with small torsion and then applied \cref{theorem:original-torsion-free-existence-theorem} to perturb this $G_2$-structure into a torsion-free $G_2$-structure.

The main difference to Joyce's original construction is the following:
if one uses the cutoff procedure from the $T^7/\Gamma$ case in the new setting, one produces a $G_2$-structure that does not satisfy the necessary estimates to apply \cref{theorem:original-torsion-free-existence-theorem}.
The authors of \cite{Joyce2017} overcome this problem by constructing a $G_2$-structure with \emph{even} smaller torsion, to which \cref{theorem:original-torsion-free-existence-theorem} \emph{can} be applied.

\subsubsection{Ingredients for the Construction}

Let $Y$ be a compact manifold endowed with a torsion-free $G_2$-structure $\varphi$.
Write $g$ for the metric induced by $\varphi$.
Let $\iota: Y \rightarrow Y$ be a $G_2$-involution, i.e. satisfying $\iota^2=\Id$, $\iota \neq \Id$, $\iota^* \varphi=\varphi$.
We then have:

\begin{proposition}[Proposition 2.13 in \cite{Joyce2017}]
 Let $L= \fix(\iota)$ and assume $L \neq \emptyset$.
 Then $L$ is a smooth, orientable $3$-dimensional compact submanifold of $Y$ which is totally geodesic, and, with respect to a canonical orientation, is associative.
\end{proposition}

\begin{assumption}
 We assume that $L$ is nonempty, and we assume we are given a closed, coclosed, nowhere vanishing $1$-form $\lambda$ on $L$.
\end{assumption}

Such a $1$-form need not exist, and cases in which its existence can be guaranteed are discussed in \cite[Section 7.1]{Joyce2017}.

\subsubsection{$G_2$-structures on the Normal Bundle $\nu$ of $L$}
\label{subsection:g2-structures-on-the-normal-bundle}

The metric defined by $\varphi$ defines a splitting
\begin{align}
\label{equation:TM-metric-splitting}
 TY|_L
 \simeq
 \nu \oplus TL,
\end{align}
which is orthogonal with respect to $g$.
Write $g_L$ for the metric on $L$ induced by $g$ and $g|_L=h_\nu \oplus g_L$.
Write $\tilde{\nabla}^\nu$ for some connection on $\nu$.
For now, we may think of $\tilde{\nabla}^\nu$ as being the restriction of the Levi-Civita connection of $g$ to $\nu \rightarrow L$, but later we will need the freedom to choose another connection.
We write elements in $\nu$ as $(x,\alpha)$, where $x \in L$, $\alpha \in \nu_x$.
For $R>0$ let
\begin{align*}
 U_R
 =
 \{
 (x,\alpha) \in \nu:
 | \alpha |_{h_\nu} < R
 \}.
\end{align*}
Write $\pi : U_R \rightarrow L$ for the projection $(x,\alpha)\mapsto x$.
We will make use of a map $\Upsilon: U_R \rightarrow Y$ satisfying the following:
\begin{enumerate}
 \item
 $\Upsilon$ is a diffeomorphism onto its image,
 
 \item
 $\Upsilon(x,0)=x$ for $x \in L$,
 
 \item
 $\Upsilon(x,-\alpha)=\iota \circ \Upsilon(x,\alpha)$ for $(x,\alpha) \in U_R$,
 
 \item
 the induced pushforward $\Upsilon_*: TU_R \rightarrow TY$ restricted to the zero section of $TU_R$ is the identity map on $T_xL \oplus \nu_x$.
\end{enumerate}
For example, $\Upsilon = \exp$ would satisfy these four conditions for small $R$.
But later on we require $\Upsilon$ to satisfy an extra condition that $\exp$ need not satisfy.

Write $(\cdot t): \nu \rightarrow \nu$ for the dilation map $(x,\alpha) \mapsto (x,t\alpha)$, and for $t \neq 0$, define $\Upsilon_t = \Upsilon \circ (\cdot t): U_{\left| t \right|^{-1}R} \rightarrow Y$.

The connection $\tilde{\nabla}^\nu$ defines a splitting 
\begin{align}
\label{equation:T-nu-connection-splitting}
 T\nu = V \oplus H,
 \;\;\;
 \text{ where }
 V \simeq \pi^*(\nu)
 \text{ and }
 H \simeq \pi^*(TL),
\end{align}
where $V$ and $H$ are the vertical and horizontal subbundles of the connection.
Combining \cref{equation:TM-metric-splitting,equation:T-nu-connection-splitting}, we have that $T \nu \simeq \pi^*(TY|_L)$.
Denote by
\begin{align}
\label{equation:varphi-nu}
 \varphi^\nu \in \Omega^3(\nu), 
 \psi^\nu \in \Omega^4(\nu), \text{ and }
 g^\nu \in S^2(\nu)
\end{align}
the structures obtained from $\varphi$, $\psi$, and $g$ via this isomorphism and for $t > 0$ write
$\varphi^\nu_t=(\cdot t)^* \varphi^\nu$, as well as
$\psi^\nu_t=(\cdot t)^* \psi^\nu$, and
$g^\nu_t=(\cdot t)^* g^\nu$.
Note that this definition implicitly depends on the choice of $\tilde{\nabla}^\nu$.
The main result of \cite[Section 3]{Joyce2017} is then:

\begin{proposition}
\label{proposition:jk-normal-bundle-choices}
 There exist $R>0$, a connection $\tilde{\nabla}^\nu$ on $\nu$ and a map $\Upsilon: U_R \rightarrow M$ satisfying
 \begin{enumerate}
 \item
 $\Upsilon$ is a diffeomorphism onto its image,
 
 \item
 $\Upsilon(x,0)=x$ for $x \in L$,
 
 \item
 $\Upsilon(x,-\alpha)=\iota \circ \Upsilon(x,\alpha)$ for $(x,\alpha) \in U_R$,
 
 \item
 the induced pushforward $\Upsilon_*: TU_R \rightarrow TY$ restricted to the zero section of $TU_R$ is the identity map on $T_xL \oplus \nu_x$,
 \end{enumerate}
 and for $t>0$ a closed $G_2$-structure $\tilde{\varphi}^\nu_t$ on $\nu/ \{ \pm 1\}$ and closed $4$-form $\tilde{\psi}^\nu_t \in \Omega^4(\nu/\{\pm 1\})$ satisfying the following properties:
 first,
 \begin{align}
 \label{equation:phi-nu-tilde-phi-nu-difference}
  \varphi^\nu_t-\tilde{\varphi}^\nu_t
  =
  \mathcal{O}(t^2r^2)
  \quad \text{ and } \quad
  \psi^\nu_t-\tilde{\psi}^\nu_t
  =
  \mathcal{O}(t^2r^2).
 \end{align}
 Second, there exist $\eta \in \Omega^2(\nu), \zeta \in \Omega^3(\nu)$ such that
 \begin{align*}
  | \eta |_{g^\nu}
  &=
  \mathcal{O}(r^3)
  &
  \text{ and }
  &&
  | \d \eta |_{g^\nu}
  &=
  \left|
  \Upsilon^* \varphi- \tilde{\varphi}_1^\nu|_{U_R} 
  \right|_{g^\nu}
  =
  \mathcal{O}(r^2),
  \\
  | \zeta |_{g^\nu}
  &=
  \mathcal{O}(r^3)
  &
  \text{ and }
  &&
  | \d \zeta |_{g^\nu}
  &=
  \left|
  \Upsilon^* \psi- \tilde{\psi}_1^\nu|_{U_R} 
  \right|_{g^\nu}
  =
  \mathcal{O}(r^2).
 \end{align*} 
\end{proposition}

\subsubsection{$G_2$-structures on the Resolution $P$ of $\nu /\{\pm 1\}$}
\label{subsection:g2-structures-on-the-resolution}

The $G_2$-structure $\varphi \in \Omega^3(Y)$ defines for all $x \in Y$ a cross product $\times : T_x Y \times T_x Y \rightarrow T_x Y$ as in \cref{definition:cross-product}.
We then have a complex structure $I \in \End(\nu)$ given by
 \begin{align}
 \label{equation:complex-structure-on-normal-bundle}
  I(V)
  =
  \frac{\lambda}{\left| \lambda \right|} \times V
  \text{ for }
  V \in \nu_x, x \in L.
 \end{align}
Recall the metric $h_\nu$ on $\nu$ defined by $g|_L=h_\nu \oplus g_L$, cf. \cref{subsection:g2-structures-on-the-normal-bundle}. 
Then $I$ and $h_\nu$ together define a $\U(2)$-reduction of the frame bundle of $\nu$.
Denote by ${\XEH}$ the Eguchi-Hanson space with Hyperkähler triple $\tilde{\omega}_1,\tilde{\omega}_2,\tilde{\omega}_3$.
Denote by $\rho: {\XEH} \rightarrow \C^2/\{ \pm 1\}$ the blowup map of the blowup with respect to the complex structure induced by $\tilde{\omega}_1$ and let
\begin{align}
\label{equation:definition-of-P}
 P
 =
 \Fr \times_{\U(2)} {\XEH}.
\end{align}
Denote by $\sigma:P \rightarrow L$ the projection of this bundle.
Analogously, we have
\begin{align*}
 \nu/\{\pm 1\}
 =
 \Fr \times_{\U(2)} \C^2/\{\pm 1\}.
\end{align*}
Let $L' \subset L$ be a nonempty, open set on which we can extend $e_1:= \frac{\lambda}{| \lambda |} \in T^*(L')$ to an orthonormal basis $(e_1,e_2,e_3)$.
Then there exist $\hat{\omega}^I,\hat{\omega}^J,\hat{\omega}^K \in \Omega^2((\nu/\{\pm1 \})|_{L'})$ such that $\varphi^\nu$ from \cref{equation:varphi-nu} has the form
\begin{align}
\label{definition:varphi-nu-t}
 \varphi^\nu
 =
 e_1 \wedge e_2 \wedge e_3
 -
 \hat{\omega}^I \wedge e_1
 -
 \hat{\omega}^J \wedge e_2
 -
 \hat{\omega}^K \wedge e_3.
\end{align}
We define $\check{\omega}^I,\check{\omega}^J,\check{\omega}^K \in \Omega^2(P|_{L'})$ as follows:
For $x \in L'$, let $f \in \Fr_x$ such that $f: (\nu/\{\pm 1\})_x \rightarrow \C^2/\{ \pm 1\}$ satisfies
\begin{align*}
 f^*(\omega_1,\omega_2,\omega_3)
 =
 (
 \hat{\omega}^I|_{\nu_x},
 \hat{\omega}^J|_{\nu_x},
 \hat{\omega}^K|_{\nu_x}
 ),
\end{align*}
where $(\omega_1,\omega_2,\omega_3)$ denotes the flat Hyperkähler triple on $\C^2/\{\pm 1\}$.
This choice of $f$ defines isomorphisms of complex surfaces $P_x \simeq {\XEH}$ and $(\nu / \{\pm 1\})_x \simeq \C^2/\{\pm 1\}$.
Let $\check{\omega}^I,\check{\omega}^J,\check{\omega}^K \in \Omega^2(P_x)$ be the pullback of $\tilde{\omega}_1,\tilde{\omega}_2,\tilde{\omega}_3 \in \Omega^2({\XEH})$ under this isomorphism.
This is independent of the choice of $f$, and therefore defines $\check{\omega}^I,\check{\omega}^J,\check{\omega}^K \in \Omega^2(P_x)$.
The following diagram sums up the situation:

\begin{equation}
\label{equation:hyperkaehler-commutative-diagram}
\begin{tikzcd}
{(P_x,\check{\omega}^I|_{P_x},\check{\omega}^J|_{P_x},\check{\omega}^K|_{P_x})} \arrow[r, "\simeq"] \arrow[d, "\rho"]                           & {({\XEH},\tilde{\omega}_1,\tilde{\omega}_2,\tilde{\omega}_3)} \arrow[d, "\rho"] \\
{(\nu_x/\{\pm 1\},\hat{\omega}^I|_{\nu_x/\{\pm 1\}},\hat{\omega}^J|_{\nu_x/\{\pm 1\}},\hat{\omega}^K|_{\nu_x/\{\pm 1\}})} \arrow[r, "\simeq"] & {(\C^2/\{\pm 1\},\omega_1,\omega_2,\omega_3)}     
\end{tikzcd}
\end{equation}
Here, by abuse of notation we denoted the map $P_x \rightarrow \nu_x/\{\pm 1\}$ which makes the diagram commutative also by $\rho$.
Horizontal arrows pull Hyperkähler triples back to one another, Hyperkähler triples connected by vertical arrows are asymptotically close to each other.
A complicated point is the actual definition of $\check{\omega}^I,\check{\omega}^J,\check{\omega}^K$ as $2$-forms on $P|_{L'}$.
\Cref{equation:hyperkaehler-commutative-diagram} tells us what they look like fibrewise.
To make sense of them as global objects on $P$, one needs to choose a connection on $P$.
In \cite{Joyce2017}, the horizontal subspaces $\breve{H}$ were defined to this end which allows us to decompose forms on $P$ into vertical and horizontal components, much like for forms on $\nu$.
There are then unique vertical $2$-forms which restrict to $\check{\omega}^I|_{P_x},\check{\omega}^J|_{P_x},\check{\omega}^K|_{P_x}$ on every fibre.

We are now ready to define $\varphi^P_t \in \Omega^3(P|_{L'})$, $\psi^P_t \in \Omega^4(P|_{L'})$ via
\begin{align}
\label{definition:varphi-P-t}
\begin{split}
 \varphi^P_t
 &:=
 \check{\varphi}_{0,3}
 +
 t^2\check{\varphi}_{2,1}
 \\
 &
 :=
 \sigma^*(e_1 \wedge e_2 \wedge e_3)
 -
 t^2
 \left(
  \sigma^*(e_1) \wedge \check{\omega}^I-
  \sigma^*(e_2) \wedge \check{\omega}^J-
  \sigma^*(e_3) \wedge \check{\omega}^K
 \right)
 ,
\end{split}
 \\
 \notag
 \psi^P_t
 &:=
 t^4 \check{\psi}_{4,0}
 +
 t^2 \check{\psi}_{2,2}
 \\
 \notag
 &
 :=
 \frac{1}{2} \check{\omega}^I \wedge \check{\omega}^I
 -
 \sigma^*(e_2 \wedge e_3) \wedge \check{\omega}^I-
 \sigma^*(e_3 \wedge e_1) \wedge \check{\omega}^J-
 \sigma^*(e_1 \wedge e_2) \wedge \check{\omega}^K.
\end{align}
These expressions are independent of the choice of $(e_2,e_3)$, and therefore define forms $\varphi^P_t \in \Omega^3(P), \psi^P_t \in \Omega^4(P)$, not just forms over $L' \subset L$.
Let also $g^P_t$ denote the metric induced by $\varphi^P_t$.

As in the previous section, we add terms to $\varphi^P_t$ and $\psi^P_t$ to define \emph{closed} forms on $P$, and we have the following control over how they are asymptotic to forms on $\nu/\{\pm 1\}$:
\begin{proposition}[Section 4.5 in \cite{Joyce2017}]
\label{proposition:xi-estimates}
 There exist $\xi_{1,2},\xi_{0,3} \in \Omega^3(P)$, $\tau_{1,1} \in \Omega^2(\{x \in P: \check{r}(x) > 1)$, such that
 \begin{align*}
  \tilde{\varphi}^P_t
  :=
  \varphi^P_t+t^2 \xi_{1,2}+t^2 \xi_{0,3}
 \end{align*}
 is closed and satisfies
 \begin{align}
 \label{equation:phi-P-tilde-phi-nu-tilde-difference}
  \tilde{\varphi}^P_t
  =
  \rho^* \tilde{\varphi}^\nu_t
  +
  t^2 \d \tau_{1,1}
 \end{align}
 where $\check{r}>1$.
 These forms satisfy the following estimates:
 \begin{align}
  \notag
  \left|
   \nabla^k (t^2 \xi_{1,2})
  \right|_{g^P_t}
  &=
  \begin{cases}
   \mathcal{O}(t^{1-k}), & \check{r} \leq 1,\\
   \mathcal{O}(t^{1-k} \check{r}^{-3-k}), & \check{r} > 1,
  \end{cases}
  \\  
  \label{equation:xi-0-3-estimate}
  \left|
   \nabla^k (t^2 \xi_{0,3})
  \right|_{g^P_t}
  &=
  \begin{cases}
   \mathcal{O}(t^{2-k}), & \check{r} \leq 1,\\
   \mathcal{O}(t^{2-k} \check{r}^{2-k}), & \check{r} > 1,
  \end{cases}
  \\
  \label{equation:tau-1-1-estimate}
  \left|
   \nabla^k (t^2 \tau_{1,1})
  \right|_{g^P_t}
  &=
  \mathcal{O}(t^{1-k}\check{r}^{-3-k}).
 \end{align}
\end{proposition}

\begin{proposition}[Section 4.5 in \cite{Joyce2017}]
\label{proposition:v-estimates}
 There exist $\chi_{1,3}, \theta_{3,1}, \theta_{2,2} \in \Omega^4(P)$, $v_{1,2} \in \Omega^3(\{x \in P: \check{r}(x) > 1)$, such that
 \begin{align}
  \tilde{\psi}^P_t
  :=
  \psi^P_t+t^2 \chi_{1,3}+t^4 \theta_{3,1}+t^4\theta_{2,2}
 \end{align}
 is closed and satisfies
 \begin{align}
  \label{equation:psi-P-tilde-psi-nu-tilde-difference}
  \tilde{\psi}^P_t
  =
  \rho^* \tilde{\psi}^\nu_t
  +
  t^2 \d v_{1,2}
 \end{align}
 where $\check{r}>1$.
 These forms satisfy the following estimates:
 \begin{align}
  \left|
   \nabla^k (t^2 \chi_{1,3})
  \right|_{g^P_t}
  &:=
  \begin{cases}
   \mathcal{O}(t^{1-k}), & \check{r} \leq 1,\\
   \mathcal{O}(t^{1-k} \check{r}^{-3-k}), & \check{r} > 1,
  \end{cases}
  \\  
  \left|
   \nabla^k (t^4 \theta_{3,1})
  \right|_{g^P_t}
  &:=
  \begin{cases}
   \mathcal{O}(t^{1-k}), & \check{r} \leq 1,\\
   0, & \check{r} > 1,
  \end{cases}
  \\  
  \left|
   \nabla^k (t^4 \theta_{2,2})
  \right|_{g^P_t}
  &:=
  \begin{cases}
   \mathcal{O}(t^{2-k}), & \check{r} \leq 1,\\
   \mathcal{O}(t^{2-k} \check{r}^{2-k}), & \check{r} > 1,
  \end{cases}
  \\
  \label{equation:v-1-2-estimate}
  \left|
   \nabla^k (t^2 v_{1,2})
  \right|_{g^P_t}
  &:=
  \mathcal{O}(t^{1-k}\check{r}^{-3-k}).
 \end{align}
\end{proposition}

\subsubsection{Correcting for the Leading-order Errors on $P$}

Armed with the $G_2$-structures $\varphi$ on $Y$ and $\tilde{\varphi}^P_t$ on $P$, we could define a glued together $G_2$-structure just as is done in the case of resolutions of $T^7/\Gamma$.
However, in this case it would turn out that the torsion of the glued together $G_2$-structure is too big and \cref{theorem:original-torsion-free-existence-theorem} cannot be applied.
We thus make use of the following correction terms which will make the torsion of the glued together $G_2$-structure small enough.

\begin{theorem}[Theorem 5.1 in \cite{Joyce2017}]
\label{theorem:jk-correction-theorem}
 There exist 
 $\alpha_{0,2}, \alpha_{2,0} \in \Omega^2(P)$,
 $\beta_{0,3}, \beta_{2,1} \in \Omega^3(P)$,
 satisfying for all $t>0$ the equation
 \[
  (D_{\varphi^P_t} \Theta)
  \left( t^2 [\d \alpha_{0,2}]_{1,2}+t^4[\d \alpha_{2,0}]_{3,0}+t^2\xi_{1,2} \right)
  \\
  =
  t^2 \d \beta_{0,3} +t^4 [\d \beta_{2,1}]_{3,1}+t^2\chi_{1,3}+t^4\theta_{3,1}.
 \]
 Moreover, for $\gamma > 0$ sufficiently small and for all $k \geq 0$, these forms satisfy the following estimates
 \begin{align*}
  \left|
   \nabla^k (t^2 \alpha_{0,2})
  \right|_{g^P_t}
  &=
  \begin{cases}
   \mathcal{O}(t^{2-k}), & \check{r} \leq 1, \\
   \mathcal{O}(t^{2-k}\check{r}^{-2-k+\gamma}), 
   & \check{r} \geq 1,
  \end{cases}
  \\
  \left|
   \nabla^k (t^4 \alpha_{2,0})
  \right|_{g^P_t}
  &=
  \begin{cases}
   \mathcal{O}(t^{2-k}), & \check{r} \leq 1, \\
   \mathcal{O}(t^{2-k}\check{r}^{-2-k+\gamma}), 
   & \check{r} \geq 1,
  \end{cases}
  \\
  \left|
   \nabla^k (t^2 \beta{0,3})
  \right|_{g^P_t}
  &=
  \begin{cases}
   \mathcal{O}(t^{2-k}), & \check{r} \leq 1, \\
   \mathcal{O}(t^{2-k}\check{r}^{-2-k+\gamma}), 
   & \check{r} \geq 1,
  \end{cases}
  \\
  \left|
   \nabla^k (t^4 \beta{2,1})
  \right|_{g^P_t}
  &=
  \begin{cases}
   \mathcal{O}(t^{2-k}), & \check{r} \leq 1, \\
   \mathcal{O}(t^{2-k}\check{r}^{-2-k+\gamma}), 
   & \check{r} \geq 1,
  \end{cases}
 \end{align*} 
\end{theorem}

\subsubsection{$G_2$-structures on the Resolution $N_t$ of $Y/\<\iota\>$}

We are now ready to glue together $P$ and $Y/\< \iota \>$ to a manifold, and define a $G_2$-structure with small torsion on it.

\begin{definition}
\label{definition:joyce-karigiannis-N-t}
Define
\begin{align}
 N_t
 :=
 \left[
 \rho^{-1}(U_{t^{-1}R}/\{ \pm 1\})
 \coprod
 (Y \setminus L)/\< \iota \>
 \right]
 /\sim,
\end{align}
where $x \sim \Upsilon_t \circ \rho(x)$ for $x \in \rho^{-1}(U_{t^{-1}R}/\{\pm 1\})$.
\end{definition}

\begin{definition}
Let $a:[0,\infty) \rightarrow \R$ be a smooth function with $a(x)=0$ for $x \in [0,1]$, and $a(x)=1 \in [2, \infty)$.
Define then
\begin{align}
\label{equation:varphi-on-joyce-karigiannis}
 \varphi^N_t
 &=
 \begin{cases}
  \tilde{\varphi}^P_t+
  \d \,[ t^2 \alpha_{0,2}+t^4 \alpha_{2,0} ]
  ,
  & \text{ if }
  \check{r} \leq t^{-1/9},
  \\
  \tilde{\varphi}^P_t+
  \d \,[ t^2 \alpha_{0,2}+t^4 \alpha_{2,0} 
  +a(t^{1/9} \check{r}) \cdot \Upsilon_* \eta
  ]
  ,
  & \text{ if }
  t^{-1/9} \leq \check{r} \leq 2t^{-1/9},
  \\
  \tilde{\varphi}^P_t+
  \d \,[ t^2 \alpha_{0,2}+t^4 \alpha_{2,0} 
  + \Upsilon_* \eta
  ]
  ,
  & \text{ if }
  2t^{-1/9} \leq \check{r} \leq t^{-4/5},
  \\
  \tilde{\varphi}^\nu_t+
  \d \,[ (1-a(t^{4/5}\check{r}))(t^2 \tau_{1,1}+ t^2 \alpha_{0,2}+t^4 \alpha_{2,0})
  + \Upsilon_* \eta
  ]
  ,
  & \text{ if }
  t^{-4/5} \leq \check{r} \leq 2t^{-4/5},
  \\
  \varphi
  ,
  & \text{ elsewhere},
 \end{cases}
 \\
\label{equation:psi-on-joyce-karigiannis}
 \psi^N_t
 &=
 \begin{cases}
  \tilde{\psi}^P_t+
  \d \,[ t^2 \beta_{0,3}+t^4 \beta_{2,1} ]
  ,
  & \text{ if }
  \check{r} \leq t^{-1/9},
  \\
  \tilde{\psi}^P_t+
  \d \,[ t^2 \beta_{0,3}+t^4 \beta_{2,1} 
  +a(t^{1/9} \check{r}) \cdot \Upsilon_* \zeta
  ]
  ,
  & \text{ if }
  t^{-1/9} \leq \check{r} \leq 2t^{-1/9},
  \\
  \tilde{\psi}^P_t+
  \d \,[ t^2 \beta_{0,3}+t^4 \beta_{2,1} 
  + \Upsilon_* \zeta
  ]
  ,
  & \text{ if }
  2t^{-1/9} \leq \check{r} \leq t^{-4/5},
  \\
  \tilde{\psi}^\nu_t+
  \d \,[ (1-a(t^{4/5}\check{r}))(t^2 v_{1,2}+ t^2 \beta_{0,3}+t^4 \beta_{2,1})
  + \Upsilon_* \zeta
  ]
  ,
  & \text{ if }
  t^{-4/5} \leq \check{r} \leq 2t^{-4/5},
  \\
  \psi
  ,
  & \text{ elsewhere},
 \end{cases}
\end{align}
\end{definition}

The important properties of these forms are that $\varphi^N_t$ and $\psi^N_t$ are closed, and that $\psi^N_t$ is close to being the Hodge dual of $\varphi^N_t$.
That is, the $3$-form $\varphi^N_t - *_{\varphi^N_t} \psi^N_t$ satisfies the assumption of \cref{theorem:original-torsion-free-existence-theorem} and $\varphi^N_t$ can be perturbed to a torsion-free $G_2$-structure.
This yields the following theorem:

\begin{theorem}[Theorem 6.4 in \cite{Joyce2017}]
 For small $t$ there exists $\eta_t \in \Omega^2(N_t)$ such that $\tilde{\varphi}^N_t := \varphi^N_t + \d \eta_t$ is a torsion-free $G_2$-structure, and
 \begin{align}
  \|{
   \tilde{\varphi}^N_t-\varphi^N_t
  }_{L^\infty}
  \leq
  ct^{1/18}
 \end{align}
 for some constant $c>0$ independent of $t$.
\end{theorem}

\section{The Gluing Construction for Instantons}
\label{section:the-gluing-construction-for-instantons}

We now turn to constructing $G_2$-instantons on the resolutions of $Y/\<\iota\>$ explained in the previous chapter.
As is common in gluing constructions in differential geometry, we obtain this result by following the three step procedure of (1) constructing an approximate solution, (2) estimating the linearisation of the equation to be solved, (3) perturbing the approximate solution to a genuine solution.

This method was first employed in \cite{Taubes1982} for the construction of anti-self-dual connections over $4$-manifolds.
A similar but simpler proof of the same results is given in \cite[Section 7.2]{Donaldson1990}.

In \cref{subsection:the-pregluing-construction} we explain how a section $s$ of a moduli bundle gives rise to a connection $s(A)$ on the bundle of Eguchi-Hanson spaces $P$ from \cref{equation:definition-of-P}, cf. \cref{proposition:section-gives-rise-to-connection}.
If the topological compatibility condition \cref{assumption:gluing-topological-compatibility} is satisfied, we can glue $s(A)$ to a $G_2$-instanton $\theta$ on the orbifold $Y/\<\iota\>$.
The resulting connection $A_t$ is close to being a $G_2$-instanton and in \cref{subsection:pregluing-estimate} we will quantify this.
We will see that this error is small in a suitable norm if $s$ satisfies a first order partial differential equation, the Fueter equation.
\Cref{subsection:linear-estimates} is the difficult part of the analysis, where we give an estimate for the inverse of the linearised instanton operator.
In \cref{subsection:quadratic-estimate,subsection:deforming-to-genuine-solutions} we complete the argument and construct the perturbation that turns the approximate solution from before into a genuine solution to the $G_2$-instanton equation.

Throughout we will use the notation from the previous chapter.
That is, $Y$ is a $G_2$-manifold with $G_2$-involution $\iota:Y \rightarrow Y$, and $N_t$ is the resolution of $Y/\< \iota \>$.
The resolution $N_t$ is obtained by gluing in the Eguchi-Hanson bundle $P$ over the singular locus $L=\fix (\iota)$.
On $P$ we have the $G_2$-structures $\varphi^P_t$ and $\tilde{\varphi}^P_t$, and on $N_t$ we have the $G_2$-structure $\varphi^N_t$ with small torsion and the torsion-free $G_2$-structure $\tilde{\varphi}^N_t$.
In the case that $N_t$ is a resolution of $T^7/\Gamma$, we also defined the $G_2$-structures $\varphi^t$ and $\tilde{\varphi}^t$.
These two will also be denoted by $\varphi^N_t$ and $\tilde{\varphi}^N_t$ respectively and the special case of $T^7/\Gamma$ will need no special treatment most of the time.
The exception is the pre-gluing estimate for resolutions of $T^7/\Gamma$, \cref{corollary:pregluing-estimate-on-T7}, which is better than in the general case.

In the case of resolutions of $T^7/\Gamma$, our main result is \cref{theorem:instanton-existence3}, in the general case it is \cref{theorem:instanton-existence}.
We will use both theorems in \cref{section:examples} to construct new examples of $G_2$-instantons on the resolution of $T^7/\Gamma$ and the resolution of $(T^3 \times \text{K3})/\mathbb{Z}^2_2$.

\subsection{The Pregluing Construction}
\label{subsection:the-pregluing-construction}

\subsubsection{Moduli Bundles of ASD-Instantons}

Let $\pi: E_0 \rightarrow Y/\< \iota \>$ be an orbifold $G$-bundle with connection $\theta$, i.e. a $G$-bundle with connection over $Y$ together with a lift $\hat{\iota}$ of $\iota$ such that $\hat{\iota}^2=\Id$ and such that $\hat{\iota}^* \theta= \theta$.
As before, $\fix(\iota)=L$ and we now set $E_\infty = E_0|_L$, which is a $G$-bundle with $\mathbb{Z}_2$-action, and $A_\infty = \theta|_{E_\infty}$.
For each connected component of $L$ choose a framed moduli space of ASD instantons $M$ on a bundle $E$ over Eguchi-Hanson space ${\XEH}$, cf. \cref{subsubsection:gauge-theory-on-ale}.
The homomorphism $\rho: \Z_2 \rightarrow G$ used in the definition of $M$ defines a $\Z_2$ left action on $G$.
We then ask for $E_0$ and $M$ to be compatible in the following sense:

\begin{assumption}
\label{assumption:gluing-topological-compatibility}
 For all $l \in L$ there exists an isomorphism of manifolds with $G$ right action and $\Z_2$ left action $\phi: E_\infty|_l \rightarrow G$.
\end{assumption}

\begin{proposition}
 Let $G_\rho \subset G$ be the stabiliser of $\rho$ as in \cref{equation:ALE-moduli-definitions}.
 Then there exists a $G_\rho$-reduction $\check{E}$ of $E_\infty$ such that $A_\infty$ reduces to $\check{E}$.
\end{proposition}

\begin{proof}
 As before, let $\rho: \Z_2 \rightarrow G$ be the representation that defines the asymptotic limit for connections in $M$.
 Define
 \begin{align}
 \label{equation:reduction-of-orbifold-bundle}
  \check{E}
  :=
  \{
   u \in E_\infty :
   u \cdot \rho(-1)
   =  
   \hat{\iota}(u)
  \}. 
 \end{align}
 To see that this is a $G_\rho$-bundle, fix $l \in L$ and let $\phi: E_\infty|_l \rightarrow G$ be the isomorphism from \cref{assumption:gluing-topological-compatibility}.
 Then $u \in \check{E}|_l$ if and only if $\phi(u) \in G_\rho$.
 
 It remains to check that $A_\infty$ reduces to $\check{E}$.
 To this end, let $\gamma:I \rightarrow \check{E}$ be a curve.
 Then
 \begin{align}
 \label{equation:connection-reduces}
 \begin{split}
  A_\infty(\dot{\gamma}(0))
  &=
  \hat{\iota}^* A_\infty(\dot{\gamma}(0))
  \\
  &=
  A_\infty
  \left(
  \ddt (\gamma(t) \cdot \rho(-1))|_{t=0}
  \right)
  \\
  &=
  \Ad(\rho(-1))
  \left(
  A_\infty(\dot{\gamma}(0))
  \right).
 \end{split}
 \end{align}
 In the first step we used $\hat{\iota}^* \theta= \theta$.
 The second step is the defining property of $\check{E}$ from \cref{equation:reduction-of-orbifold-bundle}.
 Now, for any subgroup $H \subset G$ we define the \emph{centraliser of $H$ in $G$} as $Z(H)=\{g \in G: hgh^{-1}=g $ for all $ h \in H\}$.
 Then
 \begin{align}
 \label{equation:lie-algebra-of-centraliser}
  \Lie(Z(H))
  =
  \mathfrak{z}_H
  :=
  \{
   V \in \mathfrak{g}:
   \Ad(h)V=V
   \text{ for all }
   h \in H
  \}.
 \end{align}
 This equality holds, because for $X=\dot{g}(0) \in \Lie(Z(H))$, where $g:I \rightarrow Z(H)$ is a curve, we have that
 $\Ad(h)X=\ddt (hg(t)h^{-1})|_{t=0}=X$ by definition of $Z(H)$.
 Conversely, for $V \in \mathfrak{z}_H$, we have that $g(t):=\exp(tV)$ is a curve with $\dot{g}(0)=V$ in $Z(H)$, because
 $hg(t)h^{-1}=\exp(t \cdot \Ad(h) V)=\exp(tV)=g(t)$
 for all $h \in H$.
 
 Therefore, by \cref{equation:connection-reduces,equation:lie-algebra-of-centraliser}, we have that $A_\infty|_{\check{E}}$ takes values in $\Lie(G_\rho)$, i.e. restricts to a connection on $\check{E}$.
\end{proof}

\begin{definition}
 Define the \emph{moduli bundle}
 \begin{align}
 \label{equation:moduli-bundle}
  \mathfrak{M}
  :=
  (\Fr \times \check{E}) \times_{\U(2) \times G_\rho} M
 \end{align}
 and its \emph{vertical tangent space}
 \begin{align}
  V\mathfrak{M}
  :=
  (\Fr \times \check{E}) \times_{\U(2) \times G_\rho} TM.
 \end{align}
\end{definition}

\subsubsection{Fueter Sections and Connections on Bundles over $P$}

In the following, we will study sections $s:L \rightarrow \mathfrak{M}$.
It will turn out that such a section $s$ gives rise to a connection that is almost a $G_2$-instanton, if it satisfies a first order differential equation, the \emph{Fueter equation} (cf. \cref{definition:fueter-section}).

\begin{definition}
\label{definition:moduli-bundle-cov-derivative}
 Let $s:L \rightarrow \mathfrak{M}$ be a section.
 We define its covariant derivative $\nabla s \in \Omega^1(L, V \mathfrak{M})$ as follows:
 for $x \in L$, $X \in T_x L$ let $f \in C^\infty(\Fr)$ and $e \in C^\infty(\check{E})$ be local sections around $x$ such that $A ^{\text{LC}} \d f(x)=0$ and $A_\infty(\d e(X))=0$, where $A ^{\text{LC}}$ is the Levi-Civita connection of $Y$.
 Let $B: L \rightarrow M$ be a local section around $x$ such that $s=[(f,e),B]$.
 Then
 \[
  \nabla _X(s)
  =
  [
   (f,e), dB(X)
  ]
  \in
  (\Fr \times \check{E}) \times_{\U(2) \times G_\rho} TM.
 \]
\end{definition}

\begin{definition}
 Let $s:L \rightarrow \mathfrak{M}$ be a section.
 Fix $x \in L$ and let $e_1$, $e_2$, $e_3$ be an orthonormal basis of $T_x L$.
 The $G_2$-structure on $Y$ defines a map
 \begin{align}
 \begin{split}
  \Lambda^1 (T_x L) & \rightarrow
  \Lambda^+ P_x
  \\
  e_i & \mapsto \check{\omega}_i|_{P_x} =: \omega_i.
  \end{split}
 \end{align}
 The $\omega_i$ correspond to complex structures on $P_x$ and therefore, by \cref{theorem:moduli-hyperkaehler-structure}, to elements $I_i \in \End(V_x \mathfrak{M})$.
 We thus have a Clifford multiplication given by
 \begin{align}
  \begin{split}
   e_i \cdot : V_x \mathfrak{M} & \rightarrow V_x \mathfrak{M}
   \\
   a & \mapsto I_i (a).
  \end{split}
 \end{align}
\end{definition}

\begin{definition}
\label{definition:fueter-section}
 A section $s:L \rightarrow \mathfrak{M}$ is called a \emph{Fueter section} if
 \begin{align}
 \label{equation:fueter-equation}
  \mathfrak{F} s
  :=
  \sum_{i=1}^3
  e_i \cdot \nabla_{e_i} s
  =0
  \in \Gamma (s^* V \mathfrak{M}),
 \end{align}
 where $(e_1,e_2,e_3)$ is a local orthonormal frame.
\end{definition}

The following is an extension of \cite[Theorem 1]{Donaldson2011}:

\begin{theorem}
\label{proposition:section-gives-rise-to-connection}
 Denote by $\tilde{\mathbb{P}} \rightarrow M \times {\XEHh}$ the tautological bundle with tautological connection $\tilde{\mathbb{A}}$ over $M \times {\XEH}$ from \cref{proposition:tautological-bundle-properties} and assume that there exists a lift of the $\U(2)$-action on $M \times {\XEHh}$ to $\tilde{\mathbb{P}}$ preserving $\tilde{\mathbb{A}}$.
 Let $s \in C^\infty(\mathfrak{M})$, and denote $\hat{P} = \Fr \times_{\U(2)} {\XEHh}$.
 Then there exists a natural $G$-bundle $s(E)$ over $\hat{P}$ with connection $s(A)\in \mathscr{A}(s(E)|_P)$ together with an isomorphism of $G$-bundles with $\mathbb{Z}_2$ left action $\Phi: s(E)|_{\hat{P} \setminus P} \rightarrow E_\infty$ so that:
 
 \begin{enumerate}[label=(\roman*)]
  \item
  The pair $(s(E),s(A))|_{P_x}$ represents $s(x)$.
  That means:
  if $s(x)=[(f,e), [B]]$ for $f \in \Fr_x$, $e \in (E_0)_x$, $[B] \in M$, then under the diffeomorphism ${\XEH} \simeq P_x$, $y \mapsto [f,y]$, the $G$-bundles $s(E)|_{P_x}$ and $E$ are isomorphic, and $B$ and $s(A)$ are gauge equivalent.
  
  \item
  The map $\Phi$ identifies $A_\infty$ and $s(A)$ over the fibre at infinity, i.e. $\Phi^*A_\infty = s(A)|_{\hat{P} \setminus P}$.
  
  \item
  The connection $s(A)|_P$ is a $(\psi^P_t)^*$-instanton if and only if $s$ is a Fueter section.
  Here, $s(A)$ being a $(\psi^P_t)^*$-instanton means that $F_{s(A)} \wedge(\psi^P_t)^*=0$, where $(\psi^P_t)^*=\sum \sigma^*(e^i) \wedge \sigma^*(e^j) \wedge \check{\omega}^k$ and $\sigma:P \rightarrow L$ is the projection of the bundle $P$ (cf. \cref{equation:definition-of-P}).
 \end{enumerate}
\end{theorem}

\begin{proof}
\textbf{Construction of $s(E)$, $s(A)$, and $\Phi$:}
together with the connections $\nabla ^{\text{LC}}$ on $\Fr$ and $A_\infty$ on $\check{E}$, the connection $\tilde{\mathbb{A}}$ induces a connection $\alpha$ on the principal $G$-bundle $(\Fr \times \check{E}) \times_{\U(2) \times G_\rho} \tilde{\mathbb{P}} \rightarrow (\Fr \times \check{E}) \times_{\U(2) \times G_\rho} (M \times {\XEHh})$ via the formula
\begin{align}
\label{equation:definition-auxiliary-alpha-connection}
 \alpha ([(U, V),T])
 :=
 \tilde{\mathbb{A}}(T),
\end{align}
where $U \in T \Fr$, $V \in T \check{E}$ are horizontal vectors and $T \in T \tilde{\mathbb{P}}$.
By assumption, $\tilde{\mathbb{A}}$ is left-invariant, which makes the definition of $\alpha$ independent of the chosen representative.

Consider the map
\begin{align*}
 (s \times \Id): \hat{P} = \Fr \times_{\U(2)} {\XEHh}
 & \rightarrow
 (\Fr \times \check{E}) \times_{\U(2) \times G_\rho}
 (M \times {\XEHh})
 \\
 [f,y]
 &\mapsto
 [(f,e),(B,y)],
\end{align*}
where $s(\sigma(e))=[(f,e),B] \in \mathfrak{M}_{\pi(e)}$.
Then define
\[
s(E) := (s \times \Id)^*((\Fr \times \check{E}) \times_{\U(2) \times G_\rho} \tilde{\mathbb{P}}),
\quad
s(A) := (s \times \Id)^* \alpha
\]
and the trivialisation
$\underline{\phi}:
  \tilde{\mathbb{P}}|_{M^{\orb} \times \{ \infty \}}
  \rightarrow
  G \times M^{\orb}$
from \cref{proposition:tautological-bundle-properties} induces an isomorphism 
\begin{align}
\begin{split}
 \Phi: s(E)|_{\hat{P} \setminus P} 
 \quad\quad\quad\quad\quad\quad\quad\quad\quad\quad\quad\quad\quad\quad\quad
 &
 \\
 \simeq
 (s \times \Id|_{{\XEHh} \setminus {\XEH}})^*
 \left(
 (\Fr \times \check{E})
 \times_{\U(2) \times G_\rho}
 \tilde{\mathbb{P}}|_{M \times \{ \infty\} }
 \right)
 &\rightarrow 
 s^*
 \left(
 (\Fr \times \check{E})
 \times_{\U(2) \times G_\rho}
 G \times M
 \right)
 \\
 &
 \quad\quad
 \simeq
 \check{E} \times_{G_\rho} G
 \simeq
 E_\infty.
 \end{split}
\end{align}
The last point of \cref{proposition:tautological-bundle-properties} states that
$\underline{\phi}^*A_{\text{product}}
 =
 \tilde{\mathbb{A}}|_{M \times \{\infty\}}$
which implies that 
$\Phi^*A_\infty = s(A)|_{\hat{P} \setminus P}$.

\textbf{$s(A)$ is a $(\psi^P_t)^*$-instanton if and only if $s$ is a Fueter section:}
for easier notation, assume that the bundle $\Fr$ is trivial and $\nabla^{\text{LC}}$ is the product connection.
The proof of the general case works the same.
In this case, $L \times {\XEHh} = \hat{P}$ and $s(E) = (s \times \Id)^*(\check{E} \times_{G_\rho} \tilde{\mathbb{P}})$.
Then fix $(l,x) \in L \times {\XEHh} = \hat{P}$, an orthonormal basis $(e_1,e_2,e_3)$ of $T_lL$ and denote by $(e^1,e^2,e^3)$ its dual basis.
Around $l$, write $s(x)=[e,B]$ with the property that $\d e(V)$ is parallel for all $V \in T_l L$.
Then, for $Z \in T_x {\XEHh}$:
\begin{align}
\label{equation:mixed-curvature-computation}
\begin{split}
 F_{s(A)}(e_i, Z)
 &=
 \left( (s \times \Id)^* F_\alpha \right)
 (e_i, Z)
 \\
 &=
 F_\alpha
 \left(
  \left[ 
   \d e (e_i), (\d B(e_i), 0)
  \right],
  \left[
   \d e (e_i), (0, Z)
  \right]
 \right)
 \\
 &=
 F_{\tilde{\mathbb{A}}}
 (\d B(e_i), Z)
 \\
 &=
 \d B(e_i)(Z).
\end{split}
\end{align}
In the first step we used that the curvature of a pullback connection is the pullback of its curvature.
The third step is the definition of $\alpha$ from \cref{equation:definition-auxiliary-alpha-connection}, and in the last step we used the curvature properties of the tautological connection $\tilde{\mathbb{A}}$ from \cref{proposition:tautological-bundle-properties}.
As before, denote by $I_1, I_2, I_3$ the Hyperkähler triple of complex structures on ${\XEH}$ and $\omega_1,\omega_2,\omega_3$ the corresponding symplectic forms.
The Fueter condition from \cref{definition:fueter-section} for $s$ is equivalent to the following equation of elements in $\Omega^1({\XEH}, \Ad P)$:
\begin{align*}
 0 =
 \sum_{i=1}^3
 I_i (\d B(e_i))
 =
 \sum_{i=1}^3
 \omega_i 
 (\d B(e_i), \cdot)
 &=
 \sum_{i=1}^3
 \omega_i 
 (F_{s(A)}(e_i, \cdot), \cdot)
 \\
 &=
 *
 \left(
 \sum_{i=1}^3
 \omega_i 
 \wedge
 F_{s(A)}(e_i, \cdot)
 \right)
\end{align*}
where $*$ denotes the Hodge star on ${\XEH}$.
The first equality is the Fueter equation, the third equality is \cref{equation:mixed-curvature-computation}, and the second and fourth equality are linear algebra computations that can be computed in standard coordinates.

Applying $*$ to both sides gives
\[
 0 =
 \left(
 \sum_{i=1}^3
 \omega_i 
 \wedge
 F_{s(A)}(e_i, \cdot)
 \right)
\]
which in turn implies
\[
 0=
 \sum_{i,j,k \text{ cyclic}}
 \omega_i \wedge e^j \wedge e^k \wedge [F_{s(A)}]_{(1,1)},
\]
where $[F_{s(A)}]_{(1,1)}$ denotes the $(1,1)$-component of $F_{s(A)}$ according to the bi-grading on $\Lambda^* T^*(L \times {\XEH})$ induced by $T^* (L \times {\XEH})=T^*L \oplus T^* {\XEH}$.
On the other hand, $[F_{s(A)}]_{(0,2)} \in \Omega^2({\XEH}, \Ad P)$ is anti-self-dual by \cref{proposition:tautological-bundle-properties}, thus
\[
 0=
 \sum_{i,j,k \text{ cyclic}}
 \omega_i \wedge e^j \wedge e^k \wedge [F_{s(A)}]_{(0,2)}.
\]
Last, $0=
 \sum_{i,j,k \text{ cyclic}}
 \omega_i \wedge e^j \wedge e^k \wedge [F_{s(A)}]_{(2,0)}$,
because this is a sum of forms of type $(2,4)$ which must vanish as $L$ has dimension $3$.
\end{proof}

\subsubsection{Gluing Connections over $P$ and $Y/\<\iota\>$}

Throughout the rest of the article, we will use weighted the Hölder norms from \cite[Section 6]{Walpuski2017}:

\begin{definition}
\label{definition:hoelder-norms-for-gauge-with-cases}
For $\delta, l \in \R$, let
\begin{align}
 \begin{split}
  w_{l,\delta;t}
  : N_t & \rightarrow \R
  \\
  x & \mapsto
  \begin{cases}
   t^\delta (t+ r_t(x))^{-l-\delta},
   & \text{ if } r_t(x) \leq \sqrt{t}
   \\
   r_t^{-l+\delta}
   & \text{ if } r_t(x) > \sqrt{t}
  \end{cases}
 \end{split}
\end{align}
and by slight abuse of notation use the same symbol to denote $w_{l,\delta;t}:N_t \times N_t \rightarrow \R$ given by $w_{l,\delta;t}(x,y)=\min
  \{
   w(x),w(y)
  \}$.
 Let $U \subset N_t$.
 For $\alpha \in (0,1)$, $\beta \in \R$, $k \in \N$, and $f$ a tensor field on $N_t$ define the \emph{weighted Hölder norm of $f$} via
 \begin{align*}
  \left[
   f
  \right]
  _{C^{0,\alpha}_{l,\delta;t}(U)}
  &:=
  \sup
  _{
   \substack{
   x,y \in U, \; x \neq y \\
   d (x,y) \leq t+\min \{r_t(x),r_t(y)\}
   }
  }
  w_{l-\alpha,\delta;t}(x,y)
  \frac{\left| f(x)-f(y) \right|}{d(x,y)^\alpha},
  \\
  \|{
   f
  }_{L^\infty_{l,\delta;t}(U)}
  &:=
  \|{
   w_{l,\delta;t}
   f
  }_{L^\infty(U)},
  \\
  \|{
   f
  }_{C^{k,\alpha}_{l,\delta;t}(U)}
  &:=
  \sum_{j=0}^k
  \|{
   \nabla^j f
  }_{L^\infty_{l-j,\delta;t}(U)}
  +
  \left[
   \nabla^j f
  \right]_{C^{0,\alpha}_{l-j,\delta;t}(U)}.
 \end{align*}
 The term $f(x)-f(y)$ in the first line denotes the difference between $f(x)$ and the parallel transport of $f(y)$ to the fibre over $x$ along one of the shortest geodesics connecting $x$ and $y$.
 When $U$ is not specified, take $U=N_t$.
 We use the notation $\|{
   \cdot
  }_{C^{k,\alpha}_{l,\delta;t}(U),g}$
 for the weighted Hölder norm with respect to the metric $g$, i.e. use parallel transport with respect to the Levi-Civita connection induced by the metric $g$, and measure vectors in $g$.
If no metric $g$ is specified, we take $g=g^N_t$.
For our analysis, we need $\delta \in (-1,0)$, $\alpha \in (0,1)$, $\alpha \ll |\delta|$, for example $\delta=-1/64$, $\alpha=1/256$ will work.
\end{definition}

\begin{remark}
Note that $w_{l,\delta;t}$ is not continuous, but that does not cause any problems.
\end{remark}

\begin{proposition}[Proposition 6.2 in \cite{Walpuski2017}]
\label{proposition:norm-weights-basic-estimate}
 If $(f,g) \mapsto f \cdot g$ is a bilinear form satisfying $\left| f \cdot g \right| \leq \left| f \right| \left| g \right|$, then
 \begin{align*}
  \|{
   f \cdot g
  }_{C^{k,\alpha}_{l_1 + l_2, \delta_1 + \delta_2; t}}
  \leq
  \|{
   f
  }_{C^{k,\alpha}_{l_1, \delta_1 ; t}}
  \cdot
  \|{
   g
  }_{C^{k,\alpha}_{l_2, \delta_2; t}}.
 \end{align*}
\end{proposition}

We have shown that $s(A)$ is a $(\psi^P_t)^*$-instanton.
It is, however, in general not a $G_2$-instanton with respect to $\psi^P_t$ because of the $(2,0)$ part of its curvature.
We will later estimate the failure of $s(A)$ of being a $G_2$-instanton.

\begin{definition}
\label{definition:pullback-connection-from-L}
 For $l \in L$ choose a neighbourhood $l \in V_l \subset L$ over which $E_\infty$ is trivial.
 Use the identification $\Phi: s(E)|_{\hat{P} \setminus P} \rightarrow E_\infty$ and parallel transport with respect to $s(A)$ to get a trivialisation of $s(E)$ around $\hat{P}|_{V_l} \setminus P|_{V_l}$, say on a neighbourhood $U_l \subset \hat{P}$.
 Using this, we can view the pullback of $s(A)|_{\hat{P} \setminus P}$ under the projection $U_l \rightarrow V_l$ as a connection $\overline{A_\infty}^l \in \mathscr{A}(s(E)|_{U_l})$.
 This definition is independent of the choice of $l \in L$, and therefore defines some connection $\overline{A_\infty} \in \mathscr{A}(s(E)|_U)$, where $U \subset \hat{P}$ is a neighbourhood of the points at infinity $\hat{P} \setminus P$.
\end{definition}

Now is the first time we cite a non-trivial result from \cite{Walpuski2017}.
Therein, Fueter sections into a moduli bundle of ASD-instantons on $\R^4$ were considered, while in this chapter ASD-instantons on ${\XEH}$ are considered.
At some points this changes the analysis, and these results are reproved in this new setting in the coming sections.
At some points, results carry over without having to change the proof.
The following proposition is the first such result:

\begin{proposition}[Proposition 7.4 in \cite{Walpuski2017}]
\label{proposition:estimates-for-s(A)}
 There exists $c>0$ such that for all $t \in (0,T)$:
 \begin{align}
  \|{
   [F_{s(A)}]_{2,0}
   -
   F_{\overline{A_\infty}}
  }_{C^{0,\alpha}_{-2,0;t}(U),g^P_t}
  &\leq
  ct^2,
  \\
  \|{
   [F_{s(A)}]_{1,1}
  }_{C^{0,\alpha}_{-3,0;t}(U),g^P_t}
  &\leq
  ct^2, \text{ and }
  \\
  \|{
   [F_{s(A)}]_{0,2}
  }_{C^{0,\alpha}_{-4,0;t}(U),g^P_t}
  &\leq
  ct^2.
 \end{align}
\end{proposition}

\begin{proposition}
\label{definition:approximate-solution}
\label{proposition:approximate-solution}
 Let $E_0 \rightarrow Y / \< \iota \>$ be an orbifold bundle with connection $\theta$ satisfying \cref{assumption:gluing-topological-compatibility}, $L= \fix(\iota)$, and $s:L \rightarrow \mathfrak{M}$ be a Fueter section.
 
 Then there exists a $G$-bundle $E_t$ over $N_t$ and a connection $A_t$ on $E_t$ such that 
 \begin{align*}
  (E_t,A_t)|_{N_t \setminus \Upsilon_t(U_{t^{-1}R})} 
  &\simeq (E_0,\theta)|_{N_t \setminus \Upsilon_t(U_{t^{-1}R})}
  \quad
  \text{ and}
  \\
  (E_t,A_t)|_{\Upsilon_t(U_1)} 
  &\simeq (s(E),s(A))|_{\rho^{-1}(U_1)}.
 \end{align*}
\end{proposition}

\begin{proof}

\textbf{Construction of $E_t$:}
 By \cref{proposition:section-gives-rise-to-connection} we have a bundle isomorphism $\Phi: E_\infty \rightarrow s(E)|_{\hat{P} \setminus P}$.
 Let $U \subset \hat{P}$ be a neighbourhood of $\hat{P} \setminus P$.
 Now use radial parallel transport with respect to $\theta$ on $E_0$ and parallel transport with respect to $\overline{A_\infty}$ (the pullback of $\Phi^* A_\infty$ to a neighbourhood of $\hat{P} \setminus P$ defined in \cref{proposition:estimates-for-s(A)}) to extend $\Phi$ to the neighbourhood $\Upsilon(U_R) \subset Y$ of $L$, denote the extension by $\Psi$.
 The conditions $\hat{\iota}^* \theta=\theta$ and \cref{assumption:gluing-topological-compatibility} ensure that this is well-defined.
 
 As in \cref{subsection:g2-structures-on-the-resolution} we use the symbol $\rho$ to denote the map $\rho: P \rightarrow \nu/\{\pm 1\}$ induced by the blowup map ${\XEH} \rightarrow \C^2/\{\pm 1\}$ on Eguchi-Hanson space.
 For small enough $t$ we have that the overlap $O:=U_{t^{-1}R} \cap \rho(U)$ is non-empty. 
 Use this to define $E_t$ by gluing together $E_0$ and $s(E)$ via $\Psi$ over $O$, i.e.
 \begin{align}
  E_t
  :=
  E_0|_{Y \setminus \Upsilon_t(U_{t^{-1}R} \setminus O)}
  \cup
  s(E)|_{\rho^{-1}(U_{t^{-1}R})}
  /_\sim,
 \end{align}
 where $v \sim \Psi(v)$ for $v \in E_0|_{\Upsilon_t(O)}$. 
 
 \textbf{Construction of $A_t$:} 
 Let $\chi^-_t: N_t \rightarrow [0, 1]$ and $\chi^+_t: N_t \rightarrow [0, 1]$ be rescalings of a smooth cut-off function such that 
 \begin{align}
  \label{equation:cut-off-functions}
  \begin{split}
  \chi^-_t |_{\{r_t \leq t \} } \equiv 0
  \text{ and }
  \chi^-_t |_{\{r_t \geq 2t \} } \equiv 1,
  \\
  \chi^+_t |_{\{r_t \leq R/2 \} } \equiv 1
  \text{ and }
  \chi^+_t |_{\{r_t \leq R \} } \equiv 0.
  \end{split}
 \end{align}
 
 \begin{figure}
 \begin{center}
  \input{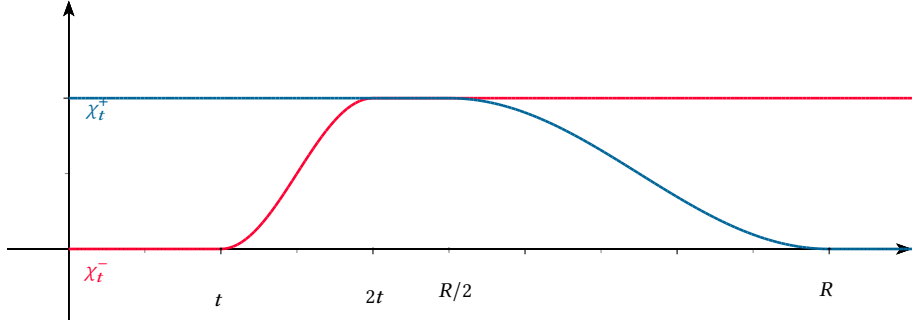}
 \end{center}
 \caption{The cut-off functions $\chi^-_t$ and $\chi^+_t$ from \cref{equation:cut-off-functions} for small $t$.}
 \end{figure}

 Similar to the definition of $\overline{A_\infty} \in \mathscr{A}\left( s(E)|_{U} \right)$, define $\underline{A_\infty} \in \mathscr{A}\left( E_0|_{\Upsilon_t \left( U_{t^{-1}R} \right) } \right)$ by pulling back $A_\infty \in \mathscr{A}(E_\infty)$.
 By definition of $E_t$, we have that $\overline{A_\infty}$ and $\underline{A_\infty}$ are both connections on $E_t$.
 The map $\Phi$ identifies $A_\infty$ and $s(A)$ by the second point of \cref{proposition:section-gives-rise-to-connection}.
 Because $\Psi$ is an extension of $\Phi$ defined by radial parallel transport, and $\overline{A_\infty}$ and $\underline{A_\infty}$ are also defined via radial parallel transport, we have that $\overline{A_\infty}=\underline{A_\infty}$ as connections on $E_t|_{\Upsilon_t(O)}$.
 
 We then have $\sigma \in \Omega^1(\Ad s(E)|_O)$ and $b \in \Omega^1(\Ad E_0|_O)$ such that
 \begin{align}
 \label{equation:definition-of-sigma-and-b}
  s(A)
  =
  \overline{A_\infty}+\sigma,
  \;\;\;\;\;\;\;
  \theta
  =
  \underline{A_\infty}+b
  \;\;\;
  \text{ over }
  O.
 \end{align}
 Define then
 \begin{align}
  A_t
  :=
  \begin{cases}
   s(A) & \text{on } r_t < t
   \\
   \underline{A_\infty}+
   \chi^-_t b+
   \chi^+_t \sigma
   & \text{on } t \leq r_t \leq R
   \\
   \theta & \text{on } r_t > R.
  \end{cases}
 \end{align} 
\end{proof}

The following proposition follows immediately from \cref{definition:hoelder-norms-for-gauge-with-cases}.

\begin{proposition}
\label{proposition:cut-off-norms}
 Let $\chi^-_t$ and $\chi^+_t$ as in \cref{equation:cut-off-functions}.
 Then there exists $c>0$ such that for all $t \in (0,T)$:
 \begin{align*}
  \|{
   \chi^-_t
  }_{C^{0,\alpha}_{0,0;t}}
  +
  \|{
   \d \chi^-_t
  }_{C^{0,\alpha}_{-1,0;t}}  
  &\leq
  c,
  \\
  \|{
   \chi^+_t
  }_{C^{0,\alpha}_{0,0;t}}
  +
  \|{
   \d \chi^+_t
  }_{C^{0,\alpha}_{0,0;t}}
  &\leq
  c.
 \end{align*}
\end{proposition}

The following proposition is proved like \cref{proposition:estimates-for-s(A)} with the proof from \cite{Walpuski2017} directly carrying over to this setting.
The estimate for $\sigma$ holds because of the fast decay of the curvature of ASD connections on ALE spaces, see \cref{proposition:moduli-independent-of-regularity}.
The estimate for $b$ holds because over $L$ we have that $\underline{A_\infty}=\theta$, not just in the $L$-direction.
That is because $\underline{A_\infty}$ is defined using parallel transport with respect to $\theta$ as in \cref{definition:pullback-connection-from-L}.

\begin{proposition}[Proposition 7.6 in \cite{Walpuski2017}]
\label{proposition:pregluing-summands-norms}
 Let $\sigma \in \Omega^1(\Ad s(E)|_O)$ and $b \in \Omega^1(\Ad E_0|_O)$ as defined in \cref{equation:definition-of-sigma-and-b}.
 Then there exists $c>0$ such that for all $t \in (0,T)$:
 \begin{align*}
  \|{
   \sigma
  }_{C^{0,\alpha}_{-3,0;t}(t \leq r_t \leq R)}
  +
  \|{
   \d _{\overline{A_\infty}}
   \sigma
  }_{C^{0,\alpha}_{-4,0;t}(t \leq r_t \leq R)}
  &\leq
  c t^2
  \text{ and}
  \\
  \|{
   b
  }_{C^{0,\alpha}_{1,0;t}(r_t \leq R)}
  +
  \|{
   \d_{\underline{A_\infty}} b
  }_{C^{0,\alpha}_{0,0;t}(r_t \leq R)}
  &\leq
  c t^2.
 \end{align*}
\end{proposition}

\subsection{Pregluing Estimate}
\label{subsection:pregluing-estimate}

The goal of this section is to derive an estimate for $F_{A_t} \wedge \tilde{\psi}^N_t$.
This is achieved in \cref{proposition:pregluing-estimate} in the general case, and in \cref{corollary:pregluing-estimate-on-T7} in the special case of resolutions of $T^7/\Gamma$.

\subsubsection{Estimates for the $G_2$-structures Involved}

We have constructed a connection $A_t$ that looks like $s(A)$ near $L$ and looks like $\theta$ far away from $L$.
The connection $s(A)$ is close to being a $G_2$-instanton with respect to $\psi^P_t$, so in order to control the pregluing error, we will need to estimate the difference $\psi^N_t-\varphi^P_t$.
This will be done in \cref{proposition:psi-n-psi-p-estimate,proposition:psi-N-psi-P-comparison-on-T7}.

On the other hand, $\theta$ is a $G_2$-instanton with respect to $\psi$, so we will need to estimate the difference $\psi^N_t-\psi$.
This will be done in \cref{proposition:psi-n-psi-estimate}.

\begin{proposition}
\label{proposition:psi-n-psi-p-estimate}
There exists $c>0$ independent of $t$ such that
\begin{align}
\label{equation:psi-n-psi-p-estimate}
  \|{
   \psi^N_t - \psi^P_t
  }_{C^{0,\alpha}_{2,0;t}(U_R)}
  \leq
  c t^{-1}.
 \end{align}
\end{proposition}

\begin{proof}
We have
\begin{align}
\label{equation:psi-n-psi-p-pointwise-estimate}
\begin{split}
& \quad | \psi^N_t - \psi^P_t |_{g^N_t}
\\ & =
\begin{cases}
 \d \,[
  t^2 \beta_{0,3}+t^4 \beta_{2,1}+
  ]
  t^2 \chi_{1,3}+t^4 \theta_{3,1}+t^4\theta_{2,2}
  &
  \text{ if } \check{r} \leq t^{-1/9}
  \\
   \d \,[
  t^2 \beta_{0,3}+t^4 \beta_{2,1}+a(t^{1/9}\check{r}) \cdot \Upsilon_* \zeta]+
  t^2 \chi_{1,3}+t^4 \theta_{3,1}+t^4\theta_{2,2}
  &
  \text{ if } t^{-1/9} \leq \check{r} \leq 2t^{-1/9}
  \\
   \d \,[
  t^2 \beta_{0,3}+t^4 \beta_{2,1}+ \Upsilon_* \zeta]+
  t^2 \chi_{1,3}+t^4 \theta_{3,1}+t^4\theta_{2,2}
  &
  \text{ if } 2t^{-1/9} \leq \check{r} \leq t^{-4/5}
  \\
  \begin{matrix}
   \d \,[(1-a(t^{4/5}\check{r}))(
  t^2 \beta_{0,3}+t^4 \beta_{2,1})+ \Upsilon_* \zeta
  ]+
  \\
  t^2 \chi_{1,3}+t^4 \theta_{3,1}+t^4\theta_{2,2}
  -
  a(t^{4/5}\check{r}) t^2 v_{1,2}
  \end{matrix}
  &
  \text{ if } t^{-4/5} \leq \check{r} \leq 2t^{-4/5}
  \\
   \d \,(\Upsilon_* \zeta)+
  t^2 \chi_{1,3}+t^4 \theta_{3,1}+t^4\theta_{2,2}
  -t^2 v_{1,2}
  &
  \text{ if } 2t^{-4/5} \leq \check{r}
\end{cases}
\\ & =
\begin{cases}
  \mathcal{O}(t)
  &
  \text{ if } \check{r} \leq t
  \\
  \mathcal{O}(t \check{r}^{-3})
  &
  \text{ if } t \leq \check{r} \leq t^{-1/9}
  \\
  \mathcal{O}(t \check{r}^{-3}+t^2 \check{r}^2)
  &
  \text{ if } t^{-1/9} \leq \check{r} \leq 2t^{-1/9}
  \\
  \mathcal{O}(t \check{r}^{-3}+t^2 \check{r}^2)
  &
  \text{ if } 2t^{-1/9} \leq \check{r} \leq t^{-4/5}
  \\
  \mathcal{O}(t^2 \check{r}^2+\check{r}^{-4})
  &
  \text{ if } t^{-4/5} \leq \check{r} \leq 2t^{-4/5}
  \\
  \mathcal{O}(t^2 \check{r}^2+\check{r}^{-4})
  &
  \text{ if } 2t^{-4/5} \leq \check{r},
\end{cases}
\end{split}
\end{align}
where we used \cref{proposition:v-estimates,theorem:jk-correction-theorem,proposition:jk-normal-bundle-choices} in the second step.
Multiplying with the weight function $(t+r_t)^{-2}$ gives the estimate for the $L^\infty_{2,0;t}$-norm, and the estimate for the $C^{0,\alpha}_{2,0;t}$-norm is proved analogously.
\end{proof}

\begin{proposition}
\label{proposition:psi-N-psi-P-comparison-on-T7}
Let $N_t$ be the resolution of $T^7/\Gamma$ from \cref{section:torsion-free-structures-on-the-generalised-kummer-construction}.
There exists $c>0$ independent of $t$ such that
\begin{align}
\label{equation:psi-n-psi-p-estimate-on-T7}
  \|{
   \psi^N_t - \psi^P_t
  }_{C^{0,\alpha}_{2,0;t}(U_R)}
  \leq
  ct^4.
 \end{align}
\end{proposition}

\begin{proof}
This is a restatement of \cref{equation:pregluing-estimate-G2-structure-T7}.
In the case that $N_t$ is the resolution of $T^7/\Gamma$ we have that $\psi^P_t$ is closed, so the forms $t^2 \chi_{1,3}$, $t^4 \theta_{3,1}$, $t^4\theta_{2,2}$ from \cref{proposition:v-estimates} can be chosen to be $0$.
Furthermore, in this case $\tilde{\psi}^\nu_t=\Upsilon_t^*(* \varphi)$, so $\zeta=0$.
Using this and that the cut-off happens where $\zeta t^{-1}/2 \leq \check{r} \leq \zeta t^{-1}$, the same proof as for \cref{equation:psi-n-psi-p-estimate} shows the claim.
\end{proof}

The following estimate holds in general, not just for resolutions of $T^7/\Gamma$:

\begin{proposition}
\label{proposition:psi-n-psi-estimate}
There exists $c>0$ independent of $t$ such that
\begin{align}
\label{equation:psi-n-psi-estimate}
  \|{
   \psi^N_t - \psi
  }_{C^{0,\alpha}_{-2,0;t}(\{x \in N_t: \check{r}(x) \geq 1 \})}
  \leq
  c t^2.
 \end{align}
\end{proposition}

\begin{proof}
Using \cref{proposition:v-estimates,theorem:jk-correction-theorem,proposition:jk-normal-bundle-choices}, the proof is analogous to \cref{proposition:psi-n-psi-p-estimate}.
\end{proof}

Last we need an estimate comparing $\tilde{\psi}^N_t$ and $\psi^N_t$ in a Hölder norm.
In \cref{theorem:jk-correction-theorem} we had this estimate for the $L^\infty$-norm, but not for the $C^{0,\alpha}_{0,0;t}$-norm.
Going through the proof of \ref{theorem:original-torsion-free-existence-theorem}, one can improve this to a $C^{0,\alpha}_{0,0;t}$-estimate as stated in the following proposition.
For the case of resolutions of $T^7/\Gamma$, this was done in \cite[Proposition 4.20]{Walpuski2013a}, and the proof carries over to resolutions of $Y/\<\iota\>$.

\begin{proposition}
\label{proposition:equation:torsion-free-hoelder-difference}
There exists $c>0$ independent of $t$ such that
 \begin{align}
 \label{equation:torsion-free-hoelder-difference}
 \|{
   \tilde{\psi}^N_t-\psi^N_t
 }_{C^{0,\alpha}_{0,0;t}}
 \leq
 ct^{1/18}. 
 \end{align}
\end{proposition}

\subsubsection{Principal Bundle Curvature Estimates}

For our pregluing estimate we will want to estimate $* (F_{A_t} \wedge \tilde{\psi}^N_t)$.
This is done in \cref{proposition:pregluing-estimate,corollary:pregluing-estimate-on-T7}.
Most of the heavy lifting is done by the following \cref{proposition:pregluing-estimate-without-tilde}:
here we get an estimate for $* (F_{A_t} \wedge \psi^N_t)$ which then is combined with the estimate for $\tilde{\psi}^N_t-\psi^N_t$.

\begin{proposition}
\label{proposition:pregluing-estimate-without-tilde}
 There exists $c>0$ such that for all $t \in (0,T)$ we have
 \begin{align}
  \|{
   * (F_{A_t} \wedge \psi^N_t)
  }_{C^{0,\alpha}_{-2,0;t}}
  \leq
  c t.
 \end{align}
\end{proposition}

\begin{proof}
 We will estimate $* (F_{A_t} \wedge \psi^N_t)$ separately on some regions:
 
 \begin{enumerate}
  \item 
  On $r_t \leq 2t$ we have
  \begin{align*}
   F_{A_t}
   &=
   F_{s(A)}
   +
   \chi^-_t \d_{A_\infty} b
   +
   \chi^-_t [\sigma, b]
   +
   \frac{1}{2}
   (\chi^-_t)^2 [b, b]
   +
   \d \chi^-_t \wedge b.
  \end{align*}
  Thus by \cref{proposition:norm-weights-basic-estimate}, \cref{proposition:cut-off-norms}, and \cref{proposition:pregluing-summands-norms}:
  \begin{align}
  \label{equation:pregluing-first-region-first-estimate}
  \begin{split}
   &
   \|{
    F_{A_t}-F_{s(A)}
   }_{C^{0,\alpha}_{-2,0;t}(r_t \leq 2 t)}
   \\
   &\;\;\; \leq
   \|{
    1
   }_{C^{0,\alpha}_{-2,0;t}(r_t \leq 2 t)}
   \|{
    \chi^-_t
   }_{C^{0,\alpha}_{0,0;t}(r_t \leq 2 t)}
   \|{
    \d_{A_\infty} b
   }_{C^{0,\alpha}_{0,0;t}(r_t \leq 2 t)}   
   \\
   &\;\;\;\;\;\;
   +
   \|{
    \chi^-_t
   }_{C^{0,\alpha}_{0,0;t}(r_t \leq 2 t)}
   \|{
    \sigma
   }_{C^{0,\alpha}_{-3,0;t}(r_t \leq 2 t)}
   \|{
    b
   }_{C^{0,\alpha}_{1,0;t}(r_t \leq 2 t)}   
   \\
   &\;\;\;\;\;\;
   +
   \frac{1}{2}
   \|{
    1
   }_{C^{0,\alpha}_{-3,0;t}(r_t \leq 2 t)}
   \|{
    \chi^-_t
   }_{C^{0,\alpha}_{0,0;t}(r_t \leq 2 t)}^2
   \|{
    b
   }_{C^{0,\alpha}_{1,0;t}(r_t \leq 2 t)}^2
   \\
   &\;\;\;\;\;\;
   +
   \|{
    1
   }_{C^{0,\alpha}_{-2,0;t}(r_t \leq 2 t)}
   \|{
    \d \chi^-_t
   }_{C^{0,\alpha}_{-1,0;t}(r_t \leq 2 t)}
   \|{
    b
   }_{C^{0,\alpha}_{1,0;t}(r_t \leq 2 t)}
   \\
   &\;\;\; \leq
   ct^2
   \end{split}
  \end{align}
  where we also used the fact that
  $\|{
    1
   }_{C^{0,\alpha}_{-l,0;t}(r_t \leq 2 t)}
  \leq c t^l$ if $l>0$,
  which follows from \cref{definition:hoelder-norms-for-gauge-with-cases} using $r_t \leq 2 t$.
  
 Remember that $[F_{s(A)}]_{2,0} \wedge \psi^P_t = 0$ by the ASD condition and $[F_{s(A)}]_{1,1} \wedge \psi^P_t = 0$ by the Fueter condition (cf. \cref{proposition:section-gives-rise-to-connection}).
 By \cref{proposition:estimates-for-s(A)}, we therefore have:
 \begin{align}
 \label{equation:pregluing-estimate-first-region-second-estimate}
 \begin{split}
  &\;\;\;\;\;\;
  \|{
   F_{s(A)} \wedge \psi^P_t
  }_{C^{0,\alpha}_{-2,0;t}(r_t \leq 2 t)}
  \\
  &\leq
  \|{
   [F_{s(A)}]_{(0,2)} \wedge \psi^P_t
  }_{C^{0,\alpha}_{-2,0;t}(r_t \leq 2 t)}
  \\
  &\leq
  \|{
   [F_{s(A)}-F_{\theta|_L}]_{(0,2)}
  }_{C^{0,\alpha}_{-2,0;t}(r_t \leq 2 t)}
  \cdot
  \|{
   \psi^P_t
  }_{C^{0,\alpha}_{0,0;t}(r_t \leq 2 t)}
  +
  \\
  &
  \;\;\;\;\;\;
  \|{
   F_{\theta|_L}
  }_{C^{0,\alpha}_{0,0;t}(r_t \leq 2 t)}
  \cdot
  \|{
   \psi^P_t
  }_{C^{0,\alpha}_{0,0;t}(r_t \leq 2 t)}
  \cdot
  \|{
   1
  }_{C^{0,\alpha}_{-2,0;t}(r_t \leq 2 t)}
  \\
  &\leq
  c t^2,
 \end{split}
 \end{align}
 where we again used \cref{proposition:norm-weights-basic-estimate}.
 Last, note that by \cref{proposition:estimates-for-s(A),equation:pregluing-first-region-first-estimate} we have
 $\|{
   F_{A_t}
  }_{C^{0,\alpha}_{-4,0;t}(r_t \leq 2 t)}
 \leq ct^2$.
 Thus, by
 \cref{proposition:norm-weights-basic-estimate,equation:psi-n-psi-p-estimate}:
 \begin{align}
 \label{equation:pregluing-first-region-psi-differences}
 \begin{split}
  \|{
   F_{A_t} \wedge (\psi^N_t - \psi^P_t)
  }_{C^{0,\alpha}_{-2,0;t}(r_t \leq 2 t)}
  &\leq
  \|{
   F_{A_t}
  }_{C^{0,\alpha}_{-4,0;t}(r_t \leq 2 t)}
  \|{
   \psi^N_t - \psi^P_t
  }_{C^{0,\alpha}_{2,0;t}(r_t \leq 2 t)}
  \\
  &\leq
  c t.
  \end{split}
 \end{align}
 Putting the estimates from \cref{equation:pregluing-first-region-first-estimate,equation:pregluing-estimate-first-region-second-estimate,equation:pregluing-first-region-psi-differences} together, we get
 \begin{align*}
 &\;\;\;\;\;\;
  \|{
   * (F_{A_t} \wedge \psi^N_t)
  }_{C^{0,\alpha}_{-2,0;t}(r_t \leq 2 t)}
  \\
  &\;\;\;
  \leq
  \|{
   F_{s(A)} \wedge \psi^P_t)
  }_{C^{0,\alpha}_{-2,0;t}(r_t \leq 2 t)}
  +
  \|{
   (F_{s(A)}-F_{A_t}) \wedge \psi^P_t
  }_{C^{0,\alpha}_{-2,0;t}(r_t \leq 2 t)}
  \\
  &\;\;\;
  \quad
  +
  \|{
   F_{A_t} \wedge (\psi^N_t - \psi^P_t)
  }_{C^{0,\alpha}_{-2,0;t}(r_t \leq 2 t)}
  \\
  &\;\;\; \leq
  c (t^2+t^2+t)
  \leq c t.
 \end{align*}

 \item
 On $2t \leq r_t \leq R/2$ 
 we have $A_t = A_\infty + \sigma + b$ and therefore
 \begin{align}
 \label{equation:pregluing-curvature-on-second-region}
  F_{A_t}
  =
  F_\theta + [\sigma, b]
  +
  F_{s(A)}
  -
  F_{A_\infty}.
 \end{align}
 First,
 \begin{align}
 \label{equation:pregluing-second-region-fueter-consequence}
 \begin{split}
  &\;\;\;
  \|{
   (F_{s(A)} - F_{A_\infty}) \wedge \psi^P_t
  }_{C^{0,\alpha}_{-2,0;t}(2t \leq r_t \leq R/2)}
  \\
  &\leq
  \|{
   \left[ F_{s(A)} - F_{A_\infty} \right]_{2,0} \wedge \psi^P_t
  }_{C^{0,\alpha}_{-2,0;t}(2t \leq r_t \leq R/2)}
  \\
  &\leq
  \|{
   \left[ F_{s(A)} - F_{A_\infty} \right]_{2,0}
  }_{C^{0,\alpha}_{-2,0;t}(2t \leq r_t \leq R/2)}
  \|{
   \psi^P_t
  }_{C^{0,\alpha}_{0,0;t}(2t \leq r_t \leq R/2)}
  \\
  &\leq
  ct^2,
  \end{split}
 \end{align}
 where we used point (ii) of \cref{proposition:section-gives-rise-to-connection} in the first step and \cref{proposition:estimates-for-s(A)} in the last step.
 We also have
 \begin{align}
 \label{equation:pregluing-estimate-second-region-psi-N-minus-psi-P-curvature}
 \begin{split}
  &\;\;\;
  \|{
   (F_{s(A)} - F_{A_\infty}) \wedge (\psi^N_t-\psi^P_t)
  }_{C^{0,\alpha}_{-2,0;t}(2t \leq r_t \leq R/2)}
  \\
  &\leq
  \|{
   (F_{s(A)} - F_{A_\infty})
  }_{C^{0,\alpha}_{-4,0;t}(2t \leq r_t \leq R/2)}
  \|{
   \psi^N_t-\psi^P_t
  }_{C^{0,\alpha}_{2,0;t}(2t \leq r_t \leq R/2)}
  \\
  &\leq
  ct
 \end{split}
 \end{align}
 where we used \cref{proposition:estimates-for-s(A),equation:psi-n-psi-p-estimate}, therefore
 \begin{align}
 \begin{split}
 \label{equation:pregluing-second-region-third-estimate}
  &\;\;\;
  \|{
   (F_{s(A)} - F_{A_\infty}) \wedge \psi^N_t
  }_{C^{0,\alpha}_{-2,0;t}(2t \leq r_t \leq R/2)}
  \\
  &\leq
  \|{
   (F_{s(A)} - F_{A_\infty}) \wedge \psi^P_t
  }_{C^{0,\alpha}_{-2,0;t}(2t \leq r_t \leq R/2)}
  \\
  &\;\;\;
  +
  \|{
   (F_{s(A)} - F_{A_\infty}) \wedge (\psi^N_t- \psi^P_t)
  }_{C^{0,\alpha}_{-2,0;t}(2t \leq r_t \leq R/2)}
  \\
  &\leq
  ct.
 \end{split}
 \end{align}
 Second,
 \begin{align}
 \label{equation:pregluing-second-region-sigma-b-estimate}
 \begin{split}
  & \quad
  \|{
   [\sigma, b] \wedge \psi^N_t
  }_{C^{0,\alpha}_{-2,0;t}(2t \leq r_t \leq R/2)}
  \\
  &
  \leq
  c
  \|{
   \sigma
  }_{C^{0,\alpha}_{-3,0;t}(2t \leq r_t \leq R/2)}
  \|{
   b 
  }_{C^{0,\alpha}_{1,0;t}(2t \leq r_t \leq R/2)}
  \|{
   \psi^N_t
  }_{C^{0,\alpha}_{0,0;t}(2t \leq r_t \leq R/2)}
  \\
  &
  \leq
  ct^4
  \end{split}
 \end{align}
 by \cref{proposition:pregluing-summands-norms}.
 
 Third,
 \begin{align}
 \label{equation:pregluing-second-region-f-theta-estimate}
 \begin{split}
 & \quad
  \|{
   F_{\theta} \wedge \psi^N_t
  }_{C^{0,\alpha}_{-2,0;t}(2t \leq r_t \leq R/2)}
  \\
  &\leq
  \|{
   F_{\theta} \wedge \psi
  }_{C^{0,\alpha}_{-2,0;t}(2t \leq r_t \leq R/2)}
  \\
  &\quad +
  \|{
   F_{\theta}
  }_{C^{0,\alpha}_{0,0;t}(2t \leq r_t \leq R/2)}
  \|{
   \psi^N_t -\psi
  }_{C^{0,\alpha}_{-2,0;t}(2t \leq r_t \leq R/2)}
  \\
  &\leq
  ct^2
 \end{split}
 \end{align}
 where we used the fact that $\theta$ is a $G_2$-instanton with respect to $\psi$ as well as \cref{equation:psi-n-psi-estimate} in the second step.
 So, altogether
 \begin{align*}
  \|{
   * (F_{A_t} \wedge \psi^N_t)
  }_{C^{0,\alpha}_{-2,0;t}(2t \leq r_t \leq R/2)}
  &\leq
  \|{
   F_\theta \wedge \psi^N_t
  }_{C^{0,\alpha}_{-2,0;t}(2t \leq r_t \leq R/2)}
  \\
  & \quad +
  \|{
   [\sigma, b] \wedge \psi^N_t
  }_{C^{0,\alpha}_{-2,0;t}(2t \leq r_t \leq R/2)}
  \\
  & \quad +
  \|{
   (F_{s(A)} - F_{A_\infty}) \wedge \psi^N_t
  }_{C^{0,\alpha}_{-2,0;t}(2t \leq r_t \leq R/2)}
  \\
  & \leq
  c t
 \end{align*}
 by combining \cref{equation:pregluing-curvature-on-second-region,equation:pregluing-second-region-third-estimate,equation:pregluing-second-region-sigma-b-estimate,equation:pregluing-second-region-f-theta-estimate}.
 
 \item
 On $R/2 \leq r_t \leq R$ we have $A_t= \theta + \chi^+_t \sigma$ and therefore
 \begin{align*}
  F_{A_t}
  =
  F_\theta + \chi^+_t \d_\theta \sigma
  +
  \frac{1}{2}
  (\chi^+_t)^2[ \sigma, \sigma]
  +
  \d \chi^+_t \wedge \sigma.
 \end{align*}
 Therefore, we find that
 \begin{align*}
  \begin{split}
   \|{
   F_{A_t}- F_\theta
  }_{C^{0,\alpha}_{-2,0;t}(R/2 \leq r_t)}
  &\leq
  \|{
   \chi^+_t
  }_{C^{0,\alpha}_{0,0;t}(R/2 \leq r_t)}
  \|{
   \d_\theta \sigma
  }_{C^{0,\alpha}_{-4,0;t}(R/2 \leq r_t)}
  \|{
   1
  }_{C^{0,\alpha}_{2,0;t}(R/2 \leq r_t)}
  \\
  &\;\;\;
  +
  \frac{1}{2}  
  \|{
   \chi^+_t
  }_{C^{0,\alpha}_{0,0;t}(R/2 \leq r_t)}^2
  \|{
   \sigma
  }_{C^{0,\alpha}_{-3,0;t}(R/2 \leq r_t)}^2
  \|{
   1
  }_{C^{0,\alpha}_{4,0;t}(R/2 \leq r_t)}
  \\
  &\;\;\;
  +
  \|{
   \d \chi^+_t
  }_{C^{0,\alpha}_{0,0;t}(R/2 \leq r_t)}
  \|{
   \sigma
  }_{C^{0,\alpha}_{-3,0;t}(R/2 \leq r_t)}
  \|{
   1
  }_{C^{0,\alpha}_{1,0;t}(R/2 \leq r_t)}
  \\
  &\leq
  c
  t^2
  \end{split}
 \end{align*}
 where we used \cref{proposition:norm-weights-basic-estimate,proposition:cut-off-norms,proposition:pregluing-summands-norms} in the second step.
 Using this, we see
 \begin{align*}
  \|{
   F_{A_t} \wedge \psi^N_t
  }_{C^{0,\alpha}_{-2,0;t}(R/2 \leq r_t)}
  &\leq
  \|{
   (F_{A_t}- F_\theta) \wedge \psi^N_t
  }_{C^{0,\alpha}_{-2,0;t}(R/2 \leq r_t)}
  \\
  &\;\;\;
  +
   \|{
   F_\theta \wedge \psi^N_t
  }_{C^{0,\alpha}_{-2,0;t}(R/2 \leq r_t)}
  \\
  &\leq
  ct^2,
 \end{align*}
 where we used the fact that $\psi^N_t=\psi$ where $r_t \geq R/2$ and that $\theta$ is a $G_2$-instanton with respect to $\psi$.
 \end{enumerate}
 We have that $F_{A_t} \wedge \psi^N_t=0$ outside the three considered regions, which proves the claim.
\end{proof}

\begin{corollary}
\label{proposition:pregluing-estimate}
 There exists $c>0$ such that
 \begin{align}
  \|{
   * (F_{A_t} \wedge \tilde{\psi}^N_t)
  }_{C^{0,\alpha}_{-2,0;t}}
  \leq
  c t^{1/18}.
 \end{align}
\end{corollary}

\begin{proof}
 First, observe that
 \begin{align}
 \label{equation:pregluing-corollary-curvature-boundedness}
  \|{
   F_{A_t}
  }_{C^{0,\alpha}_{-2,0;t}}
  \leq
  c.
 \end{align}
 This follows from estimating $F_{A_t}$ separately on the three regions from the proof of \cref{proposition:pregluing-estimate-without-tilde}.
 Then
 \begin{align*}
  \|{
   * (F_{A_t} \wedge \tilde{\psi}^N_t)
  }_{C^{0,\alpha}_{-2,0;t}}
  &\leq
  \|{
   * (F_{A_t} \wedge \psi^N_t)
  }_{C^{0,\alpha}_{-2,0;t}}
  +
  \|{
   * (F_{A_t} \wedge (\tilde{\psi}^N_t-\psi^N_t))
  }_{C^{0,\alpha}_{-2,0;t}}
  \\
  &\leq
  \|{
   * (F_{A_t} \wedge \psi^N_t)
  }_{C^{0,\alpha}_{-2,0;t}}
  +
  \|{
   F_{A_t}
  }_{C^{0,\alpha}_{-2,0;t}}
  \|{
   \tilde{\psi}^N_t-\psi^N_t
  }_{C^{0,\alpha}_{0,0;t}}
  \\
  &\leq
  c(t+t^{1/18})
  \leq
  ct^{1/18}
 \end{align*}
 where we used \cref{proposition:pregluing-estimate-without-tilde} to estimate the first summand in the last step, and \cref{equation:pregluing-corollary-curvature-boundedness,equation:torsion-free-hoelder-difference} to estimate the second summand in the last step.
\end{proof}

As promised, we now turn to the special case of resolutions of $T^7/\Gamma$, rather than general $G_2$-orbifolds.
We get a better pregluing estimate here, which is due to the following two facts:
first, we get a better estimate for $* (F_{A_t} \wedge \psi^N_t)$ on the resolution of $T^7/\Gamma$, because near the associative, $A_t$ is close to $s(A)$, which is close to being a $G_2$-instanton with respect to $\psi^P_t$, and \cref{proposition:psi-N-psi-P-comparison-on-T7} says that $\psi^N_t-\psi^P_t$ is small.
Second, the difference $\tilde{\psi}^N_t-\psi^N_t$ is smaller on resolutions of $T^7/\Gamma$ than in the general case.

\begin{corollary}
\label{corollary:pregluing-estimate-on-T7}
 Let $N_t$ be the resolution of $T^7/\Gamma$ from \cref{corollary:kummer-construction-simplified-torsion-free-estimate}.
 Then there exists $c>0$ such that for all $t \in (0,T)$ we have
 \begin{align}
 \label{equation:pregluing-estimate-on-T7}
  \|{
   * (F_{A_t} \wedge \tilde{\psi}^N_t)
  }_{C^{0,\alpha}_{-2,0;t}}
  \leq
  c t^2.
 \end{align}
\end{corollary}

\begin{proof}
We first prove
 \begin{align}
 \label{equation:pregluing-estimate-on-T7-without-tilde}
  \|{
   * (F_{A_t} \wedge \psi^N_t)
  }_{C^{0,\alpha}_{-2,0;t}}
  \leq
  c t^2.
 \end{align}
as in \cref{proposition:pregluing-estimate-without-tilde}, the only difference being that \cref{equation:psi-n-psi-p-estimate-on-T7} in \cref{equation:pregluing-first-region-psi-differences,equation:pregluing-estimate-second-region-psi-N-minus-psi-P-curvature} gives a factor of $t^2$ rather than $t$, yielding \cref{equation:pregluing-estimate-on-T7-without-tilde}.
For small enough $\alpha \in (0,1)$ we have that
\begin{align}
\label{equation:torsion-free-perturbation-size-on-T7}
 \|{
  \tilde{\psi}^N_t-\psi^N_t
 }_{C^{0,\alpha}_{0,0;t}}
 \leq
 ct^{5/2}
\end{align}
by \cref{corollary:kummer-construction-simplified-torsion-free-estimate}.
Taking \cref{equation:pregluing-estimate-on-T7-without-tilde,equation:torsion-free-perturbation-size-on-T7} together gives \cref{equation:pregluing-estimate-on-T7} as in the proof of \cref{proposition:pregluing-estimate}.
\end{proof}

\subsection{Linear Estimates}
\label{subsection:linear-estimates}

We now arrived in the second step of the three step process of (1) constructing an approximate solution, (2) estimating the linearisation of the instanton equation, and (3) perturbing the approximate solution to a genuine solution.
The estimate in question is \cref{proposition:inverse-operator-estimate-xy-norm}.
It makes use of the norms
$\|{
   \cdot
  }_{\mathfrak{X}_t}$
and 
$\|{
   \cdot
  }_{\mathfrak{Y}_t}$
that are defined in \cref{subsubsection:stating-the-estimate}.

The idea of the proof is this:
near the resolution locus of the associative $L$, the linearisation of the instanton equation is approximately equal to the linearisation of the Fueter equation.
Deformations of the approximate solution and deformations of the Fueter section live in different spaces, so some work will need to go into making this statement precise.

Over the course of \cref{subsubsection:the-model-operator,subsubsection:schauder-estimate,subsubsection:estimate-of-eta-a} we work out an estimate for the linearised operator modulo deformations of the approximate instanton that come from deformations of the Fueter section.
This estimate is given in \cref{proposition:crucial-proposition}.
We use a Schauder estimate for the linearised operator, which is given in section \cref{subsubsection:schauder-estimate}, together with analysis on the local models $\R^3 \times {\XEH}$ and $\R^3 \times \C^2/\{ \pm 1\}$, which is explained in \cref{subsubsection:the-model-operator}.

So we have estimates for the linearised operator on instanton deformations that come from deformations of the Fueter section from \cref{subsubsection:comparison-with-the-fueter-operator} and on the other instanton deformations from \cref{subsubsection:estimate-of-eta-a}.
In \cref{subsubsection:cross-term-estimates,subsubsection:proof-of-linear-estimate-proposition} we combine both and complete the proof of \cref{proposition:crucial-proposition}.

\subsubsection{Stating the Estimate}
\label{subsubsection:stating-the-estimate}

In the previous section, we constructed a connection $A_t \in \mathscr{A}(E_t)$.
The linearisation of the $G_2$-instanton equation together with the Coulomb gauge condition is
 \begin{align*}
  L_t:=L_{A_t}:
  (\Omega^0 \oplus \Omega^1)(M,\Ad E)
  & \rightarrow 
  (\Omega^0 \oplus \Omega^1)(M,\Ad E)
  \\
  \begin{pmatrix}
   \xi \\ a
  \end{pmatrix}
  & \mapsto
  \begin{pmatrix}
   0 & \d^*_{A_t} \\
   \d_{A_t} & *(\tilde{\psi}^N_t \wedge \d_{A_t})
  \end{pmatrix}
  \begin{pmatrix}
   \xi \\ a
  \end{pmatrix},
 \end{align*}
cf. \cref{eqution:g2-instanton-equation-made-elliptic}.
We introduce the following notation for the constant part and the quadratic part of the $G_2$-instanton equation:
for $\underline{a}=(\xi,a) \in (\Omega^0 \oplus \Omega^1)(N_t, \Ad E_t)$ define $e_t$ as well as $Q_t(\underline{a}) \in \Omega^0(N_t, \Ad E_t)$ via
\begin{align}
\label{equation:g2-instanton-equation}
\begin{split}
 &*(F_{A_t+a} \wedge \tilde{\psi}^N_t)+\d _{A_t+a}\xi
 \\
 &\;\;\;\;\;\;\;\;\;\;\;\;\;\;\;\;\;\;\;\;=
 \underbrace{
 *(F_{A_t} \wedge \tilde{\psi}^N_t)
 }_{=:e_t}
 +
 *(\d _{A_t} a \wedge \tilde{\psi}^N_t) + \d_{A_t} \xi
 +
 \underbrace{
 \frac{1}{2}
 *([a \wedge a] \wedge \tilde{\psi}^N_t)+[\xi,a]
 }_{=:Q_t(\underline{a})}.
 \end{split}
\end{align}
In this section we will study the operator $L_t$ and derive an estimate for the operator norm of the inverse of $L_t$.
This operator norm will be taken with respect to the complicated norms $\|{ \cdot}_{\mathfrak{X}}$ and $\|{ \cdot}_{\mathfrak{Y}}$, taken from \cite[Section 8]{Walpuski2017}, which we will explain now.

We need a way to decompose elements in $\Omega^1(N_t,\Ad E_t)$ into a part coming from a section of $s^*(V \mathfrak{M})$, which is nonzero only near the gluing area, and a rest:

\begin{definition}
\label{definition:pi-iota-splitting}
The section $s$ gives rise to a connection $s(A) \in \mathscr{A}(s(E))$ by \cref{proposition:section-gives-rise-to-connection}.
 A section $f \in \Gamma(s^*V \mathfrak{M})$ analogously defines an element in $T_{s(A)} \mathscr{A}(s(E))=\Omega^1(P, \Ad s(E))$, say $i_*f$.
 Use this to define
 \begin{align}
  \begin{split}
   \iota_t: \Gamma(s^* V \mathfrak{M}) & \rightarrow \Omega^1(N_t,\mathfrak{g}_{E_t})
   \\
   f & \mapsto \chi^+_t \cdot i_*f.
  \end{split}
 \end{align}
 Further define
 $\pi_t
   :
   \Omega^1(N_t,\Ad E_t) \rightarrow \Gamma(s^* V \mathfrak{M})$
 for $a \in \Omega^1(N_t,\Ad E_t)$ and $x \in L$ by
 \begin{align}
  (\pi_t a )(x)
  :=
  \sum_{\kappa}
  \int_{P_x}
  \< a, \iota_t \kappa \>_{g^P_t} 
  \vol_{g^P_t|_{P_x}}
  \cdot \kappa,
 \end{align}
 where $\kappa$ runs through an orthonormal basis of $(V\mathfrak{M})_{s(x)}$ with respect to the inner product $\< \iota_t \cdot, \iota_t \cdot \>_{g_P^t}$.
 Here the integral is taken with respect to the metric induced by $\varphi^P_t$ restricted to $P_x$.
 Let further $\overline{\pi}_t:= \iota_t \pi_t$ and
 $
  \eta_t
  :=
  \Id - \overline{\pi}_t
 $.
\end{definition}

The following proposition states that $\iota_t$ and $\pi_t$ are bounded operators.
The proof of these estimates is similar to the proof of \cite[Proposition 6.4]{Walpuski2017}.

\begin{proposition}
\label{proposition:iota-pi-bounds}
 For $l \leq -1$ and $\delta \in \R$ such that $l-\alpha +\delta > -3$ and $l+\delta < -1$ there is a constant $c >0$ such that for all $t \in (0,T)$ we have
 \begin{align*}
  \|{
   \iota_t f
  }_{C^{0,\alpha}_{l,\delta;t}}
  &\leq
  c t^{-1-l}
  \|{
   f
  }_{C^{0,\alpha}}
  \text{ and}
  \\
  \|{
   \pi_t a
  }_{C^{0,\alpha}}
  &\leq
  c t^{1+l-\alpha}
  \|{
   a
  }_{C^{0,\alpha}_{l,\delta;t}(V_{[0,R),t})}.
 \end{align*}
\end{proposition}

\begin{proof}
 The proof of the first inequality is the same as the proof of \cite[Proposition 6.4]{Walpuski2017}.
 
 To prove the second inequality, note that by \cref{proposition:ale-asd-kernel-cokernel} we have for $x \in L, \kappa \in (V \mathfrak{M})_{s(x)}$
 \[
  |i_* \kappa|_{g^P_1}
  \leq
  c_{\kappa} (1+\check{r})^{-3}
 \]
 for a constant $c_{\kappa}$ depending on $x \in L$ and on $\kappa$.
 Because $(V \mathfrak{M})_{s(x)}$ is a finite-dimensional vector space we can take $c=\max_{\|{\kappa}_{L^2,g^P_1}=1}c_\kappa$ to get the estimate 
 \begin{align}
 \label{equation:decay-of-kappa-in-g-P-1}
  |i_* \kappa|_{g^P_1}
  \leq
  c (1+\check{r})^{-3}
  \|{
   \kappa
  }_{g^P_1,L^2}
 \end{align}
 for a constant $c$ independent of $\kappa$.
 By compactness of $L$, we can assume $c$ to also be independent of $x \in L$.
 By measuring in $g^P_t$ instead of $g^P_1$ we get from \cref{equation:decay-of-kappa-in-g-P-1}:
 \begin{align}
 \label{equation:decay-of-kappa-in-g-P-t}
  |i_* \kappa|_{g^P_t}
  =
  t^{-1}
  |i_* \kappa|_{g^P_1}
  \leq
  c t^2(t+t\check{r})^{-3}
  \|{
   \kappa
  }_{g^P_1,L^2}.
 \end{align}
 For some interval $J \subset \R$ and $x \in L$ we denote
  $P_{x,J}
  :=
  \left\{
   u \in P_x
   :
   \check{r}(u) \in J
  \right\}$
 and similarly for $(\nu /\{ \pm1 \})_{x,J}$.
 By abuse of notation we write $\vol_{g^P_t}$ for $\vol_{g^P_t|_{P_x}} \in \Omega^4(P_x)$ and similarly for $\vol_{g^\nu_t}$.
 \begin{align}
 \label{equation:iota-integral-estimate}
 \begin{split}
  \int_{P_x}
  \< a, \chi_t^+ \cdot i_* \kappa \>_{g^P_t}
  \vol_{g^P_t}
  &\leq
  \int_{P_x}
  |a|_{g^P_t}
  |\chi_t^+ \cdot i_* \kappa|_{g^P_t}
  \vol_{g^P_t}
  \\
  &\leq
  c
  \int_{P_{x,\left[ 0, 1 \right]}}
  \frac{t^2}{(t+t\check{r})^3}
  w_{l,\delta;t}^{-1}
  \vol_{g^P_t}
  \;
  \|{
   a
  }_{L^\infty_{l,\delta;t},g^P_t}
  \|{
   \kappa
  }_{L^2,g^P_1}
  \\
  &\quad\quad +
  c
  \int_{P_{x,\left[ 1, R t^{-1} \right]}}
  \frac{t^2}{(t+t\check{r})^3}
  w_{l,\delta;t}^{-1}
  \vol_{g^P_t}
  \;
  \|{
   a
  }_{L^\infty_{l,\delta;t},g^P_t}
  \|{
   \kappa
  }_{L^2,g^P_1}
  \\
  &\leq
  c
  \vol_{g^P_t}(P_x,[0,1])
  \cdot
  t^{l-1}
  \;
  \|{
   a
  }_{L^\infty_{l,\delta;t},g^P_t}
  \|{
   \kappa
  }_{L^2,g^P_1}
  \\
  &\quad\quad +
  c
  \int_{(\nu/\{\pm 1\})_{x,\left[ 0, R t^{-1} \right]}}
  \frac{t^2}{(t+t\check{r})^3}
  w_{l,\delta;t}^{-1}
  \vol_{g^\nu_t}
  \;
  \|{
   a
  }_{L^\infty_{l,\delta;t},g^P_t}
  \|{
   \kappa
  }_{L^2,g^P_1}
  \\
  &\leq
  c
  t^{l+3}
  \;
  \|{
   a
  }_{L^\infty_{l,\delta;t},g^P_t}
  \|{
   \kappa
  }_{L^2,g^P_1}
  \\
  &\quad\quad +
  c
  \int_0^{\sqrt{t}}
  t^{2-\delta}
  (t+r)^{l+\delta-3} r^3 \d r
  \cdot
  \|{
   a
  }_{L^\infty_{l,\delta;t},g^P_t}
  \|{
   \kappa
  }_{L^2,g^P_1}
  \\ &\;\;\;\;\;\;
  + c
  \int_{\sqrt{t}}^{R}
  t^2 r^{l-\delta} (t+r)^{-3} r^3 \d r 
  \cdot
  \|{
   a
  }_{L^\infty_{l,\delta;t},g^P_t}
  \|{
   \kappa
  }_{L^2,g^P_1}.
  \end{split}
 \end{align}
 Here we used \cref{equation:decay-of-kappa-in-g-P-t} in the second step.
 In the third step, we switched from integrating over $P_{x,[1,Rt^{-1}]}$ to integrating over $\nu_{x,[1,Rt^{-1}]}$.
 We could do this because $t \check{r}$ on $P$ corresponds to the radius function $r$ on $\nu$, and $g^P_t|_{P_{x,[1,Rt^{-1}]}}-\rho^*g^\nu_t|_{P_{x,[1,Rt^{-1}]}} \rightarrow 0$ measured in $g^\nu_t$ as $t \rightarrow 0$ by \cref{equation:phi-nu-tilde-phi-nu-difference,equation:phi-P-tilde-phi-nu-tilde-difference}.
 The latter implies that we can change $\vol_{g^P_t}$ to $\vol_{g^\nu_t}$.
 We used the definition of $w_{l,\delta;t}$ and changing into sphere coordinates in the fourth step.
 
 We now treat the two integrals separately.
 \begin{align}
 \label{equation:iota-integral-estimate-first}
 \begin{split}
  \int_0^{\sqrt{t}}
  (t+r)^{l+\delta-3} r^3 \d r
  &=
  \left[
  (r+t)^{\delta+l}
  \left(
  -
  \frac{3t}{\delta+l}
  -
  \frac{t^3}{(-2+\delta+l)(r+t)^2}
  \right.\right.
  \\
  &\quad\quad
  \left.\left.
  +
  \frac{3t^2}{(-1+\delta+l)(r+t)}
  +
  \frac{r+t}{1+\delta+l}
  \right)
  \right]_0^{\sqrt{t}}
  \\
  &\leq
  c
  (
  t^{\delta+l+1}
  +
  t^{\delta/2+l/2+1/2}
  )
  \\
  &
  \leq
  c t^{\delta+l+1}
  ,
  \end{split}
  \intertext{where we used a computer algebra system to compute the integral in the first step and used $\delta+l+1<0$ in the third step. %
  For the second integral we find that}
  \label{equation:iota-integral-estimate-second}
  \begin{split}
  \int_{\sqrt{t}}^R
  r^{l-\delta} (t+r)^{-3} r^3 \d r
  &
  \leq
  \int_{\sqrt{t}}^R
  r^{l+1-\delta} \d r
  \\
  &\leq
  \left[
   r^{l+1-\delta}
  \right]^R_{\sqrt{t}}
  \\
  &\leq
  t^{l}\cdot t^{-l/2-\delta/2-1/2}\cdot t^1
  +
  c
  \\
  &\leq
  ct^{l+1}
 \end{split}
 \end{align}
 where we used the fact that $-l-\delta-1>0$ to estimate the first summand in the last step, and the fact that $l \leq -1$ to estimate the second summand in the last step.
 
 Combining \cref{equation:iota-integral-estimate,equation:iota-integral-estimate-first,equation:iota-integral-estimate-second} we get
 \begin{align}
 \label{equation:integral-estimate-combined}
  \int_{P_x}
  \< a, \chi_t \cdot i_* \kappa \>_{g^P_t}
  \vol_{g^P_t}
  &\leq
  c t^{3+l}
  \|{
   a
  }_{L^\infty_{l,\delta;t}}
  \|{
   \kappa
  }_{L^2,g^P_1}.
 \end{align}
 
 If $\kappa_1,\kappa_2 \in (V \mathfrak{M}_t)_{s(x)}$, then
 \begin{align}
 \label{equation:inner-product-rescaling-factor}
 \begin{split}
  \< \chi^+_t \cdot i_* \kappa_1, \chi^+_t \cdot i_* \kappa_2 \>_{L^2,g^P_t}
  & \sim
  \< i_* \kappa_1, i_* \kappa_2 \>_{L^2,g^P_t}
  \\
  & \sim
  t^2 \< i_* \kappa_1, i_* \kappa_2 \>_{L^2,g^P_1},
 \end{split}
 \end{align}
 where $\sim$ means comparable uniformly in $t$.
 Here, in the second step we used the fact that $\vol_{g^P_t|_{P_x}}=t^4 \vol_{g^P_1|_{P_x}}$ and $\< \kappa_1(y), \kappa_2(y) \>_{g^P_t}=t^{-2} \< \kappa_1(y), \kappa_2(y) \>_{g^P_1}$ for $y \in P_x$.
 \Cref{equation:inner-product-rescaling-factor} implies that if $\kappa$ has unit length with respect to the inner product $\< \iota_t \cdot , \iota_t \cdot \>_{g^P_t}$, then
 \begin{align}
 \label{equation:L2-norm-on-P-fibres-estimate}
  \|{
   \kappa
  }_{L^2,g^P_1}
  \leq
  c
  t^{-1}.
 \end{align}
 Because $\|{ \cdot }_{L^2,g^P_1}$ and $\|{ \cdot }_{L^\infty,g^P_1}$ are norms on a finite-dimensional vector space, they are equivalent, and thus
 \begin{align}
 \label{equation:L-infinity-norm-on-P-fibres-estimate}
  \|{
   \kappa
  }_{L^\infty,g^P_1}
  \leq
  c
  t^{-1}.
 \end{align}
 Combining \cref{equation:integral-estimate-combined,equation:L2-norm-on-P-fibres-estimate,equation:L-infinity-norm-on-P-fibres-estimate} and recalling the definition of $\pi_t$ from \cref{definition:pi-iota-splitting} gives
 \begin{align*}
  \|{
   \pi_t a
  }_{L^\infty}
  &\leq
  \left|
  \sum_{\kappa}
  \int_{P_x}
  \< a, \iota_t \kappa \>_{g^P_t} 
  \vol_{g^P_t|_{P_x}}
  \right|
  \cdot 
  \|{
   \kappa
  }_{L^\infty,g^P_1}
  \\
  &\leq  
  c t^{1+l}
  \|{
   a
  }_{L^\infty_{l,\delta;t}}.
 \end{align*}
 The estimate for the $\|{\cdot }_{C^{0,\alpha}}$ Hölder norm follows analogously.
\end{proof}

We are now ready to define the norms which we will use to prove estimates for the operator $L_t$:

\begin{definition}
\label{definition:X-Y-norm-definition}
Denote by $\mathfrak{X}_t$ and $\mathfrak{Y}_t$ the Banach spaces $C^{1,\alpha}(N_t,(\Lambda^0 \oplus \Lambda^1) \tensor \Ad E_t)$ and $C^{0,\alpha}(N_t,(\Lambda^0 \oplus \Lambda^1) \tensor \Ad E_t)$ equipped with the norms
\begin{align}
 \begin{split}
  \|{
   \underline{a}
  }_{\mathfrak{X}_t}
  &:=
  t^{-\delta/2}
  \|{
   \eta_t \underline{a}
  }_{C^{1,\alpha}_{-1,\delta;t}}
  +
  t
  \|{
   \pi_t \underline{a}
  }_{C^{1,\alpha}}
 \;\;\;
 \text{ and}
 \\
 \|{
   \underline{a}
  }_{\mathfrak{Y}_t}
  &:=
  t^{-\delta/2}
  \|{
   \eta_t \underline{a}
  }_{C^{0,\alpha}_{-2,\delta;t}}
  +
  t
  \|{
   \pi_t \underline{a}
  }_{C^{0,\alpha}}
 \end{split}
\end{align}
respectively.
\end{definition}

Using these norms, we can now state the main result of this section:

\begin{proposition}
 \label{proposition:inverse-operator-estimate-xy-norm}
 Let $N_t$ be the resolution of $T^7/\Gamma$ from \cref{section:torsion-free-structures-on-the-generalised-kummer-construction}.
 Let $s$ be the Fueter section and $\theta$ be the $G_2$-instanton used in the construction of $A_t$ (cf. \cref{proposition:approximate-solution}).
 If $s$ is infinitesimally rigid and $\theta$ is infinitesimally rigid and irreducible, then there exists a constant $c>0$ which is independent of $t$ such that for small enough $t$ and all $\underline{a} \in (\Omega^0 \oplus \Omega^1)(N_t,\Ad E_t)$:
 \begin{align}
  \|{
   \underline{a}
  }_{\mathfrak{X}_t}
  \leq
  c
  \|{
   L_t \underline{a}
  }_{\mathfrak{Y}_t}.
 \end{align}
\end{proposition}

Unfortunately, we are restricted to the case where $N_t$ is a resolution of $T^7/\Gamma$.
The reason for this is that in this case we have improved control over the $G_2$-structure $\tilde{\varphi}^N_t$ as proved in \cref{proposition:psi-N-psi-P-comparison-on-T7,corollary:kummer-construction-simplified-torsion-free-estimate}.
The proof of the proposition extends over the rest of this section.

\subsubsection{Comparison with the Fueter Operator}
\label{subsubsection:comparison-with-the-fueter-operator}

Given an element $v \in \Gamma(s^* V\mathfrak{M})$ one may do two different things to it:
either embed it into $\Omega^1(N_t, \Ad E_t)$ first, and then apply $L_t$.
Or apply the linearised Fueter operator first, and then embed it into $\Omega^1(N_t, \Ad E_t)$.
It will turn out that the two are the same, up to a small error.
In \cite{Walpuski2017}, Fueter sections into a moduli bundle of ASD-instantons on $\R^4$ were considered, and the following proposition was proved in that setting.
In this article, ASD-instantons on ${\XEH}$ are considered, but the proof works essentially the same way.
That said, we do need that
$\tilde{\psi}^N_t-\psi^P_t$
is small.
This is true on resolutions of $T^7/\Gamma$ by \cref{proposition:psi-N-psi-P-comparison-on-T7,corollary:kummer-construction-simplified-torsion-free-estimate} but not proved for general resolutions of $G_2$-orbifolds.
Consequently, we only know the following two propositions to hold on resolutions of $T^7/\Gamma$.

\begin{proposition}[Proposition 8.26 in \cite{Walpuski2017}]
\label{proposition:fueter-comparison-theorem}
 Let $N_t$ be the resolution of $T^7/\Gamma$ from \cref{section:torsion-free-structures-on-the-generalised-kummer-construction}.
 There exists a constant $c>0$ such that for all $t \in (0,T)$ and all $v \in \Gamma(s^* V\mathfrak{M})$ the following estimate holds:
 \begin{align}
  \|{
   L_t \iota_t v - \iota_t \d_s \mathfrak{F} v
  }_{C^{0,\alpha}_{-2,0;t}}
  \leq
  c t^2 \|{ v }_{C^{1,\alpha}}.
 \end{align}
\end{proposition}

The following proposition is then a simple consequence.
It essentially provides the estimate for the inverse of $L_t$ on the space $\Im \overline{\pi}_t \subset \Omega^1(N_t, \Ad E_t)$.

\begin{proposition}
 \label{proposition:pi-L-iota-inequality}
 Let $N_t$ be the resolution of $T^7/\Gamma$ from \cref{section:torsion-free-structures-on-the-generalised-kummer-construction}.
 If $s$ is infinitesimally rigid, then there exists a constant $c>0$ such that for all $t \in (0,T)$ and all $v \in \Gamma(s^* V \mathfrak{M})$ the following estimate holds:
 \begin{align}
  \|{
   v
  }_{C^{1,\alpha}}
  \leq
  c
  \|{
   \pi_t L_t \iota_t v
  }_{C^{0,\alpha}}.
 \end{align}
\end{proposition}

\begin{proof}
We have
\begin{align*}
  \|{
   v
  }_{C^{1,\alpha}}
  &\leq
  c
  \|{
   \d_s \mathfrak{F} v
  }_{C^{0,\alpha}}
  \\
  &=
  c
  \|{
   \pi_t \iota_t \d_s \mathfrak{F} v
  }_{C^{0,\alpha}}
  \\
  &\leq
  c
  \left(
  \|{
   \pi_t L_t \iota_t v
  }_{C^{0,\alpha}}
  +
  \|{
   \pi_t (L_t \iota_t v- \iota_t \d_s \mathfrak{F} v)
  }_{C^{0,\alpha}}
  \right)
  \\
  &\leq
  c
  \left(
  \|{
   \pi_t L_t \iota_t v
  }_{C^{0,\alpha}}
  +
  t^{1-\alpha}
  \|{
   v
  }_{C^{0,\alpha}}
  \right),
\end{align*}
where we used the fact that $s$ is infitesimally rigid in the first step,
and \cref{proposition:iota-pi-bounds,proposition:fueter-comparison-theorem} in the last step.
For small $t$, we can then absorb the factor $  t^{1-\alpha}
  \|{
   v
  }_{C^{0,\alpha}}$ into the left hand side.
\end{proof}

\begin{remark}
 Apart from the connection to \cite{Walpuski2017}, the situation in this subsection is also very similar to, but more complicated than, the situation in \cite[Propositions 4.29 and 4.35]{Platt2020}.
\end{remark}

\subsubsection{The Model Operators on $\R^3 \times {\XEH}$ and $\R^3 \times \C^2/\{\pm 1\}$}
\label{subsubsection:the-model-operator}

As before, let ${\XEH}$ be the Eguchi-Hanson space.
To prove the estimate in \cref{proposition:inverse-operator-estimate-xy-norm}, we will compare the operator $L_t$ with the linearised instanton equation in the model case of a pulled back ASD instanton on $\R^3 \times {\XEH}$.

\textbf{Properties of the Model Operator}

Let $A$ be a finite energy ASD instanton on a $G$-bundle $E$ over ${\XEH}$.
The infinitesimal deformations of $A$ are then governed by the operator $\delta_A$ from \cref{equation:asd-instanton-linearisation}.
Denote by $p_{\XEH}: \R^3 \times {\XEH} \rightarrow {\XEH}$ the projection onto the second factor.
By a slight abuse of notation we denote the pullbacks of $A$ and $E$ to $\R^3 \times {\XEH}$ under $p_{\XEH}$ by $A$ and $E$ as well.

Denote by $L_A$ be the linearised $G_2$-instanton operator from \cref{equation:g2-instanton-linearisation}.
We can define the map $( \; \cdot \;)^{\sharp} \lrcorner \varphi : p_{\R^3}^* T^* \R^3 
  \overset{\simeq}{\rightarrow}
  p_{\XEH}^* \Lambda^+T^*{\XEH}$, which takes a $1$-form, dualises it, and plugs it into the product $G_2$-structure $\varphi$ from \cref{equation:product-g2-structure}.
 It maps $\d x_i$ to $-\omega_i$.
 Using it, we can relate $\delta_A$ and $L_A$ as follows:

\begin{proposition}[Proposition 2.70 in \cite{Walpuski2013}]
\label{proposition:G2-linearisation-for-pullbacks}
 Under the identification
 \begin{align*}
  ( \; \cdot \;)^{\sharp} \lrcorner \varphi: p_{\R^3}^* T^* \R^3 & 
  \overset{\simeq}{\rightarrow}
  p_{\XEH}^* \Lambda^+T^*{\XEH}
 \end{align*}
 and accordingly
 \begin{align*}
  \Omega^0 \oplus \Omega^1 (\R^3 \times {\XEH}, \Ad E)
  \simeq
  \Omega^0
  (
  \R^3 \times {\XEH},
  p_{\XEH}^*
  [
   (\R \oplus \Lambda^+T^*{\XEH} \oplus T^*{\XEH}) \tensor \Ad E
  ])
 \end{align*}
 the operator $L_A$ can be written as $L_A=F+D_A$ where
 \begin{align*}
  F(\xi, \omega, a)
  =
  \sum_{i=1}^3
  (-
  \< \partial_i \omega, \omega_i \>,
  \partial_i \xi \cdot \omega_i,
  I_i \partial_i a
  )
  \;\;\;
  \text{ and }
  \;\;\;
  D_A
  =
  \begin{pmatrix}
   0&\delta_A \\
   \delta_A^*&0
  \end{pmatrix}.
 \end{align*}
 Moreover,
 \begin{align}
 \label{equation:L*L-equals-Laplacian-plus-something}
  L_A^* L_A
  =
  \Delta_{\R^3}
  +
  \begin{pmatrix}
   \delta_A \delta_A^* &
   \\
   & \delta_A^* \delta_A
  \end{pmatrix}.
 \end{align}
\end{proposition}

We define the following weighted norms on the model spaces $\R^3 \times {\XEH}$ and $\R^3 \times (\C^2/\{ \pm 1\})$:

\begin{definition}
\label{definition:model-norms}
For $\beta \in \R$, let
\begin{align*}
  w_{\beta}
  : \R^3 \times {\XEH} & \rightarrow \R
  &
  w_{\beta}
  : \R^3 \times (\C^2/\{ \pm 1\}) & \rightarrow \R  
  \\
  x &\mapsto (1+\check{r}(x))^{-\beta}
  &
  x &\mapsto (1+d(x,0))^{-\beta}
\end{align*}
and define the weighted Hölder norms $\|{\cdot }_{C^{0,\alpha}_{\beta}}$ as in \cref{definition:hoelder-norms-for-gauge-with-cases}, but using this new weight function.
\end{definition}

We then have the following result:

\begin{proposition}[Proposition 2.74 in \cite{Walpuski2013}]
\label{proposition:constant-in-r3-direction}
 Let $\tilde{X}$ be an ALE space.
 Let $\beta \in (-3,0)$.
 Then $\underline{a} \in C^{1,\alpha}_\beta$ is in the kernel of $L_I: C^{1,\alpha}_\beta \rightarrow C^{0,\alpha}_{\beta-1}$ if and only if it is given by the pullback of an element of the $L^2$ kernel of $\delta_I$ to $\R^3 \times \tilde{X}$.
\end{proposition}

\textbf{Comparison with $L_t$}

We now explain two maps $s^P$ and $s^\nu$:
the first for "zooming into" the resolution locus of the associative $L$, the second for "zooming into" the gluing region of $N_t$.
Fix a point $y \in L$, a scaling parameter $d \in \mathbb{Z}$, a gluing parameter $t \in (0,T)$, and two positive real numbers $\epsilon_1, \epsilon_2$ defining the scale of the region into which to zoom in.

Let
\begin{align*}
V_{\epsilon_1, \epsilon_2;t}^P(y)
&:=
\{
 x \in P:
 \sigma(x) \in \Im (\exp_y|_{(-\epsilon_1,\epsilon_1)^3}),
 \check{r}(x)t < \epsilon_2
\} \subset P,
\\
U_{\epsilon_1/t, \epsilon_2/t;t}^P
&:=
\{
 (x,z) \in \R^3 \times {\XEH}:
 x \in (-\epsilon_1/t,\epsilon_1/t)^3,
 \rho(z) < \epsilon_2/t
\}.
\end{align*}
Here we implicitly used an identification $T_yL \simeq \R^3$ to have $\exp_y$ acting on $(-\epsilon_1,\epsilon_1)^3$.
Choose this identification so that it maps the orthonormal basis $e_1(y),e_2(y),e_3(y) \in T^*_y L$ from \cref{subsection:g2-structures-on-the-resolution} to the standard basis $\d x_1, \d x_2, \d x_3 \in \Lambda^1((\R^3)^*)$.
Fix an element $f \in \Fr_y$ of the unitary frame bundle of $\nu$ around $y \in L$.
It induces an isometry ${\XEH} \simeq P_y$, and assume that $f$ is chosen so that $\tilde{\omega}_i$ is sent to $\check{\omega}_i|_{P_y}$ under this map for $i \in \{1,2,3\}$.
Then, for small $\epsilon_1$, we define
\begin{align}
\label{equation:P-blowup-map-without-scaling}
\begin{split}
 E^P:U_{\epsilon_1/t, \epsilon_2/t;t}^P & \rightarrow V_{\epsilon_1, \epsilon_2;t}^P(y)
 \\
 (x,z) &\mapsto
 \mathcal{P}_{s \mapsto \exp_y(stx)}
 (f(z))
 \in P.
\end{split}
\end{align}
Here, $s \mapsto \exp_y(stx)$ denotes the unique shortest geodesic from $y$ to $\exp(tx)$ in $L$, and $\mathcal{P}_{s \mapsto \exp_y(stx)}$ denotes parallel transport in $P$ with respect to $\breve{H}$ along this curve, cf. the paragraph before \cref{definition:varphi-P-t}.
For $\epsilon_1$ small enough, this is a diffeomorphism.
The reason for this definition is the following:
because of our choices of identifications $T_y L \simeq \R^3$ and $P_y \simeq {\XEH}$ we have that $(E^P)^*(\varphi^P_t)(0,z)$ is the standard $G_2$-structure on $\R^3 \times {\XEH}$ for all $z \in {\XEH}$, cf. \cref{definition:varphi-P-t}.
Let $a$ be a tensor field of valence $(p,q)$, i.e. in index notation $p$ lower indices and $q$ upper indices, on $V_{\epsilon_1, \epsilon_2;t}^P(y)$.
We then define
\begin{align}
\label{equation:s-blowup-rescaling-map-definition}
s^P(a)
:=
s^{P, \epsilon_1, \epsilon_2}_{d,y; t}(a)
:=
t^{d+p-q} (E^P)^*a,
\end{align}
which is a tensor on $U_{\epsilon_1/t, \epsilon_2/t;t}$.
The point of this is the following proposition:

\begin{proposition}
\label{proposition:blowup-comparison}
 There are constants $c>0$ and $\epsilon>0$ such that for small $t$ the following holds:
 for all $\epsilon_1, \epsilon_2 \in (0, \epsilon)$ and for all $\underline{a} \in (\Omega^0 \oplus \Omega^1)(N_t,E_t)$:
 \begin{align}
 \label{equation:blowup-comparison-tensor-scaling}
  \|{
   s_{d,t;y}^{P,\epsilon_1, \epsilon_2} \underline{a}
   }
   _{L^\infty_{l + \delta}(U_{\epsilon_1/t, \epsilon_2/t;t}^P)}
  &\sim
  t^{d+l}
  \|{
   \underline{a}
   }
   _{L^\infty_{l,\delta;t}(V_{\epsilon_1, \epsilon_2}^P(y))}
   ,\\
   \label{equation:blowup-comparison-tensor-scaling-higher-norms}
  \|{
   s_{d,t;y}^{P,\epsilon_1, \epsilon_2} \underline{a}
   }
   _{C^{k,\alpha}_{l + \delta}(U_{\epsilon_1/t, \epsilon_2/t;t}^P)}
  &\sim
  t^{d+l}
  \|{
   \underline{a}
   }
   _{C^{k,\alpha}_{l,\delta;t}(V_{\sqrt{t},\sqrt{t}}^P(y))},
 \end{align}
 where $\sim$ means comparable independently of $t$.
 Furthermore, using the Hyperkähler isomorphism $P_y \simeq {\XEH}$ induced by $f$, we can view the connection $s(A)$ over $P_y$ as a connection over ${\XEH}$, denoted by $f_*(s(y))$.
 Then
 \begin{align}
 \label{equation:blowup-comparison-instanton-linearisation}
   \|{
    L_t \underline{a}
    -
    \left( 
     s_{2,t;y}^{P,\sqrt{t},\sqrt{t}}
    \right)
    ^{-1}
    L_{p_{\XEH}^* f_*(s(y))}
    s_{1,t;y}^{P,\sqrt{t},\sqrt{t}}
    \underline{a}
    }
    _{C^{0,\alpha}_{-2,\delta;t}(V_{\sqrt{t},\sqrt{t}}^P(y))}
    &\leq
    c \sqrt{t}
    \|{
     \underline{a}
     }_{C^{1,\alpha}_{-1,\delta;t}(V_{\sqrt{t},\sqrt{t}}^P(y))}.
 \end{align}
\end{proposition}

\begin{proof}
 We first prove \cref{equation:blowup-comparison-tensor-scaling}:
 for $(0,z) \in U_{\epsilon_1/t, \epsilon_2/t;t}$ the map $\d _{(0,z)}E^P$ 
 (cf. \cref{equation:P-blowup-map-without-scaling}) is an isometry for the metric $t^2(g_{\R^3} \oplus g_{(1)})$ on $T_{(0,z)} (\R^3 \times {\XEH})$ and the metric on $T_{E^P(0,z)}P$ induced by $g^P_t$.
 Because of the scaling factor $t^{d+p-q}$ from \cref{equation:s-blowup-rescaling-map-definition} we have that
 \begin{align}
 \label{equation:rescaling-pointwise-norm-equality}
  |s_{d,t;y}^{P,\epsilon_1,\epsilon_2} \underline{a} (0,z)|_{g_{\R^3} \oplus g_{(1)}}
  =
  t^d
  |\underline{a} (E^P(0,z))|_{g^P_t}.
 \end{align}
 The map $E^P$ is not, in general, an isometry away from this one point, as $\exp_y$ need not be an isometry.
 Thus, \cref{equation:rescaling-pointwise-norm-equality} need not hold in points different from $(0,z)$.
 However, using Taylor expansions in a neighbourhood of $y$ in $L$ for $\underline{a}$ and $g^P_t$ we get
 \begin{align*}
  \|{
   s_{d,t;y}^{P,\epsilon_1,\epsilon_2} \underline{a}
   }
   _{L^\infty_{l + \delta}(U_{\epsilon_1/t,\epsilon_2/t;t})}
  &\sim
  t^{d+l}
  \|{
   \underline{a}
   }
   _{L^\infty_{l,\delta;t}(V_{\epsilon_1,\epsilon_2}(y)),g^P_t}.
 \end{align*}
 Now \cref{equation:psi-n-psi-p-pointwise-estimate,proposition:equation:torsion-free-hoelder-difference} give the claim for the metric $\tilde{g}^N_t$ instead of $g^P_t$, which is \cref{equation:blowup-comparison-tensor-scaling}.
 \Cref{equation:blowup-comparison-tensor-scaling-higher-norms} is proved analogously.
 
 Now to prove \cref{equation:blowup-comparison-instanton-linearisation}:
 as in \cref{equation:rescaling-pointwise-norm-equality}, we see that for $x \in P_y$, $\check{r}(x)<1/\sqrt{t}$,
 \begin{align}
 \label{equation:pointwise-operator-applied-zoomed-in}
  L_{s(A)} \underline{a}(x)
    -
    \left(
    \left( 
     s_{2,t;y}^{P,\sqrt{t},\sqrt{t}}
    \right)
    ^{-1}
    L_{p_{\XEH}^* f_*(s(y))}
    s_{1,t;y}^{P,\sqrt{t},\sqrt{t}}
    \underline{a}
    \right)
    (x)
  =0.
 \end{align}
 And $A_t-s(A)=\mathcal{O}(1)$ on $P_y$, so
 \begin{align}
 \label{equation:pointwise-operator-applied-zoomed-in-with-weights}
 \begin{split}
 &\quad
 \|{
  L_t \underline{a}
    -
    \left(
    \left( 
     s_{2,t;y}^{P,\sqrt{t},\sqrt{t}}
    \right)
    ^{-1}
    L_{p_{\XEH}^* f_*(s(y))}
    s_{1,t;y}^{P,\sqrt{t},\sqrt{t}}
    \underline{a}
    \right)
   }_{C^{0,\alpha}_{-2,\delta;t}(\{u \in P_y: \check{r}(u)<1/\sqrt{t} \} )}
   \\&
   \leq
   c
   \|{
     [A_t-s(A), \underline{a}]
   }_{C^{0,\alpha}_{-2,\delta;t}(\{u \in P_y: \check{r}(u)<1/\sqrt{t} \})}
   \\&
  \leq
  c
   \|{
     \underline{a}
   }_{C^{0,\alpha}_{-1,\delta;t}(\{u \in P_y: \check{r}(u)<1/\sqrt{t} \})}
   \|{
     A_t-s(A)
   }_{C^{0,\alpha}_{-1,0;t}(\{u \in P_y: \check{r}(u)<1/\sqrt{t} \})}
    \\&
  \leq
  c \sqrt{t}
   \|{
     \underline{a}
   }_{C^{0,\alpha}_{-1,\delta;t}(\{u \in P_y: \check{r}(u)<1/\sqrt{t} \})}
    \\&
  \leq
  c \sqrt{t}
   \|{
     \underline{a}
   }_{C^{1,\alpha}_{-1,\delta;t}(\{u \in P_y: \check{r}(u)<1/\sqrt{t} \})}
     \end{split}
 \end{align}
 where in the third step we used $A_t-s(A)=\mathcal{O}(1)$ to estimate the second factor as $\sqrt{t}$.
 This was possible because the weight function is bounded by $\sqrt{t}$ on $\{u \in P_y: \check{r}(u)<1/\sqrt{t} \}$. 
 
 \Cref{equation:blowup-comparison-instanton-linearisation} now follows from using Taylor expansions for $\underline{a}$, $g^P_t$, and $s$ around $y$, and comparing $g^P_t$ and $\tilde{g}^N_t$ as in the proof of \cref{equation:blowup-comparison-tensor-scaling}.
\end{proof}

We now define $s^{\nu}$:
let $\epsilon_1>0$, $\epsilon_2 > \epsilon_3 > 0$, and
\begin{align*}
V_{\epsilon_1, \epsilon_2, \epsilon_3;t}^\nu(y)
&:=
\{
 x \in \nu/\{ \pm 1\}:
 \sigma(x) \in \Im (\exp_y|_{(-\epsilon_1,\epsilon_1)^3}),
 \epsilon_3 < r(x) < \epsilon_2
\},
\\
U_{\epsilon_1/t, \epsilon_2/t, \epsilon_3/t;t}^\nu
&:=
\{
 (x,z) \in \R^3 \times \C^2/\{ \pm 1\}:
 x \in (-\epsilon_1/t,\epsilon_1/t)^3,
 \epsilon_3/t < |\rho(z)| < \epsilon_2/t
\}.
\end{align*}

Just as in the definition of $V_{\epsilon_1, \epsilon_2;t}^P$, we implicitly used an identification $T_yL \simeq \R^3$ so that $e^i$ is sent to $\d x^i$ for $i \in \{1,2,3\}$.
Recall also the frame $f$ that sends $\tilde{\omega}_i$ to $\check{\omega}_i|_{P_y}$ for $i \in \{1,2,3\}$ under the isometry ${\XEH} \simeq P_y$ induced by $f$.
We see from \cref{equation:hyperkaehler-commutative-diagram} that $\omega_i$ is sent to $\hat{\omega}_i|_{\nu_y}$ under the isometry $\C^2/\{\pm 1\} \simeq (\nu/\{\pm 1\})_y$ induced by $f$.
For small $\epsilon_1,\epsilon_2,\epsilon_3$, the map
\begin{align}
\label{equation:nu-blowup-map-without-scaling}
\begin{split}
 E^\nu:U_{\epsilon_1/t, \epsilon_2/t, \epsilon_3/t;t}^\nu & \rightarrow V_{\epsilon_1, \epsilon_2, \epsilon_3;t}^\nu(y)
 \\
 (x,z) &\mapsto
 \mathcal{P}^\nu_{s \mapsto \exp_y(tsx)}
 (f(z))
 \in \nu/\{\pm 1\}
\end{split}
\end{align}
is a diffeomorphism, where $\mathcal{P}^\nu$ denotes parallel transport in $\nu$ with respect to the connection $\tilde{\nabla}^\nu$ from \cref{proposition:jk-normal-bundle-choices}.
Because of our choices of identifications $T_y L \simeq \R^3$ and $(\nu/\{\pm 1\})_y \simeq \C^2/\{\pm 1\}$ we have that $(E^P)^*(\varphi^\nu_t)(0,z)$ is the standard $G_2$-structure on $\R^3 \times \C^2/\{\pm 1\}$, for all $z \in \C^2/\{ \pm 1 \}$, cf. \cref{definition:varphi-nu-t}.
We now define $s^\nu$ just as we defined $s^P$ in \cref{equation:s-blowup-rescaling-map-definition}, only exchanging $E^P$ for $E^\nu$:
for a tensor field $a$ of valence $(p,q)$ on $V_{\epsilon_1, \epsilon_2, \epsilon_3;t}^\nu(y)$ set
\begin{align}
\label{equation:s-blowup-rescaling-map-definition}
s^\nu(a)
:=
s^{\nu, \epsilon_1, \epsilon_2, \epsilon_3}_{d,y; t}(a)
:=
t^{d+p-q} (E^\nu)^*a.
\end{align}

In the following we use the norms from \cref{definition:model-norms}.
So, the notation $C^{0,\alpha}_0$ does not mean zero boundary condition, but means that the weight function appears with powers of $0$ and $0+\alpha$ in the two summands of the definition $\|{\cdot}_{C^{0,\alpha}_0}$.
We have the following analogue of \cref{proposition:blowup-comparison}:

\begin{proposition}
\label{proposition:blowup-comparison-nu}
 There are constants $c>0$ and $\epsilon>0$ such that for small $t$ the following holds:
 for all $\epsilon_1, \epsilon_2 \in (0, \epsilon)$, $\epsilon_3 \in (t, \epsilon)$ and for all $\underline{a} \in (\Omega^0 \oplus \Omega^1)(N_t,E_t)$:
 \begin{align}
 \label{equation:blowup-comparison-tensor-scaling-nu}
  \|{
   w^\nu_{l,\delta;t}
   s_{d,t;y}^{\nu,\epsilon_1, \epsilon_2, \epsilon_3}
   \underline{a}
   }
   _{L^\infty_{0}(U_{\epsilon_1/t, \epsilon_2/t, \epsilon_3/t;t}^\nu)}
  &\sim
  t^{d+l}
  \|{
   \underline{a}
   }
   _{L^\infty_{l,\delta;t}(V_{\epsilon_1, \epsilon_2, \epsilon_3}^\nu(y))}
   ,\\
   \label{equation:blowup-comparison-tensor-scaling-higher-norms-nu}
  \|{
  w^\nu_{l,\delta;t}
   s_{d,t;y}^{\nu,\epsilon_1, \epsilon_2, \epsilon_3}
   \underline{a}
   }
   _{C^{k,\alpha}_{0}(U_{\epsilon_1/t, \epsilon_2/t, \epsilon_3/t;t}^\nu)}
  &\sim
  t^{d+l}
  \|{
   \underline{a}
   }
   _{C^{k,\alpha}_{l,\delta;t}(V_{\epsilon_1, \epsilon_2, \epsilon_3}^\nu(y))}
   ,
 \end{align}
 where $\sim$ means uniformly comparable in $t$ and
 \begin{align*}
  w^\nu_{l,\delta;t}
  =
  \begin{cases}
   r^{-l-\delta}
   & \text{ if }
   r \leq 1/\sqrt{t}
   \\
   r^{-l+\delta}t^{\delta}
   & \text{ if }
   r > 1/\sqrt{t}.
  \end{cases}
 \end{align*}
 Furthermore, using the Hyperkähler isomorphism $P_y \simeq {\XEH}$ induced by $f$, we can view the connection $s(A)$ over $P_y$ as a connection over ${\XEH}$.
 By \cref{equation:asymptotic-at-infinity-condition,equation:ALE-moduli-definitions}, this connection is asymptotic to a flat connection, say $A_0$, on the orbifold $\C^2/\{ \pm 1\}$ with monodromy representation $\rho$.
 Then
 \begin{align}
 \label{equation:blowup-comparison-instanton-linearisation-nu}
 \begin{split}
 &\quad
   \|{
    L_t \underline{a}
    -
    \left( 
     s_{2,t;y}^{\nu,\epsilon_1, \epsilon_2, \epsilon_3}
    \right)
    ^{-1}
    L_{p_{\C^2}^* A_0}
    s_{1,t;y}^{\nu,\epsilon_1, \epsilon_2, \epsilon_3}
    \underline{a}
    }
    _{C^{0,\alpha}_{-2,\delta;t}(V_{\epsilon_1, \epsilon_2, \epsilon_3}^\nu(y))}
    \\
    &\leq
    c (\epsilon_1 + \epsilon_2 + (t/\epsilon_3)^2 )
    \|{
     \underline{a}
     }_{C^{1,\alpha}_{-1,\delta;t}(V_{\epsilon_1, \epsilon_2, \epsilon_3}^\nu(y))},
 \end{split}
 \end{align}
 where $p_{\C^2}:\R^3 \times \C^2/\{\pm 1\} \rightarrow \C^2/\{\pm 1\}$ denotes the projection onto the second factor.
\end{proposition}

\begin{proof}
 \Cref{equation:blowup-comparison-tensor-scaling-nu,equation:blowup-comparison-tensor-scaling-higher-norms-nu} are proved as in \cref{proposition:blowup-comparison}.
 
 We now prove \cref{equation:blowup-comparison-instanton-linearisation-nu}.
 Adapting \cref{equation:pointwise-operator-applied-zoomed-in-with-weights} to the area $\{u \in P_y : \epsilon_3/t < \check{r}(u) < \epsilon_2/t \}$ we get
 \begin{align}
 \label{equation:pointwise-operator-applied-zoomed-in-with-weights-nu}
 \begin{split}
 &\quad
 \|{
  L_t \underline{a}
    -
    \left(
    \left( 
     s_{2,t;y}^{P,\epsilon_1,\epsilon_2}
    \right)
    ^{-1}
    L_{p_{\XEH}^* f_*(s(y))}
    s_{1,t;y}^{P,\epsilon_1,\epsilon_2}
    \underline{a}
    \right)
   }_{C^{0,\alpha}_{-2,\delta;t}(\{u \in P_y : \epsilon_3/t < \check{r}(u) < \epsilon_2/t \})}
   \\&
  \leq
  c \epsilon_2
  \|{
     \underline{a}
     }_{C^{1,\alpha}_{-1,\delta;t}(\{u \in P_y : \epsilon_3/t < \check{r}(u) < \epsilon_2/t \})}.
     \end{split}
 \end{align}
 
 We have $\|{ p_{\XEH}^* f_*(s(y)) - A_0 }_{C^{0,\alpha}_{0;0}} = \mathcal{O}((\rho\circ p_{\XEH})^{-2})$ by \cref{equation:ALE-moduli-definitions} and the fact that we use $\delta =-2$ in our definition of moduli space (cf. \cref{proposition:moduli-independent-of-regularity}).
 Thus, for $x \in P_y$ with $\epsilon_3/t < \check{r}(x)t < R$,
 \begin{align}
 \label{equation:zooming-in-asd-flat-limit-comparison}
    \left|
    \left( 
     s_{2,t;y}^{P,\sqrt{t},\sqrt{t}}
    \right)
    ^{-1}
    \left[
    L_{p_{\XEH}^* f_*(s(y))}
    -
    L_{p_{\XEH}^* A_0}
    \right]
    s_{1,t;y}^{P,\sqrt{t},\sqrt{t}}
    \underline{a}
    \right|_{\tilde{g}^N_t}
    (x)
  	\leq
  	c(t/\epsilon_3)^2.
 \end{align}
 Combining \cref{equation:pointwise-operator-applied-zoomed-in-with-weights-nu,equation:zooming-in-asd-flat-limit-comparison} we get the desired \cref{equation:blowup-comparison-instanton-linearisation-nu} on $P_y \cap V^\nu_{\epsilon_1,\epsilon_2,\epsilon_3}(y)$. 
 \Cref{equation:blowup-comparison-instanton-linearisation-nu} then follows like \cref{equation:blowup-comparison-instanton-linearisation} by taking Taylor expansions around $y$, and this time comparing $g^\nu_t$ and $\tilde{g}^N_t$ using 
 \cref{equation:phi-nu-tilde-phi-nu-difference,proposition:xi-estimates,proposition:equation:torsion-free-hoelder-difference,proposition:psi-n-psi-p-estimate}.
\end{proof}

\subsubsection{Schauder Estimate}
\label{subsubsection:schauder-estimate}

On $Y/\<\iota\>$ we have the estimate
\[
  \|{
   \underline{a}
   }_
   {C^{1,\alpha}}
  \leq
  c
  \left(
   \|{
    L_\theta \underline{a}
   }_
   {C^{0,\alpha}}
   +
   \|{
    \underline{a}
   }_{L^\infty}
  \right)
\]
from standard elliptic theory, e.g. \cite[Section H]{Besse1987}.
With some additional work, we get an estimate for weighted norms on $\R^3 \times {\XEH}$ (see \cite[Proposition 8.15]{Walpuski2017}), and can then glue these two estimates together to obtain:

\begin{proposition}[Proposition 8.15 in \cite{Walpuski2017}]
\label{proposition:schauder-estimate}
 There exists $c >0$ such that for all $t \in (0,T)$ the following estimate holds:
 \begin{align}
  \|{
   \underline{a}
   }_
   {C^{1,\alpha}_{-1,\delta;t}}
  \leq
  c
  \left(
   \|{
    L_t \underline{a}
   }_
   {C^{0,\alpha}_{-2,\delta;t}}
   +
   \|{
    \underline{a}
   }_{L^\infty_{-1,\delta;t}}
  \right).
 \end{align}
\end{proposition}

\subsubsection{Estimate of $\eta_t \underline{a}$}
\label{subsubsection:estimate-of-eta-a}

The following proposition is the crucial ingredient in the construction of solutions to the instanton equation:

\begin{proposition}
\label{proposition:crucial-proposition}
 There exists a constant $c>0$ independent of $t$ such that for $t$ small enough and for all $\underline{a} \in (\Omega^0 \oplus \Omega^1)(N_t, \Ad E_t)$ the following estimate holds:
 \begin{align}
  \|{
   a
   }_{L^\infty_{-1,\delta;t}}
  \leq
  c
  \left(
   \|{
    L_t \underline{a}
    }_{C^{0,\alpha}_{-2,\delta;t}}
   +
   \|{
    \overline{\pi}_t \underline{a}
    }_{L^\infty_{-1,\delta;t}}
  \right).
 \end{align}
\end{proposition}

\begin{proof}
 Assume not, then there exist $t_i \rightarrow 0$ and $\underline{a}_i$ such that
 \begin{align}
  \label{equation:norm-a-i-constant-1}
  \|{
  \underline{a}_i
  }
  _{L^\infty_{-1,\delta;t_i}}
  &\equiv
  1,
  \\
  \label{equation:L-a_i-to-0}
  \lim_{i \rightarrow \infty}
   \|{
    L_{t_i} \underline{a}
    }
    _{C^{0,\alpha}_{-2,\delta;t_i}}
  &= 0,
  \\
  \label{equation:pi-a_i-to-0}
  \lim_{i \rightarrow \infty}
   \|{
   \overline{\pi}_{t_i} \underline{a}
   }
   _{L^\infty_{-1,\delta;t_i}}
  &= 0.
 \end{align}
 It follows from \cref{proposition:schauder-estimate} that
 \begin{align}
  \label{equation:sequence-bounded-in-c-1-alpha}
  \|{
   \underline{a}_i
   }_{C^{1,\alpha}_{-1,\delta;t}}
   \leq
   c.
 \end{align}
 Let $x_i \in N_{t_i}$ such that
 \begin{align}
 \label{equation:contradiction-proof-nonzero-property}
  w_{-1,\delta;t}(x_i)
  \left| \underline{a}_i \right| (x_i)
  =1.
 \end{align}
 Without loss of generality we can assume to be in one of three following cases, and we will arrive at a contradiction in each of them.
 
 \textbf{Case 1.}
 ``$\underline{a}_i$ goes to zero near $L$ and on the neck'', i.e. if $z_i \in N_{t_i}$ such that $r_{t_i}(z_i)\rightarrow 0$, then $w_{-1,\delta;t}(z_i) \left| \underline{a}_i \right| (z_i) \rightarrow 0$.
 
 Without loss of generality, the sequence $(x_i)$ accumulates away from $L$, i.e. $\lim_{i \rightarrow \infty} r_{t_i} (x_i) >0$
 (see \cref{figure:blow-up-1-gauge}).
 
 \begin{figure}[htbp]
  \begin{center}
   \includegraphics[width=10cm]{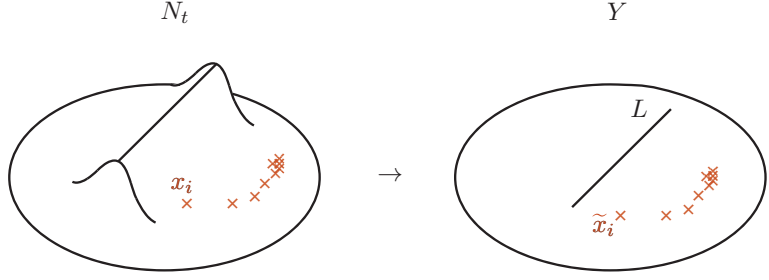}
  \end{center}
  \caption{Blowup analysis away from the associative is reduced to the analysis of $\theta$ on $Y$.}
  \label{figure:blow-up-1-gauge}
 \end{figure}
 
 Without loss of generality assume that $x_i \rightarrow x^* \in Y/\< \iota \>$, where we used that $(Y \setminus L)/\< \iota \> \subset N_{t_i}$, cf. \cref{definition:joyce-karigiannis-N-t}. 
 Now, using a diagonal argument and the Arzelà–Ascoli theorem, we find that a subsequence of $\underline{a}_i$ converges to a limit $\underline{a}^* \in \Omega^1((Y \setminus L)/\< \iota \>, \Ad E_0)$ in $C^{1,\alpha/2}_{\loc}$.
 Denote by $\pi_\iota: Y \rightarrow Y/\< \iota \>$ the quotient map, and denote by $\tilde{x}_i$ an arbitrary lift of $x_i$, i.e. $\pi_\iota(\tilde{x}_i)=x_i$.
 By passing to a subsequence we still have $\tilde{x}_i \rightarrow \tilde{x}^*$ for some $\tilde{x}^* \in Y$.
 Denote also $\tilde{\underline{a}}^*:= \pi_\iota^* \underline{a}^* \in (\Omega^0 \oplus \Omega^1)(\Ad E_0|_{Y \setminus L})$.
 
 \Cref{equation:L-a_i-to-0} implies that this limit satisfies $L_\theta \tilde{\underline{a}}^*=0$ on $Y \setminus L$.
 We can extend $\tilde{\underline{a}}^*$ to all of $Y$ as a distribution, and we find that then $L_\theta \tilde{\underline{a}}^*=0$ on $Y$ in the sense of distributions.
 By elliptic regularity, e.g. \cite[Theorem 6.33]{Folland1995}, we have that $\tilde{\underline{a}}^*$ is smooth.
 
 Lastly, we note that \cref{equation:contradiction-proof-nonzero-property} implies $\tilde{\underline{a}}^*(\tilde{x}^*) \neq 0$.
 By assumption, $\theta$ is infinitesimally rigid and irreducible, which is a contradiction.
 
 \textbf{Case 2.}
 ``The sequence does not go to zero near $L$'', i.e. there exists $y_i \in N_{t_i}$ such that $t_i^{-1}  r_{t_i}(y_i)$ is bounded, but $w_{-1,\delta;t}(y_i) \left| \underline{a}_i \right| (y_i) \nrightarrow 0$.
 
 Without loss of generality assume that this is the sequence $(x_i)$, i.e. $\lim_{i \rightarrow \infty} t_i^{-1} r_{t_i} (x_i) < \infty$
 (see \cref{figure:blow-up-2-gauge}).
 
 \begin{figure}[htbp]
  \begin{center}
   \includegraphics[width=10cm]{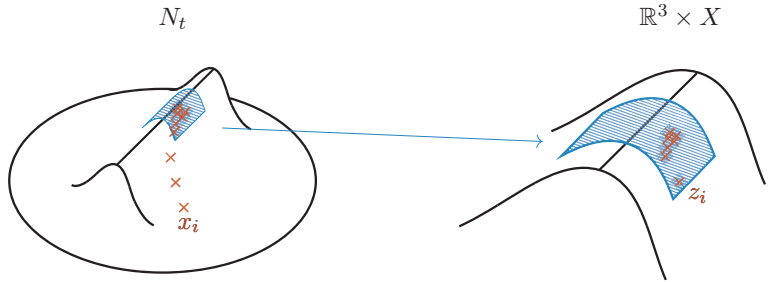}
  \end{center}
  \caption{Blowup analysis near the associative is, by means of the map $s^P$, reduced to the analysis of the pull-back of the ASD instanton defined by $s(\sigma(y^*))$ to $\R^3 \times {\XEH}$.}
  \label{figure:blow-up-2-gauge}
 \end{figure}
 
 For $\underline{a}_i=(\xi_i,a_i) \in (\Omega^0 \oplus \Omega^1)(N_t, \Ad E_t)$, let
 \begin{align*}
  \underline{b}_i
  :=
  \left(
  s_{1,\sigma(x_i);t_i}^{P,\sqrt{t_i},\sqrt{t_i}}(\xi_i),
  s_{1,\sigma(x_i);t_i}^{P,\sqrt{t_i},\sqrt{t_i}}(a_i)
  \right).
 \end{align*}
 \Cref{proposition:blowup-comparison} then gives
 \begin{align*}
  \|{
   \underline{b}_i
  }_{C^{1,\alpha}_{-1+\delta}(U^P_{1/\sqrt{t_i},1/\sqrt{t_i}})}
  \leq c
  \text{ and }
  \lim_{i \rightarrow \infty}
  \|{
   L_{p_{\XEH}^* f_* s(\sigma(x_i))}
   \underline{b}_i
   }_{C^{0,\alpha}_{-2+\delta}}
   =0.
 \end{align*}
 Without loss of generality we can assume $\sigma(x_i)\rightarrow y^* \in L$.
 By a diagonal argument and the Arzelà–Ascoli theorem, we have $\underline{b}_i \rightarrow \underline{b}^* \in (\Omega^0 \oplus \Omega^1)(\R^3 \times {\XEH}, \Ad p_{\XEH}^*f_* s(\sigma(y^*)))$ in $C^{1,\alpha/2}_{\loc}$, satisfying $L_{p_{\XEH}^*f_*s(\sigma(y^*))} \underline{b}^*=0$.
 \Cref{proposition:constant-in-r3-direction} implies that $\underline{b}^*=p_{\XEH}^* \underline{c}$, for some $\underline{c} \in \Ker \delta_{f_*s(\sigma(y^*)} \subset \Omega^1({\XEH}, f_*s(\sigma(y^*)))$.
 (Here, $\delta$ is the linearisation of the ASD equation as defined in \cref{equation:asd-instanton-linearisation}.)
 \Cref{equation:pi-a_i-to-0} then implies that $\underline{c}=0$.
 
 This contradicts \cref{equation:contradiction-proof-nonzero-property} as follows:
 denote by $(z_i) \subset \R^3 \times {\XEH}$ the sequence corresponding to $(x_i)$ under the map $s_{1,t_i;\sigma(x_i)}^{\sqrt{t},1/\sqrt{t}}$.
 Then $(z_i)$ is a bounded sequence, as the $\R^3$-coordinate of all $z_i$ is $0$, and the ${\XEH}$-coordinates are bounded by the assumption that $\lim_{i \rightarrow \infty} t_i^{-1} r_{t_i} (x_i) < \infty$.
 Thus we can assume without loss of generality that $z_i \rightarrow z^* \in \R^3 \times {\XEH}$, and find that
 \begin{align*}
  w(z^*)^{1-\delta}
  \left| \underline{b}^*(z^*) \right|
  =
  \lim_{i \rightarrow \infty}
  w^\nu_{l,\delta;t}(z_i)^{1-\delta}
  \left| \underline{b}_i(z_i) \right|
  \geq
  \frac{1}{c}
 \end{align*}
 by \cref{proposition:blowup-comparison}, which is a contradiction to $\underline{b}^*=0$.
 
 \textbf{Case 3.}
 ``The sequence does not go to zero on the neck'', i.e. there exists $y_i \in N_{t_i}$ such that $r_{t_i}(y_i) \rightarrow 0$, $t_i^{-1} r_{t_i}(y_i) \rightarrow \infty$, but $w_{-1,\delta;t}(y_i) \left| \underline{a}_i \right| (y_i) \nrightarrow 0$.
 
 Assume without loss of generality that this is the sequence $(x_i)$, i.e. $\lim_{i \rightarrow \infty} t_i^{-1} r_{t_i}(x_i) =\infty$ and $\lim_{i \rightarrow \infty} r_{t_i}(x_i)=0$
 (see \cref{figure:blow-up-3-gauge}).
 
 \begin{figure}[htbp]
  \begin{center}
   \includegraphics[width=10cm]{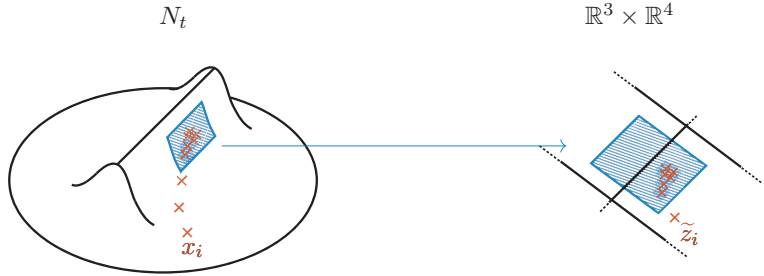}
  \end{center}
  \caption{Blowup analysis in the neck region is reduced to the analysis of the flat $G_2$-instanton defined on the pull-back of the framing at infinity defined by $s(\sigma(y^*))$ to $\R^3 \times \R^4$.}
  \label{figure:blow-up-3-gauge}
 \end{figure}
 
 Let 
 \begin{itemize}
 \item 
 $\epsilon_2^{(i)}$ such that $\epsilon_2^{(i)} \rightarrow 0$ and $\epsilon_2^{(i)}/ r_{t_i}(x_i) \rightarrow \infty$,
 
 \item
 $\epsilon_3^{(i)}$ such that $\epsilon_3^{(i)}/r_{t_i}(x_i) \rightarrow 0$ and $\epsilon_3^{(i)}/t_i \rightarrow \infty$.
 \end{itemize} 
 To ease notation, we write $\epsilon_2$ instead of $\epsilon_2^{(i)}$ and $\epsilon_3$ instead of $\epsilon_3^{(i)}$ in what follows.
 As before, write $\underline{a}_i=(\xi_i,a_i) \in (\Omega^0 \oplus \Omega^1)(N_t, \Ad E_t)$, and
 set
 \begin{align*}
  \underline{b}_i
  :=
  \left(
   \zeta_i,
   b_i
  \right)
  :=
  \left(
  s
  ^{\nu,\sqrt{t_i}, \epsilon_2, \epsilon_3}
  _{1,\sigma(x_i);t_i}
  (\xi_i)
  ,
  s
  ^{\nu,\sqrt{t_i}, \epsilon_2, \epsilon_3}
  _{1,\sigma(x_i);t_i}
  (a_i)
  \right)
 \end{align*}
 and denote by $(z_i)$ the sequence in $\R^3 \times \C^2/\{\pm 1\}$ corresponding to $(x_i)$ under the map 
 $s
  ^{\nu,\sqrt{t_i}, \epsilon_2, \epsilon_3}
  _{1,\sigma(x_i);t_i}$.
 \Cref{equation:contradiction-proof-nonzero-property} implies
 \begin{align}
 \label{equation:contradiction-proof-nonzero-property-case-3}
  |\underline{b}_i(z_i)| \cdot w(z_i) > c,
 \end{align}
 \cref{proposition:blowup-comparison-nu,equation:sequence-bounded-in-c-1-alpha} imply that
 \begin{align}
 \label{equation:sequence-bounded-in-c-1-alpha-case-3}
  \|{
   w^\nu_{l,\delta;t}
   s_{d,t;y}^{\nu,\epsilon_1, \epsilon_2, \epsilon_3}
   \underline{a}
   }
   _{C^{1,\alpha}_{0}(U_{1/\sqrt{t}, \epsilon_2/t, \epsilon_3/t;t}^\nu)}
   \leq
   c,
 \end{align}
 \cref{proposition:blowup-comparison-nu,equation:L-a_i-to-0} imply that
 \begin{align}
 \label{equation:L-a_i-to-0-part-3}
  \|{
   w^\nu_{l,\delta;t}
   L_{p_{\XEH}^* A_0}
    s_{1,t;y}^{\nu,\epsilon_1, \epsilon_2, \epsilon_3}
    \underline{a}
   }
   _{C^{1,\alpha}_{0}(U_{1/\sqrt{t}, \epsilon_2/t, \epsilon_3/t;t}^\nu)}
   \rightarrow
   0
   \text{ as }
   i \rightarrow \infty.
 \end{align}
 We will now conclude the argument as in case 2.
 The only difference is that, as it stands, the points $z_i$ tend to infinity.
 Because of this, we cannot directly conclude that a limit of $\underline{b}_i$ would be non-zero.
 That is why we rescale by $|z_i|$ first.
 By passing to a subsequence we can assume without loss of generality to be in case 3.1 or 3.2 as below:
 
 \textbf{Case 3.1.:}
 $|z_i| \leq 1/\sqrt{t_i}$.
 In this case let
 \begin{align}
  \tilde{\underline{b}}_i
  :=
  (\tilde{\zeta}_i,\tilde{b}_i)
  :=
  \left(
   |z_i|^{1-\delta}
   (\cdot |z_i|)^*\zeta_i,
   |z_i|^{-\delta}
   (\cdot |z_i|)^*b_i
  \right).
 \end{align}
 \Cref{equation:contradiction-proof-nonzero-property-case-3} implies $|\tilde{\underline{b}}_i(z_i/|z_i|)|\cdot r^{1-\delta}(z_i/|z_i|)=|\tilde{\underline{b}}_i(z_i/|z_i|)| > c$, and
 \cref{equation:sequence-bounded-in-c-1-alpha-case-3} implies that on the sets $B^3(0, 1/\sqrt{t}) \times [B^4(0, \epsilon_2 / |x_i|) \setminus B^4(0, \epsilon_3/|x_i|)]$, which exhaust $\R^3 \times (\C^2/\{\pm 1\} \setminus \{0\})$, we have:
 \begin{align}
 \label{equation:sequence-bounded-in-c-1-alpha-case-3-1}
  \|{
  \begin{cases}
  \tilde{\underline{b}}_i r^{1-\delta}
  & \text{ if }
  r \leq 1/(\sqrt{t} \cdot |z_i|)
  \\
  \tilde{\underline{b}}_i r^{1+\delta}t^\delta |z_i|^{2\delta}
  & \text{ if }
  r > 1/(\sqrt{t} \cdot |z_i|).
  \end{cases}
  }_{C^{1,\alpha}_0(B^3(0, 1/\sqrt{t}) \times [B^4(0, \epsilon_2 / |x_i|) \setminus B^4(0, \epsilon_3/|x_i|)])}
  &\leq
  c.
 \end{align}
 Here is how to arrive at the exponents of the weight function for $\tilde{\zeta}_i$ in the area $\{(u,v) \in \R^3 \times \C^2/\{\pm 1\}: r(v) > 1/(\sqrt{t} \cdot |z_i|) \}$:
 \begin{align*}
  \tilde{\zeta}_i r^{1+\delta}t^\delta |z_i|^{2\delta}
  &=
  (\cdot |z_i|)^*\zeta_i
  |z_i|^{1+\delta}
  r^{1+\delta}
  t^\delta
  \\
  &=
  (\cdot |z_i|)^*
  \left[ 
  \zeta_i r^{1+\delta} t^{\delta}
  \right],
 \end{align*}
 and $\zeta_i r^{1+\delta} t^{\delta}$ was bounded by \cref{equation:sequence-bounded-in-c-1-alpha-case-3}.
 The exponents of the weight function on the area $\{(u,v) \in \R^3 \times \C^2/\{\pm 1\}: r(v) > 1/(\sqrt{t} \cdot |z_i|) \}$ and also for the $1$-form part $\tilde{b}_i$ are computed analogously and precisely give \cref{equation:sequence-bounded-in-c-1-alpha-case-3-1}.
 Now, because of \cref{equation:sequence-bounded-in-c-1-alpha-case-3-1}, we can use the Arzelà-Ascoli theorem and a diagonal sequence argument to extract a limit $\underline{b}^*$ on $\R^3 \times (\C^2/\{\pm 1\} \setminus \{0\})$.
 We denote the pullback under the quotient map $\R^3 \times (\C^2 \setminus \{0\}) \rightarrow \R^3 \times (\C^2/\{\pm 1\} \setminus \{0\})$ by the same symbol and end up with a tensor $\underline{b}^*$ on $\R^3 \times (\C^2 \setminus \{0\})$.
 Again, by passing to a subsequence we can assume without loss of generality that we are in one of the following two cases:
 
 \textbf{Case 3.1.1:}
 $\sqrt{t_i} |z_i| \rightarrow 0$ as $i \rightarrow \infty$.
 
 In this case, the area $\{u \in \R^3 \times \C^2/\{\pm 1\}: r(u) > 1/(\sqrt{t} \cdot |z_i|) \}$ disappears as $i \rightarrow \infty$, and from \cref{equation:sequence-bounded-in-c-1-alpha-case-3-1} we get the estimate
 \begin{align}
 \label{equation:sequence-bounded-in-c-1-alpha-case-3-1-1}
  \|{
   \underline{b}^*
   r^{1-\delta}
  }_{C^{1,\alpha/2}_0(\R^3 \times (\R^4 \setminus \{0\}))}
  \leq
  c.
 \end{align}
 The element $\underline{b}^*$ defines a distribution on all of $\R^3 \times \C^2$ and is smooth by elliptic regularity, e.g. \cite[Theorem 6.33]{Folland1995}.
 We also get an $L^\infty$-bound for $\underline{b}^*$ as in the proof of \cite[Proposition 8.7]{Walpuski2013a}:
 away from $\R^3 \times \{0\}$, this is given by \cref{equation:sequence-bounded-in-c-1-alpha-case-3-1-1}.
 To see that $\underline{b}^*$ does not blow up in the $\R^3$-direction near $\R^3 \times \{0\}$, consider any $y \in \R^3 \times \{0\}$.
 Let $1 < p < -4/(-1+\delta)$, then $\|{\underline{b}^*}_{L^p(B_1(y))} \leq c$, independent of $y$, by \cref{equation:sequence-bounded-in-c-1-alpha-case-3-1-1}.
 So, by elliptic regularity $\|{\underline{b}^*}_{L_m^p(B_1(y))} \leq c$ for any $m \in \mathbb{N}$, and by the Sobolev embedding we have $\|{\underline{b}^*}_{L^\infty} \leq c$, where all of these estimates were independent of $y$.
 
 Thus, by \cref{proposition:constant-in-r3-direction} applied to the case $\tilde{X}=\C^2$, we get that $\underline{b}^*$ is independent of the $\R^3$-direction.
 Because of \cref{equation:L*L-equals-Laplacian-plus-something} we have that $\underline{b}^*$ is the pullback of a harmonic form of mixed degree (in degrees $0$ and $1$) on $\C^2$.
 So, $\underline{b}^*$ is harmonic and bounded on $\C^2$ by \cref{equation:sequence-bounded-in-c-1-alpha-case-3-1-1}, therefore vanishes by Liouville's theorem.
 That contradicts \cref{equation:contradiction-proof-nonzero-property-case-3}.
 
 \textbf{Case 3.1.2:}
 $\sqrt{t_i} |z_i| \rightarrow \kappa \in (0,\infty)$ as $i \rightarrow \infty$.
 
 In this case, after passing to a subsequence, \cref{equation:sequence-bounded-in-c-1-alpha-case-3-1} gives the estimate
 \begin{align}
 \label{equation:sequence-bounded-in-c-1-alpha-case-3-1-2}
  \|{
  \begin{cases}
  \underline{b}^*
  r^{1-\delta}
  & \text{ if }
  r \leq 1/\kappa
  \\
  \underline{b}^*
  r^{1+\delta}
  & \text{ if }
  r > 1/\kappa.
  \end{cases}
  }_{C^{1,\alpha}_0(\R^3 \times (\C^2 \setminus \{0\})}
  &\leq
  c.
 \end{align}
 Here is how to obtain this estimate:
 the assumption $\sqrt{t_i} |z_i| \rightarrow \kappa$ implies that $\sqrt{t_i} |z_i| > c$, at least up to a subsequence.
 Thus, we have $t^\delta \cdot |z_i|^{2\delta}<c$, and \cref{equation:sequence-bounded-in-c-1-alpha-case-3-1} becomes
 \begin{align*}
  \|{
  \begin{cases}
  \tilde{\underline{b}}_i r^{1-\delta}
  & \text{ if }
  r \leq 1/(\sqrt{t} \cdot |z_i|)
  \\
  \tilde{\underline{b}}_i r^{1+\delta}
  & \text{ if }
  r > 1/(\sqrt{t} \cdot |z_i|).
  \end{cases}
  }_{C^{1,\alpha}_0(B^3(0, 1/\sqrt{t}) \times [B^4(0, \epsilon_2 / |x_i|) \setminus B^4(0, \epsilon_3/|x_i|)])}
  &\leq
  c.
 \end{align*}
 Here, taking the limit $i \rightarrow \infty$ gives \cref{equation:sequence-bounded-in-c-1-alpha-case-3-1-2}.
 In this case, we arrive at a contradiction as in case 3.1.1.
 
 \textbf{Case 3.2.:}
 $|z_i| > 1/\sqrt{t_i}$.
 In this case let
 \begin{align}
  \tilde{\underline{b}}_i
  :=
  (\tilde{\zeta}_i,\tilde{b}_i)
  :=
  \left(
   t^\delta
   |z_i|^{1+\delta}
   (\cdot |z_i|)^*\zeta_i,
   t^\delta
   |z_i|^{\delta}
   (\cdot |z_i|)^*b_i
  \right).
 \end{align}
 This gives us the following analogue of \cref{equation:sequence-bounded-in-c-1-alpha-case-3-1}:
 \begin{align}
 \label{equation:sequence-bounded-in-c-1-alpha-case-3-2}
  \|{
  \begin{cases}
  \tilde{\underline{b}}_i r^{1-\delta}
  t^{-\delta} |z_i|^{-2\delta}
  & \text{ if }
  r \leq 1/(\sqrt{t} \cdot |z_i|)
  \\
  \tilde{\underline{b}}_i r^{1+\delta}
  & \text{ if }
  r > 1/(\sqrt{t} \cdot |z_i|).
  \end{cases}
  }_{C^{1,\alpha}_0(B^3(0, 1/\sqrt{t}) \times [B^4(0, \epsilon_2 / |x_i|) \setminus B^4(0, \epsilon_3/|x_i|)])}
  &\leq
  c.
 \end{align}
 We can extract a limit $\underline{b}^*$ as in case 3.1 and are, without loss of generality, in one of the following two cases:
 
 \textbf{Case 3.2.1:}
 $\sqrt{t_i} \cdot |z_i| \rightarrow \infty$ as $i \rightarrow \infty$.
 In this case we have the estimate
 \begin{align}
 \label{equation:sequence-bounded-in-c-1-alpha-case-3-2-1}
  \|{
   \underline{b}^*
   r^{1+\delta}
  }_{C^{1,\alpha/2}_0(\R^3 \times (\R^4 \setminus \{0\}))}
  \leq
  c
 \end{align}
 and arrive at a contradiction as in case 3.1.1.
 
 \textbf{Case 3.2.2:}
 $\sqrt{t_i} \cdot |z_i| \rightarrow \kappa \in (0,\infty)$ as $i \rightarrow \infty$.
 In this case we have exactly \cref{equation:sequence-bounded-in-c-1-alpha-case-3-1-2} and can conclude the proof as in case 3.1.2. 
\end{proof}

\subsubsection{Cross-term Estimates}
\label{subsubsection:cross-term-estimates}

In the beginning of \cref{subsection:linear-estimates} we explained the idea for the proof of the linear estimate.
Namely, we want to separately consider two parts of the linearisation of the instanton equation:
the first part near the resolution locus of the associative $L$, which is approximately equal to the linearisation of the Fueter equation.
The second part is the linearised operator modulo deformations of the Fueter section.
These parts were estimated in \cref{subsubsection:comparison-with-the-fueter-operator,subsubsection:estimate-of-eta-a}.

However, it is not true that the linearised instanton operator neatly decomposes as a sum of these two operators.
We can take a deformation of the Fueter section, apply $L_t$ to it, and then project it onto the part that does \emph{not} come from a deformation of the Fueter section.
In an ideal world, $L_t$ near the resolution locus of the associative is exactly equal to the linearisation of the Fueter equation and the result is $0$.
In reality, we do not have that the result is $0$, but we have that it is small.
That is \cref{equation:cross-term-1}.
There is also, roughly speaking, the converse of this, which is \cref{equation:cross-term-2}.

Like the results from \cref{subsubsection:comparison-with-the-fueter-operator}, this proposition has been proved in a slightly different setting in \cite{Walpuski2017}.
Again, the proof given therein carries over to our situation if we only have that
$\tilde{\psi}^N_t-\psi^P_t$
is small, which is true on resolutions of $T^7/\Gamma$ by \cref{proposition:psi-N-psi-P-comparison-on-T7,corollary:kummer-construction-simplified-torsion-free-estimate}.

\begin{proposition}[Proposition 8.29 in \cite{Walpuski2017}]
\label{proposition:cross-term-estimates}
 Let $N_t$ be the resolution of $T^7/\Gamma$ from \cref{section:torsion-free-structures-on-the-generalised-kummer-construction}.
 There exists a constant $c>0$ such that for all $t \in (0,T)$ we have
 \begin{align}
 \label{equation:cross-term-1}
  \|{
   \eta_t L_t \iota_t v
  }_{C^{0,\alpha}_{-2,0;t}}
  &\leq
  c t^{2-\alpha}
  \|{
   v
  }_{C^{1,\alpha}}
 \intertext{%
 as well as}
  \label{equation:cross-term-2}
  \|{
   \pi_t L_t \eta_t \underline{a}
  }_{C^{0,\alpha}}
  &\leq
  c t^{-\alpha}
  \|{
   \eta_t \underline{a}
  }_{C^{1,\alpha}_{-1,0;t}}.
 \end{align}
\end{proposition}

\subsubsection{Proof of Proposition \ref{proposition:inverse-operator-estimate-xy-norm}}
\label{subsubsection:proof-of-linear-estimate-proposition}

\begin{proof}
 Assume that the claim does not hold, and let $t_i \rightarrow 0$, $\underline{a}_i \in (\Omega^0 \oplus \Omega^1)(N_t,\Ad E_t)$ such that $\|{ \underline{a}_i}_{\mathfrak{X}_t}=1$, but $\|{ L_t \underline{a}_i }_{\mathfrak{Y}_t} \rightarrow 0$.
 
 We first prove that
 \begin{align}
 \label{equation:linear-estimate-xy-norm-eta-claim}
 t_i^{-\delta/2} \|{ \eta_{t_i} \underline{a}_i}_{C^{1,\alpha}_{-1,\delta;t_i}} \rightarrow 0.
 \end{align}
 
 We have that
 \begin{align*}
  \|{
   \eta_{t_i} \underline{a}_i
  }_{C^{1,\alpha}_{-1,\delta;t_i}}
  &\leq
  \|{
   L_{t_i} \eta_{t_i} \underline{a}_i
  }_{C^{0,\alpha}_{-2,\delta;t_i}}
  \\
  &\leq
  \|{
   \eta_{t_i}L_{t_i} \eta_{t_i} \underline{a}_i
  }_{C^{0,\alpha}_{-2,\delta;t_i}}
  +
  \|{
   \overline{\pi}_{t_i}L_{t_i} \eta_{t_i} \underline{a}_i
  }_{C^{0,\alpha}_{-2,\delta;t_i}}  
  \\
  &\leq
   \|{
    \eta_{t_i} L_t \underline{a}
   }_{C^{0,\alpha}_{-2,\delta;t_i}}
   +
   \|{
    \eta_{t_i}
    L_{t_i}
    \overline{\pi}_{t_i}
    \underline{a}_i
   }_{C^{0,\alpha}_{-2,\delta;t_i}}
   +
   \|{
    \overline{\pi}_{t_i}
    L_{t_i}
    \eta_{t_i}
    \underline{a}_i
   }_{C^{0,\alpha}_{-2,\delta;t_i}}
   \\
  &\leq
   \|{
    \eta_{t_i} L_t \underline{a}
   }_{C^{0,\alpha}_{-2,\delta;t_i}}
   +
   \|{
    1
   }_{C^{0,\alpha}_{0,\delta;t_i}}
   \|{
    \eta_{t_i}
    L_{t_i}
    \overline{\pi}_{t_i}
    \underline{a}_i
   }_{C^{0,\alpha}_{-2,0;t_i}}
   +
   t^{1-\alpha}
   \|{
    \pi_{t_i}
    L_{t_i}
    \eta_{t_i}
    \underline{a}_i
   }_{C^{0,\alpha}}
   \\
   &\leq
   c
   \left(
    \|{
     \eta_{t_i} L_t \underline{a}
    }_{C^{0,\alpha}_{-2,\delta;t_i}}
    +
    ct^{\delta/2} t^{2-\alpha}
    \|{
     \pi_t \underline{a}_i
    }_{C^{1,\alpha}}
    +
    t^{1-2\alpha}
    \|{
     \eta_{t_i} \underline{a}_i
    }_{C^{1,\alpha}_{-1,0;t}}
   \right)
   \\
   &\leq
   c
   \left(
    \|{
     \eta_{t_i} L_t \underline{a}
    }_{C^{0,\alpha}_{-2,\delta;t_i}}
    +
    \mathcal{O}(t^{\delta/2+1-\alpha})
    +
    \mathcal{O}(t^{1-2\alpha+\delta/2})
   \right)
 \end{align*}
 where we used \cref{proposition:crucial-proposition} in the first step;
 we used $\overline{\pi}_{t_i} + \eta_{t_i}=1$ in the second and third steps;
 \cref{proposition:norm-weights-basic-estimate,proposition:iota-pi-bounds} in the fourth step;
 and \cref{proposition:cross-term-estimates} together with
 $\|{
    1
   }_{C^{0,\alpha}_{0,\delta;t_i}}
   \leq ct^{\delta/2}$
 in the fifth step.
 Multiplying the last line with $t_i^{-\delta/2}$, the last two summands tend to zero as they are bounded by positive powers of $t$.
 The first summand tends to zero by the assumption $\|{ L_t \underline{a}_i }_{\mathfrak{Y}_t} \rightarrow 0$.
 
 It remains to prove that
 \begin{align}
 t_i \|{ \pi_{t_i} \underline{a}_i}_{C^{1,\alpha}} \rightarrow 0.
 \end{align}
 
 We have that
 \begin{align*}
  \lim_{i \rightarrow \infty}
  {t_i}
  \|{
   \pi_{t_i} \underline{a}_i
  }_{C^{1,\alpha}}
  &\leq
  \lim_{i \rightarrow \infty}
  {t_i}
  \|{
   \pi_{t_i} L_{t_i} \iota_{t_i} \pi_{t_i} \underline{a}_i
  }_{C^{0,\alpha}}
  \\
  &\leq
  \lim_{i \rightarrow \infty}
  t_i
  \left(
  \|{
   \pi_t L_t \underline{a}
  }_{C^{0,\alpha}}
  +
  \|{
   \pi_t L_t \eta_t \underline{a}
  }_{C^{0,\alpha}}
  \right)
  \\
  &\leq
  \lim_{i \rightarrow \infty}
  t_i
  \left(
  \|{
   \pi_t L_t \underline{a}
  }_{C^{0,\alpha}}
  +
  ct^{-\alpha}
  \|{
   \eta_t \underline{a}
  }_{C^{1,\alpha}_{-1,0;t}}
  \right).
 \end{align*}
 where we used \cref{proposition:pi-L-iota-inequality} in the first step, $\overline{\pi}_{t_i} + \eta_{t_i}=1$ in the second step, \cref{proposition:cross-term-estimates} in the third step.
 Here, the second summand tends to zero by \cref{equation:linear-estimate-xy-norm-eta-claim}, and the first summand tends to zero by the assumption $\|{ L_t \underline{a}_i }_{\mathfrak{Y}_t} \rightarrow 0$.  
 Altogether, $\|{\underline{a}_i}_{\mathfrak{X}_t} \rightarrow 0$, which is a contradiction.
\end{proof}

\subsection{Quadratic Estimate}
\label{subsection:quadratic-estimate}

We state an estimate for the quadratic form $Q_t$ from \cref{equation:g2-instanton-equation}, where we denote its associated bilinear form by the same symbol.
This statement is taken from \cite{Walpuski2017} and the proof can be directly adapted to our slightly different setting.

\begin{proposition}[Proposition 9.1 in \cite{Walpuski2017}]
\label{proposition:quadratic-estimate}
 There exists a constant $c>0$ such that for $t \in (0,1)$ we have
 \begin{align}
  \begin{split}
   &
   \|{
    \eta_t Q_t(\underline{a}_1,\underline{a}_2)
   }_{C^{0,\alpha}_{-2,\delta;t}}
   \\
   &
   \;\;\;\;\;\;\;\;\;\;\;\;
   \leq
   c t^{-\alpha}
   \left(
    \|{
     \eta_t \underline{a}_1
    }_{C^{0,\alpha}_{-1,\delta;t}}
    \cdot 
    \|{
     \eta_t \underline{a}_2
    }_{C^{0,\alpha}_{-1,\delta;t}}
    +
    \|{
     \eta_t \underline{a}_1
    }_{C^{0,\alpha}_{-1,\delta;t}}
    \cdot 
    \|{
     \pi_t \underline{a}_2
    }_{C^{0,\alpha}}
    \right.
    \\
    &
    \;\;\;\;\;\;\;\;\;\;\;\;
    \;\;\;\;\;\;\;\;
    \left.
    +
    \|{
     \pi_t \underline{a}_1
    }_{C^{0,\alpha}}
    \cdot 
    \|{
     \eta_t \underline{a}_2
    }_{C^{0,\alpha}_{-1,\delta;t}}
    +
    \|{
     \pi_t \underline{a}_1
    }_{C^{0,\alpha}}
    \cdot 
    \|{
     \pi_t \underline{a}_2
    }_{C^{0,\alpha}}
   \right)    
  \end{split}
 \end{align}
 and
 \begin{align}
  \begin{split}
   &t \|{
    \pi_t Q_t(\underline{a}_1,\underline{a}_2)
   }_{C^{0,\alpha}}
   \\
   &
   \;\;\;\;\;\;\;\;\;\;\;\;
   \leq
   c t^{-\alpha}
   \left(
    \|{
     \eta_t \underline{a}_1
    }_{C^{0,\alpha}_{-1,\delta;t}}
    \cdot 
    \|{
     \eta_t \underline{a}_2
    }_{C^{0,\alpha}_{-1,\delta;t}}
    +
    \|{
     \eta_t \underline{a}_1
    }_{C^{0,\alpha}_{-1,\delta;t}}
    \cdot 
    \|{
     \pi_t \underline{a}_2
    }_{C^{0,\alpha}}
    \right.
    \\
    &
    \;\;\;\;\;\;\;\;\;\;\;\;
    \;\;\;\;\;\;\;\;
    \left.
    +
    \|{
     \pi_t \underline{a}_1
    }_{C^{0,\alpha}}
    \cdot 
    \|{
     \eta_t \underline{a}_2
    }_{C^{0,\alpha}_{-1,\delta;t}}
    +
    t
    \|{
     \pi_t \underline{a}_1
    }_{C^{0,\alpha}}
    \cdot 
    \|{
     \pi_t \underline{a}_2
    }_{C^{0,\alpha}}
   \right).
  \end{split}
 \end{align}

\end{proposition}

\subsection{Deforming to Genuine Solutions}
\label{subsection:deforming-to-genuine-solutions}

In this subsection we will complete the construction of $G_2$-instantons and show that in two favourable situations the $G_2$-instanton $\theta$ and the Fueter section $s$ can be glued together to a $G_2$-instanton on $N_t$.
The favourable situations are:
\begin{enumerate}
 \item 
 The Fueter section is a section of rigid ASD-instantons (cf. \cref{theorem:instanton-existence}).
 This implies, in particular, that the Fueter section is infinitesimally rigid.
 In this case the map $\pi_t$ from \cref{definition:pi-iota-splitting} is just the zero map, which leads to better estimates of the quadratic part $Q_t$ of the instanton equation.

 \item
 We are in the special situation of \cref{equation:pregluing-estimate-on-T7}, where we resolved the orbifold $T^7/\Gamma$.
\end{enumerate}

The main reason we are confined to these two favourable scenarios is the following:
in \cref{proposition:pregluing-estimate,corollary:pregluing-estimate-on-T7} we proved a pregluing estimate with a good power of $t^{1/18}$ in the general case and a good power of $t^2$ in the case of $T^7/\Gamma$, roughly speaking.
In \cref{proposition:quadratic-estimate} we stated an estimate for the quadratic part of the instanton operator which in particular implies
\begin{align*}
 \|{
  Q_t (\underline{a}_1,\underline{a}_2)
 }_{\mathfrak{Y}}
 \leq
 t^{-2-\alpha-\delta/2}
 \|{
  \underline{a}_1
 }_{\mathfrak{X}}
 \|{
  \underline{a}_2
 }_{\mathfrak{X}}.
\end{align*}
To apply the inverse function theorem, we would need the bad power $t^{-2-\alpha-\delta/2}$ from this estimate to be absorbed by the good power from the pregluing estimate, but the pregluing estimate is only good enough to do this in the case of the orbifold $T^7/\Gamma$.
If the Fueter section is actually the constant section of a rigid ASD-instanton, then we have a better estimate for the quadratic part of the instanton equation.

\begin{theorem}
 \label{theorem:instanton-existence}
 Assume that the section $s$ is given by a rigid ASD-instanton in every point $x \in L$, and assume that the connection $\theta$ used to define the approximate $G_2$-instanton $A_t$ from \cref{proposition:approximate-solution} is infinitesimally rigid.
 
 There exists $c>0$ such that for small $t$ there exists $\underline{a}_t=(a_t, \xi_t) \in C^{1,\alpha}(\Omega^0 \oplus \Omega^1(\Ad E_t))$ such that $\tilde{A}_t := A_t + a_t$ is a $G_2$-instanton.
 Furthermore, $\underline{a}_t$ satisfies $\|{ \underline{a}_t }_{C^{1,\alpha}_{-1,\delta;t}} \leq ct^{1/18}$.
\end{theorem}

\begin{theorem}
 \label{theorem:instanton-existence3}
 Let $N \rightarrow Y'$ be the resolution of the orbifold $Y'=T^7/\Gamma$ from before.
 Assume that the connection $\theta$ used to define the approximate $G_2$-instanton $A_t$ from \cref{proposition:approximate-solution} is infinitesimally rigid and that $s$ is an infinitesimally rigid Fueter section.
 
 There exists $c>0$ such that for small $t$ there exists an $\underline{a}_t=(a_t, \xi_t) \in C^{1,\alpha}(\Omega^0 \oplus \Omega^1(\Ad E_t))$ such that $\tilde{A}_t := A_t + a_t$ is a $G_2$-instanton.
 Furthermore, $\underline{a}_t$ satisfies $\|{ \underline{a}_t }_{\mathfrak{X}_t} \leq ct^{2-2\alpha}$.
\end{theorem}

The proof of the theorems will use the following lemma:

\begin{lemma}[Lemma 7.2.23 in \cite{Donaldson1990}]
 \label{lemma:banach-space-estimate-lemma}
 Let $X$ be a Banach space and let $T: X \rightarrow X$ be a smooth map with $T(0)=0$.
 Suppose there is a constant $c>0$ such that
 \begin{align*}
  \|{
   Tx-Ty
  }
  \leq
  c
  (
   \|{
    x
   }
   +
   \|{
    y
   }
  )
  \|{
   x-y
  }.
 \end{align*}
 Then if $y \in X$ satisfies $\|{ y } \leq \frac{1}{10c}$, there exists a unique $x \in X$ with $\|{ x } \leq \frac{1}{5c}$ solving
 \begin{align*}
  x+Tx=y.
 \end{align*}
 The unique solution satisfies $\|{ x} \leq 2 \|{ y }$.
\end{lemma}

\begin{proof}[Proof of \cref{theorem:instanton-existence}]
 In the case that $s$ is a section of rigid ASD instantons, we have that the projection map $\pi_t$ is zero.
 Therefore, \cref{proposition:schauder-estimate,proposition:crucial-proposition} give
 \begin{align}
 \label{equation:linear-estimate-for-pointwise-rigid}
  \|{
   \underline{a}
   }_
   {C^{1,\alpha}_{-1,\delta;t}}
  \leq
  c
   \|{
    L_t \underline{a}
   }_
   {C^{0,\alpha}_{-2,\delta;t}}.
 \end{align}
 This means that
 \[L_t: C^{1,\alpha}((\Omega^0 \oplus \Omega^1)(N_t, \Ad E_t)) \rightarrow C^{0,\alpha}((\Omega^0 \oplus \Omega^1)(N_t, \Ad E_t))\]
 is injective.
 Because $L_t$ is formally self-adjoint, it is also bijective.
 Denote its inverse by $L_t^{-1}$.
 Furthermore, using $\pi_t=0$, and therefore $\eta_t=\Id$, \cref{proposition:quadratic-estimate} gives
 \begin{align}
 \label{equation:quadratic-estimate-for-pointwise-rigid}
   \|{
    Q_t(\underline{a}_1,\underline{a}_2)
   }_{C^{0,\alpha}_{-2,\delta;t}}
   \leq
   c t^{-\alpha}
    \|{
     \underline{a}_1
    }_{C^{0,\alpha}_{-1,\delta;t}}
    \cdot 
    \|{
     \underline{a}_2
    }_{C^{0,\alpha}_{-1,\delta;t}}.
 \end{align}
 Set $T_t := Q_t \circ L_t^{-1}$.
 We then have
 \begin{align*}
  \|{
  T_t(\underline{b}_1)-
  T_t(\underline{b}_2)
  }_{C^{0,\alpha}_{-2,\delta;t}}
  &=
  \|{
  Q(
  L^{-1}\underline{b}_1-L^{-1}\underline{b}_2,
  L^{-1}\underline{b}_1+L^{-1}\underline{b}_2
  )
  }_{C^{0,\alpha}_{-2,\delta;t}}
  \\
  &\leq
  ct^{-\alpha}
  \|{L^{-1}\underline{b}_1-L^{-1}\underline{b}_2}
  _{C^{0,\alpha}_{-1,\delta;t}}
  \|{L^{-1}\underline{b}_1+L^{-1}\underline{b}_2}
  _{C^{0,\alpha}_{-1,\delta;t}}
  \\
  &\leq
  ct^{-\alpha}
  \|{L^{-1}\underline{b}_1-L^{-1}\underline{b}_2}
  _{C^{1,\alpha}_{-1,\delta;t}}
  \|{L^{-1}\underline{b}_1+L^{-1}\underline{b}_2}
  _{C^{1,\alpha}_{-1,\delta;t}}
  \\
  &\leq
  ct^{-\alpha}
  \|{\underline{b}_1-\underline{b}_2}
  _{C^{0,\alpha}_{-2,\delta;t}}
  \left(
  \|{\underline{b}_1}
  _{C^{0,\alpha}_{-2,\delta;t}}
  +
  \|{\underline{b}_1}
  _{C^{0,\alpha}_{-2,\delta;t}}
  \right),
 \end{align*}
 where we used \cref{equation:quadratic-estimate-for-pointwise-rigid} in the first inequality and \cref{equation:linear-estimate-for-pointwise-rigid} in the last inequality.
 
 For $e_t$ we have
 \[
  \|{
   e_t
  }_{C^{0,\alpha}_{-2,0;t}}
  \leq
  c t^{1/18}
 \]
 by \cref{proposition:pregluing-estimate}.
 For small $t$, we have that
 $t^{1/18}< \left( t^{-\alpha+\delta/2} \right)^{-1}$ due to our choices of $\alpha$ and $\delta$ in \cref{definition:hoelder-norms-for-gauge-with-cases}.
 Thus, by applying \cref{lemma:banach-space-estimate-lemma} to the map $T_t$, we get a solution $\underline{b}_t$ to the equation $\underline{b}_t+T_t(\underline{b}_t)=-e_t$ for small $t$, satisfying the estimate $\|{ \underline{b}_t }_{C^{0,\alpha}_{-2,0;t}} \leq ct^{1/18}$.
 
 Letting $\underline{a}_t := L_t^{-1}(\underline{b}_t)$, this means precisely $L_t(\underline{a}_t)+Q_t(\underline{a}_t)=-e_t$, so $\tilde{A}_t=A_t+a_t$ is a $G_2$-instanton, and $\underline{a}_t$ satisfies $\|{ \underline{a}_t }_{C^{1,\alpha}_{-1,\delta;t}} \leq ct^{1/18}$ by \cref{equation:linear-estimate-for-pointwise-rigid}, which proves the claim.
\end{proof}

\begin{proof}[Proof of \cref{theorem:instanton-existence3}]
 As in the proof of \cref{theorem:instanton-existence}, set $T_t := Q_t \circ L_t^{-1}$.
 Then
 \begin{align*}
  &
  \|{
   T_t(\underline{b}_1)-T_t (\underline{b}_2)
  }_{\mathfrak{Y}_t}
  \\
  &=
  \|{
   Q(
   L^{-1}\underline{b}_1-L^{-1}\underline{b}_2, 
   L^{-1}\underline{b}_1+L^{-1}\underline{b}_2
   )
  }_{\mathfrak{Y}_t}
  \\
  &=
  t^{-\delta/2}
  \|{
   \eta_t
   Q(
   L^{-1}\underline{b}_1-L^{-1}\underline{b}_2, 
   L^{-1}\underline{b}_1+L^{-1}\underline{b}_2
   )
  }_{C^{0,\alpha}_{-2,\delta;t}}
  \\
  &
  \;\;\;\;\;\;
  +
  t
  \|{
   \pi_t
   Q(
   L^{-1}\underline{b}_1-L^{-1}\underline{b}_2, 
   L^{-1}\underline{b}_1+L^{-1}\underline{b}_2
   )
  }_{C^{0,\alpha}}
  \\
  &\leq
  c
  t^{-\alpha-\delta/2}
   \left(
    \|{
     \eta_t L^{-1}(\underline{b}_1-\underline{b}_2)
    }_{C^{0,\alpha}_{-1,\delta;t}}
    \cdot 
    \|{
     \eta_t L^{-1}(\underline{b}_1+\underline{b}_2)
    }_{C^{0,\alpha}_{-1,\delta;t}}
    \right.
    \\
    &
    \;\;\;\;\;\;\;\;\;\;\;\;
    \;\;\;\;\;\;\;\;
    \left.
    +
    \|{
     \eta_t L^{-1}(\underline{b}_1-\underline{b}_2)
    }_{C^{0,\alpha}_{-1,\delta;t}}
    \cdot 
    \|{
     \pi_t L^{-1}(\underline{b}_1+\underline{b}_2)
    }_{C^{0,\alpha}}
    \right.
    \\
    &
    \;\;\;\;\;\;\;\;\;\;\;\;
    \;\;\;\;\;\;\;\;
    \left.
    +
    \|{
     \pi_t L^{-1}(\underline{b}_1-\underline{b}_2)
    }_{C^{0,\alpha}}
    \cdot 
    \|{
     \eta_t L^{-1}(\underline{b}_1+\underline{b}_2)
    }_{C^{0,\alpha}_{-1,\delta;t}}
    \right.
    \\
    &
    \;\;\;\;\;\;\;\;\;\;\;\;
    \;\;\;\;\;\;\;\;
    \left.
    +
    \|{
     \pi_t L^{-1}(\underline{b}_1-\underline{b}_2)
    }_{C^{0,\alpha}}
    \cdot 
    \|{
     \pi_t L^{-1}(\underline{b}_1+\underline{b}_2)
    }_{C^{0,\alpha}}
   \right)
   \\
   &
    \;\;\;\;\;\;
   +
  c
  t^{-\alpha}
   \left(
    \|{
     \eta_t L^{-1}(\underline{b}_1-\underline{b}_2)
    }_{C^{0,\alpha}_{-1,\delta;t}}
    \cdot 
    \|{
     \eta_t L^{-1}(\underline{b}_1+\underline{b}_2)
    }_{C^{0,\alpha}_{-1,\delta;t}}
    \right.
    \\
    &
    \;\;\;\;\;\;\;\;\;\;\;\;
    \;\;\;\;\;\;\;\;
    \left.
    +
    \|{
     \eta_t L^{-1}(\underline{b}_1-\underline{b}_2)
    }_{C^{0,\alpha}_{-1,\delta;t}}
    \cdot 
    \|{
     \pi_t L^{-1}(\underline{b}_1+\underline{b}_2)
    }_{C^{0,\alpha}}
    \right.
    \\
    &
    \;\;\;\;\;\;\;\;\;\;\;\;
    \;\;\;\;\;\;\;\;
    \left.
    +
    \|{
     \pi_t L^{-1}(\underline{b}_1-\underline{b}_2)
    }_{C^{0,\alpha}}
    \cdot 
    \|{
     \eta_t L^{-1}(\underline{b}_1+\underline{b}_2)
    }_{C^{0,\alpha}_{-1,\delta;t}}
    \right.
    \\
    &
    \;\;\;\;\;\;\;\;\;\;\;\;
    \;\;\;\;\;\;\;\;
    \left.
    +
    t
    \|{
     \pi_t L^{-1}(\underline{b}_1-\underline{b}_2)
    }_{C^{0,\alpha}}
    \cdot 
    \|{
     \pi_t L^{-1}(\underline{b}_1+\underline{b}_2)
    }_{C^{0,\alpha}}
   \right)
  \\
  &\leq
  c
  t^{-\alpha-2-\delta/2}
   \left(
    t^{-\delta}
    \|{
     \eta_t L^{-1}(\underline{b}_1-\underline{b}_2)
    }_{C^{0,\alpha}_{-1,\delta;t}}
    \cdot 
    \|{
     \eta_t L^{-1}(\underline{b}_1+\underline{b}_2)
    }_{C^{0,\alpha}_{-1,\delta;t}}
    \right.
    \\
    &
    \;\;\;\;\;\;\;\;\;\;\;\;
    \;\;\;\;\;\;\;\;
    \left.
    +
    t^{1-\delta/2}
    \|{
     \eta_t L^{-1}(\underline{b}_1-\underline{b}_2)
    }_{C^{0,\alpha}_{-1,\delta;t}}
    \cdot 
    \|{
     \pi_t L^{-1}(\underline{b}_1+\underline{b}_2)
    }_{C^{0,\alpha}}
    \right.
    \\
    &
    \;\;\;\;\;\;\;\;\;\;\;\;
    \;\;\;\;\;\;\;\;
    \left.
    +
    t^{1-\delta/2}
    \|{
     \pi_t L^{-1}(\underline{b}_1-\underline{b}_2)
    }_{C^{0,\alpha}}
    \cdot 
    \|{
     \eta_t L^{-1}(\underline{b}_1+\underline{b}_2)
    }_{C^{0,\alpha}_{-1,\delta;t}}
    \right.
    \\
    &
    \;\;\;\;\;\;\;\;\;\;\;\;
    \;\;\;\;\;\;\;\;
    \left.
    +
    t^2
    \|{
     \pi_t L^{-1}(\underline{b}_1-\underline{b}_2)
    }_{C^{0,\alpha}}
    \cdot 
    \|{
     \pi_t L^{-1}(\underline{b}_1+\underline{b}_2)
    }_{C^{0,\alpha}}
   \right)
   \\
   &\leq
   c
   t^{-\alpha-2-\delta/2}
   \|{
    L^{-1}(b_1-b_2)
   }_{\mathfrak{X}_t}
   \|{
    L^{-1}(b_1+b_2)
   }_{\mathfrak{X}_t}
   \\
   &\leq
   c
   t^{-\alpha-2-\delta/2}
   \|{
    b_1-b_2
   }_{\mathfrak{Y}_t}
   \|{
    b_1+b_2
   }_{\mathfrak{Y}_t}
   \\
   &\leq   
   c
   t^{-\alpha-2-\delta/2}
   \|{
    b_1-b_2
   }_{\mathfrak{Y}_t}
   \left(
   \|{
    b_1
   }_{\mathfrak{Y}_t}
   +
   \|{
    b_2
   }_{\mathfrak{Y}_t}
   \right).
 \end{align*}
 Here we used \cref{proposition:quadratic-estimate} in the third step, and \cref{proposition:inverse-operator-estimate-xy-norm} in the second to last step.
 
 We have
 \begin{align*}
  \|{
   e_t
  }_{\mathfrak{Y}_t}
  &\leq
  c t^{2-\alpha},
 \end{align*}
 by \cref{corollary:pregluing-estimate-on-T7}.
 Applying \cref{lemma:banach-space-estimate-lemma} as in the proof of \cref{theorem:instanton-existence} shows the claim. 
\end{proof}

\section{Examples}
\label{section:examples}

\subsection{Examples on the Resolution of $T^7/\Gamma$}
\label{subsection:examples-on-t7-mod-gamma}

In \cite{Joyce1996}, many examples of finite groups $\Gamma$ acting on $T^7$ and resolutions of the resulting $G_2$-orbifolds $T^7/\Gamma$ are explained.
In \cite{Walpuski2013a}, $G_2$-instantons on these resolutions were constructed.
These examples can immediately be reproduced using \cref{theorem:instanton-existence3} by choosing locally constant Fueter sections.
However, \cref{theorem:instanton-existence3} is more general in two ways.
\begin{enumerate}
\item
It allows non-constant Fueter sections, as long as they are rigid.
However, we found no example of such Fueter sections.

\item
It allows $\theta$ to have non-trivial monodromy along $L$.
Previously, no such examples have been constructed.
Making use of \cref{theorem:instanton-existence3}, we now explain a large number of examples in the simplest case of the Generalised Kummer Construction.
\end{enumerate}
Consider the group $\Gamma=\< \alpha,\beta,\gamma \>$ acting on $T^7$ defined by
\begin{align}
\label{equation:alpha-beta-gamma-t7}
\begin{split}
\alpha: 
(x_1,\dots,x_7) & \mapsto
(x_1,x_2,x_3,-x_4,-x_5,-x_6,-x_7),
\\
\beta: 
(x_1,\dots,x_7) & \mapsto
(x_1,-x_2,-x_3,x_4,x_5,\frac{1}{2}-x_6,-x_7),
\\
\gamma: 
(x_1,\dots,x_7) & \mapsto
(-x_1,x_2,-x_3,x_4,\frac{1}{2}-x_5,x_6,\frac{1}{2}-x_7).
\end{split}
\end{align}
The singular set $S \subset T^7/\Gamma$ consists of $12$ copies of $T^3$.
Let $p: \R^7 \rightarrow T^7/\Gamma$ be the quotient map.
Then, $p^{-1}(T^7/\Gamma \setminus S) \subset \R^7$ is a universal cover of $T^7/\Gamma \setminus S$ and we can identify the orbifold fundamental group of $T^7/\Gamma$ with the deck transformations of $p^{-1}(T^7/\Gamma \setminus S)$, namely
\[
 \pi_1^{\orb}(T^7/\Gamma)
 =
 \pi_1(T^7/\Gamma \setminus S)
 =
 \<
 \alpha, \beta, \gamma, \tau_1, \dots, \tau_7
 \>,
\]
where $\alpha, \beta, \gamma$ are the maps from \cref{equation:alpha-beta-gamma-t7} but viewed as maps on $\R^7$ by abuse of notation, and $\tau_i$ is a translation by $1$ in the $i$-th coordinate of $\R^7$ for each $i \in \{1, \dots, 7\}$.
Let
\[
a=\diag(1,-1,-1), \quad
b=\diag(-1,1,-1), \quad
c=\diag(-1,-1,1), \quad
\mathbb{Z}_2^2 \cong \<a,b,c\> \subset \SO(3).
\]
A representation
$\rho : \pi_1^{\orb}(T^7/\Gamma) \rightarrow \SO(3)$
induces a flat connection $\theta$ on a bundle $E_0$ over $T^7/\Gamma$.
We can define $\rho$ by sending the $10$ generators $\alpha, \beta, \gamma, \tau_1, \dots, \tau_7$ of $\pi_1(T^7/\Gamma \setminus S)$ to arbitrary elements in $\<a,b,c\>$.

For any such choice we can carry out our pre-gluing construction as follows.
Let $A_0$ be the product connection on the trivial $\SO(3)$-bundle over Eguchi-Hanson space and $M_0$ its moduli space, which is just a single point.
Then, for a $T^3 \subset T^7/\Gamma$ in the fixed point set of an element in $\Gamma \setminus \{\Id\}$ which is mapped to $\Id \in \SO(3)$, the bundle 
\[
\Fr \times (E_0|_{T^3}) \times _{\U(2) \times G} M_0
\]
over $T^3$ has as its fibre a single point, so there exists a parallel section, which is in particular a Fueter section.

Likewise, let $A_{0,1}$ be the ASD instanton over ${\XEH}$ from \cref{proposition:gocho-asd-instanton}.
 This is defined on a $\U(1)$-bundle and we view it as a reducible $\SO(3)$-connection, and denote its moduli space by $M_{0,1}$.
 This connection has non-trivial holonomy $\rho_{0,1}:\mathbb{Z}_2 \rightarrow \SO(3)$ at infinity, thus $G_{\rho_{0,1}} \subsetneq G$.
 For each copy of $T^3$ fixed by an element in $\Gamma \setminus \{\Id\}$ which is mapped to a non-identity element in $\SO(3)$ we find that
\[
\Fr \times (E_0|_{T^3}) \times _{\U(2) \times G_{\rho_{0,1}}} M_{0,1}
\]
 is again a bundle whose fibre is a single point, so we again have a Fueter section.
 
 We chose a moduli bundle of connections over the singularities coming from $\alpha$, $\beta$, and $\gamma$ matching the monodromy of $\theta$ given by $\rho$.
 For example, if $\rho(\alpha)=\Id$, we chose the moduli bundle of trivial connections $\Fr \times (E_0|_{T^3}) \times _{\U(2) \times G} M_0$ over $\fix(\alpha)$.
 Because of this, the compatibility condition from \cref{assumption:gluing-topological-compatibility} is satisfied.
 
 If $\theta$ is irreducible and infinitesimally rigid, then \cref{theorem:instanton-existence3} guarantees the existence of an irreducible $G_2$-instanton with structure group $\SO(3)$ on the resolution of $T^7/\Gamma$.
 We have the following criterion to check if $\theta$ is irreducible and/or rigid:
 
 \begin{proposition}[Proposition 9.2 in \cite{Walpuski2013a}]
 \label{proposition:criterion-for-flat-connection-irreducible-rigid}
  A flat connection $\theta$ on a $G$-bundle $E$ over a flat $G_2$-orbifold $Y_0$ corresponding to a representation $\rho: \pi_1^{\orb}(Y_0) \rightarrow G$ is irreducible (resp. unobstructed or, equivalently, infinitesimally rigid) if and only if the induced representation of $\pi_1^{\orb}(Y_0)$ on $\mathfrak{g}$ (resp. $\R^7 \otimes \mathfrak{g}$) has no non-zero fixed vectors.
  
  Here, the action on $\mathfrak{g}$ is $\rho$ composed with the adjoint representation, and the action on $\R^7$ is given by identifying $\R^7 \cong T_x(T^7/\Gamma)$ for any fixed $x \in T^7/\Gamma$ and then acting by parallel transport with respect to the Levi-Civita connection of the flat metric.
 \end{proposition} 

In our case, one checks that there are no non-zero vectors in $\so(3)$ or $\R^7 \otimes \so(3)$ fixed by $\pi_1^{\orb}(T^7/\Gamma)$, if at least two of the elements $\tau_1, \dots, \tau_7$ are sent to different non-identity elements in $\<a,b,c\>$.
Thus, the connection $\theta$ is irreducible and infinitesimally rigid in this case by \cref{proposition:criterion-for-flat-connection-irreducible-rigid}.
The number of flat connections which do \emph{not} satisfy this condition is 
\[4^3 \cdot (2^7 \cdot 3 -2).\]
Here, we got the factor $4^3$ for the choice of different values for $\rho$ on $\alpha, \beta,\gamma$.
The term $2^7 \cdot 3$ is the number of choices of values for $\rho$ on $\tau_1, \dots, \tau_7$ contained in $\{\Id, a\}$ or $\{\Id, b\}$ or $\{ \Id, c\}$.
However, we triple counted the choice $\rho(\tau_1)=\dots=\rho(\tau_7)=\Id$, so overall we get the factor $(2^7 \cdot 3 -2)$.
Thus, we have $4^{10}-4^3 \cdot (2^7 \cdot 3 -2)=1024128$ flat connections to which \cref{theorem:instanton-existence3} can be applied, and we get this many $G_2$-instantons on the resolution of $T^7/\Gamma$.

Among these, there are $4^7-(2^7 \cdot 3-2)$ choices for $\rho$ giving rise to a flat $G_2$-instanton on the resolved manifold, namely when the $\alpha, \beta, \gamma$ are all sent to the identity.
There are $1 \cdot 4^7$ choices for $\rho$ in which $\alpha, \beta, \gamma$ are all sent to the identity.
These contain $1 \cdot (2^7 \cdot 3-2)$ cases in which $\tau_1, \dots, \tau_7$ are \emph{not} sent to two different non-identity elements, and these cases had already been removed in the previous paragraph.
Thus, we are left with
\[
 4^{10}-4^3 \cdot (2^7 \cdot 3 -2)
 -
 1 \cdot 4^7
 +
 1 \cdot (2^7 \cdot 3-2)
 =1008126
\]
novel, non-flat examples of $G_2$-instantons on the resolution of $T^7/\Gamma$.

 Some of these instantons will be \emph{essentially the same}.
 That is, they agree under a smooth bijective bundle map covering an isometry preserving the $G_2$-structure of the resolution of $T^7/\Gamma$. 
 This is similar to the case of different resolution data for a $G_2$-orbifold giving rise to diffeomorphic $G_2$-manifolds, cf. \cite[Corollary 12.5.5]{Joyce2000}.
 The isometry group of $T^7$ is $\Isom(\Z^7) \ltimes \R^7/\Z^7$, where $\Isom(\Z^7)$ is the isometry group of the lattice $\Z^7 \subset \R^7$ and $\R^7/\Z^7$ is the group of translations acting on $T^7$.
 One checks with computer aid that the subgroup preserving the flat $G_2$-structure is a group $H \ltimes \R^7/\Z^7$, where $H \subset \Isom(\Z^7)$ is a group with $1344$ elements.
 (The code is published in \cite{Platt2022}.)
 The automorphism group of the $G_2$-orbifold $T^7/\Gamma$ can be identified with the automorphisms $f: T^7 \rightarrow T^7$ with the property 
\[
 \text{for all }
 x \in T^7
 \text{ and for all }
 g_1 \in \Gamma
 \text{ there exists }
 g_2 \in \Gamma
 \text{ such that }
 g_2 f(g_1 x)=f(x).
\] 
 Using this definition, one checks again with computer aid that $\Aut(T^7/\Gamma)=K \ltimes \left\{0, \frac{1}{2} \right\}^7$, where
 \begin{align*}
  K = \langle &
  (\cdot 1, \cdot 1, \cdot 1, \cdot (-1), \cdot (-1), \cdot (-1), \cdot (-1)),
  (\cdot 1, \cdot (-1), \cdot (-1), \cdot 1, \cdot 1, \cdot (-1), \cdot (-1)),
  \\
  &
  (\cdot (-1), \cdot 1, \cdot (-1), \cdot 1, \cdot (-1), \cdot 1, \cdot (-1)) \rangle.
 \end{align*}
 Here, the group elements of $K$ act on $T^7$ through multiplication of coordinates by $1$ or $(-1)$.
 Each automorphism is covered by four bundle maps corresponding to the four elements of the group $\< a,b,c \> \subset \SO(3)$.
 Thus, there are at least $246 < 1008126/(|K| \cdot 4)$ different flat connections that are pairwise not essentially the same, giving rise to just as many non-flat $G_2$-instantons that are pairwise not essentially the same.
 
  \begin{corollary}
 \label{corollary:g2-instantons-on-resolution-of-T7}
  Let $\Gamma$ act on $T^7$ as defined in \cref{equation:alpha-beta-gamma-t7} and let $N'_t$ denote the one parameter family of resolutions of $T^7/\Gamma$ from \cref{section:torsion-free-on-resolution-of-T7}.
  Then, for $t$ small enough, there are at least $246$ non-flat, irreducible $G_2$-instantons with structure group $\SO(3)$ over $N'$ with the property that no one of them is mapped onto another one of them under a smooth bijective bundle map covering an isometry $N' \rightarrow N'$ that preserves the $G_2$-structure on $N'$.
 \end{corollary}

\subsection{An example Coming from a Stable Bundle}
\label{section:stable-bundle-example}

\subsubsection{Review of the Resolution of $(T^3 \times \text{K3})/\Gamma$}
\label{subsubsection:branched-double-cover}

Recall the $G_2$-manifold constructed in \cite[Section 7.3]{Joyce2017}:
consider the sextic
\[
 C=
 \{
  [z_0,z_1,z_2] \in \CP^2:
  z_0^6+z_1^6+z_2^6=0
 \}
 \subset \CP^2
\]
and let $\pi: X \rightarrow \CP^2$ be the double cover of $\CP^2$ branched over $C$.
Then $X$ is a complex K3 surface with a Hyperkähler triple of Kähler forms $\omega^I, \omega^J, \omega^K$, cf. \cite[Example 1.3]{Huybrechts2016}.
On $X$ we can define the following two maps:
first, the map $\alpha: X \rightarrow X$ which swaps the two sheets of the branched cover.
Second, there are two lifts $X \rightarrow X$ of the complex conjugation map $\sigma: \CP^2 \rightarrow \CP^2$.
One of these two lifts acts freely on $X$, the other one does not.
Denote the lift that does not act freely on $X$ by $\beta: X \rightarrow X$, which has $\fix (\beta)=\pi^{-1}(\RP^2) \simeq S^2$.
The Hyperkähler triple $\omega^I, \omega^J, \omega^K$ can be chosen to satisfy
\begin{align*}
 \alpha^* \omega^I&=\omega^I,
 &
 \alpha^* \omega^J&=-\omega^J,
 &
 \alpha^* \omega^K&=-\omega^K,
 \\
 \beta^* \omega^I&=-\omega^I,
 &
 \beta^* \omega^J&=\omega^J,
 &
 \beta^* \omega^K&=-\omega^K.
\end{align*}
Let $\alpha, \beta$ act on $T^3$ via
\[
 \alpha(x_1,x_2,x_3)
 =
 (x_1,-x_2,-x_3),
 \quad
 \beta(x_1,x_2,x_3)
 =
 \left(
  -x_1,x_2,\frac{1}{2}-x_3
 \right).
\]
Denote $\Gamma = \< \alpha, \beta \>$.
Then $\alpha, \beta: T^3 \times X \rightarrow T^3 \times X$ preserve the product $G_2$-structure $\varphi$ on $T^3 \times X$ defined by equation \cref{equation:product-g2-structure}.
Furthermore, $\fix(\alpha)=4 \cdot( S^1 \times C)$, $\fix(\beta)=4 \cdot(S^1 \times S^2)$, where the $S^2$-factors are the double cover of $\fix(\sigma)=\RP^2 \subset \CP^2$.
Therefore, $L=\fix(\alpha) \cup \fix(\beta)$ admits a nowhere vanishing harmonic $1$-form, namely the parallel $1$-form in the $S^1$-direction of each component.
Thus, this orbifold is of the type considered in \cref{section:review-of-the-manifold-construction} and its resolution $N_t \rightarrow (T^3 \times X)/ \Gamma$ admits a $1$-parameter family of $G_2$-structures with small torsion, inducing metrics $g_t$, which can be perturbed to torsion-free $G_2$-structures inducing metrics $\tilde{g}_t$.

\subsubsection{A Connection on the Orbifold $(T^3 \times \text{K3})/\Gamma$ coming from a Stable Bundle}
\label{subsubsection:connection-on-orbifold-coming-from-stable-bundle}

The tangent bundle $E$ of $\CP^2$ is a complex vector bundle of rank $2$, which has an associated $\SO(3)=\PU(2)$-bundle $F$.
The Levi-Civita connection on $E$ is a Hermite-Einstein connection and induces an ASD instanton on $F$, denoted by $A$.
We denote the standard Kähler structure on $\CP^2$ by $(J, g=g_{\text{FS}}, \omega)$, where $g_{\text{FS}}$ is the Fubini-Study metric.
The pullback $\pi^*A$ is then an ASD instanton on the bundle $\pi^*F$ over $(X,\pi^* g)$, but it need not be ASD with respect to the Calabi-Yau metric on $X$.
We will show in \cref{proposition:pullback-of-TCP2-is-stable-for-calabi-yau} that $\pi^*F$ also carries an instanton with respect to the Calabi-Yau metric.

\begin{proposition}[Lemma 9.1.9 in \cite{Donaldson1990}]
\label{proposition:pullback-of-TCP2-is-stable}
 The bundle $\pi^* E$ is stable with respect to $\omega$.
\end{proposition}

\begin{corollary}
\label{proposition:pullback-of-TCP2-is-stable-for-calabi-yau}
 The bundle $\pi^* E$ is stable with respect to the Calabi-Yau Kähler form $\omega^I$.
\end{corollary}

\begin{proof}[Proof of \cref{proposition:pullback-of-TCP2-is-stable-for-calabi-yau}]
 Denote by $\hat{\omega}=\pi^* \omega$ the pullback of the Kähler form for the Fubini-Study metric on $\CP^2$ to $X$.
 By Yau's proof of the Calabi conjecture we have that $\omega^I = \hat{\omega} + i \del \delbar \phi$ for some $\phi : X \rightarrow \R$.
 In particular, $\omega^I$ and $\hat{\omega}$ are in the same de Rham cohomology class.
 
 By \cref{proposition:pullback-of-TCP2-is-stable}, $\pi^* E$ is stable with respect to $\omega$.
 The Kähler form enters into the definition of stability only through the definition of slope.
 But slopes do not change when switching between $\omega^I$ and $\hat{\omega}$ as they are in the same cohomology class.
 Thus $\pi^* E$ is also stable with respect to $\omega^I$.
\end{proof}

We also have the following:

\begin{corollary}[p. 348 in \cite{Donaldson1990}]
\label{proposition:unique-ASD-on-K3}
 Denote by $\pi_F:F \rightarrow \CP^2$ the $\SO(3)$-bundle over $\CP^2$ from \cref{subsubsection:connection-on-orbifold-coming-from-stable-bundle}.
 Let $\pi: X \rightarrow \CP^2$ be the branched double cover from \cref{subsubsection:branched-double-cover} with Calabi-Yau metric $\hat{g}$.
 Then the bundle
 \begin{align}
 \label{equation:pullback-bundle}
  \hat{F}
  =
  \pi^* F
  =
  \{
   (x,u) \in X \times F
   :
   \pi_F(u)=\pi(x)
  \}
 \end{align}
 admits an infinitesimally rigid and unobstructed ASD instanton $\hat{A}$ with respect to $\hat{g}$.
\end{corollary}

Pulling back $(\hat{F}, \hat{A})$ under the projection onto the second factor, $p: T^3 \times X \rightarrow X$, gives a bundle with $G_2$-instanton by \cref{example:pullback-of-asd-is-g2-instanton}.
Denote the bundle by $E_0$ and the connection by $\theta$.
The connection $\hat{A}$ was infinitesimally rigid, and the following proposition, which is proved like \cref{proposition:constant-in-r3-direction}, implies that $\theta$ is infinitesimally rigid:

\begin{proposition}
 Let $I$ be an ASD instanton on a bundle $P$ over a compact $4$-fold $Y$ with deformation operator $\delta_I$.
 Let $p: T^3 \times Y \rightarrow Y$ be the projection onto the second factor.
 Then the $G_2$-instanton $p^*I$ is infinitesimally rigid if and only if $I$ is infinitesimally rigid and unobstructed.
\end{proposition}

The gluing theorems \cref{theorem:instanton-existence,theorem:instanton-existence3} require a connection on the orbifold, $(T^3 \times X)/\Gamma$.
The following proposition states that $\theta$ can be viewed as such a connection:

\begin{proposition}
\label{proposition:bundle-lifts-of-alpha-beta}
 There exist lifts 
 $\alpha_0: E_0 \rightarrow E_0$ 
 of $\alpha$
 and 
 $\beta_0: E_0 \rightarrow E_0$ 
 of $\beta$
 such that 
 $\alpha_0^2=\beta_0^2=\Id$, 
 $\alpha_0^* \theta=\beta_0^* \theta=\theta$, 
 $\alpha_0$ being the identity over $\fix(\alpha)$, 
 and $\beta_0$ \emph{not} being the identity over $\fix(\beta)$.
\end{proposition}

This relies on the following construction on $X$:

\begin{proposition}
\label{proposition:X-conjugation-lift}
 There exists a lift $\hat{\beta}: \hat{F} \rightarrow \hat{F}$ of $\beta$ such that $\hat{\beta}^2=\Id$, $\hat{\beta}^* \hat{A}=\hat{A}$, and $\hat{\beta}$ not being the identity over $\fix(\beta)$.
\end{proposition}

\begin{proof}
 Denote by $\sigma: \CP^2 \rightarrow \CP^2$ the conjugation map and $E=T\CP^2$ as before.
 We can then view $\d \sigma$ as a complex linear map $E \rightarrow \overline{E}$ covering $\sigma$.
 Define
 \begin{align}
 \label{equation:sigma-hat}
 \begin{split}
  \hat{\sigma}:
  E \tensor \overline{E} & \rightarrow E \tensor \overline{E}
  \\
  v \tensor w
  & \mapsto
  -\d \sigma w \tensor \d \sigma v,
 \end{split}
 \end{align}
 which is a complex linear map covering $\sigma : \CP^2 \rightarrow \CP^2$.
 
 The manifold $\CP^2$ is Kähler-Einstein, so the Levi-Civita connection $\nabla^{\text{LC}}$ on $E$ is a Hermite-Einstein connection.
 The connection $\nabla^{\text{LC}}$ on $E$ induces the product connection $\nabla^{\tensor}$ on $E \tensor \overline{E}$, which is again a Hermite-Einstein connection.
 We have that $\sigma$ is an isometry, so $\nabla^{\tensor}$ is preserved by $\hat{\sigma}$ in the sense of $\hat{\sigma} \circ \sigma^* \nabla^\tensor \circ \hat{\sigma}=\nabla^\tensor$.
 
 Let $\hat{\beta}$ be the lift of $\hat{\sigma}$ to $\pi^* E \tensor \overline{\pi^* E}$, i.e. $\hat{\beta}: \pi^* E \tensor \overline{\pi^* E} \rightarrow \pi^* E \tensor \overline{\pi^* E}$ covering $\beta: X \rightarrow X$ and satisfying
 $p \hat{\beta}=\hat{\sigma}p$, where $p:  \pi^* E \tensor \overline{\pi^* E} \rightarrow E \tensor \overline{E}$ is the obvious projection map.
 Then $\hat{\sigma}^* \nabla^\tensor=\nabla^\tensor$ implies $\hat{\beta}^* (\pi^* \nabla^\tensor)=\pi^* \nabla^\tensor$.
 
 If $p \in \CP^2$ and $(u_1,u_2)$ is a unitary basis of $E_p$, then $(\d \sigma(u_1), \d \sigma(u_2))$ is a unitary basis of $E_{\sigma(p)}$, and writing elements of the trace-free unitary endomorphism bundle $\mathfrak{u}_0(\pi^*E)$ in these bases, we see that $\hat{\beta}$ acts as
 \begin{align*}
  \begin{pmatrix}
   0&1\\
   -1&0
  \end{pmatrix}
  \mapsto
  \begin{pmatrix}
   0&1\\
   -1&0
  \end{pmatrix},
  \quad
  \begin{pmatrix}
   0&i\\
   i&0
  \end{pmatrix}
  \mapsto
  -\begin{pmatrix}
   0&i\\
   i&0
  \end{pmatrix},
  \quad
  \begin{pmatrix}
   i&0\\
   0&-i
  \end{pmatrix}
  \mapsto
  -
  \begin{pmatrix}
   i&0\\
   0&-i
  \end{pmatrix}.
 \end{align*}
 Thus, $\hat{\beta}$ induces a map on $\hat{F}=\SO(\mathfrak{u}_0(\pi^*E))$ that is not the identity over $\fix (\beta)$ and preserves the ASD connection $\hat{A}$ on $\hat{F}$ induced by $\pi^* \nabla^\tensor$.
\end{proof}

\begin{remark}
 This only works because we have a lift of complex conjugation $\sigma: \CP^2 \rightarrow \CP^2$ to $F$ in \cref{proposition:X-conjugation-lift}.
 No lift of $\sigma$ to $E$ exists, because $c_1(\sigma^*E)=-c_1(E)$, so it is important to change from $\U(2)$-bundles to $\SO(3)$-bundles in this example.
\end{remark}

\begin{remark}
 Without the minus sign in \cref{equation:sigma-hat}, $\hat{\beta}$ would not descend to a map on $\SO(\mathfrak{u}_0(\pi^*E))$.
 That is because the map $-\Id: \mathfrak{u}_0(\pi^*E) \rightarrow \mathfrak{u}_0(\pi^*E)$ is orientation reversing, as $\mathfrak{u}_0(\pi^*E)$ has odd rank.
\end{remark}

\begin{proof}[Proof of \cref{proposition:bundle-lifts-of-alpha-beta}]
 The bundle $\hat{F}$ from \cref{equation:pullback-bundle} is the pullback of a bundle $F$ from $\CP^2$ to $X$, thus we have the natural map
 \begin{align*}
  \hat{\alpha}:\hat{F} &\rightarrow \hat{F}
  \\
  (x,u)
  &\mapsto
  (\alpha(x),u)
 \end{align*}
 covering $\alpha: X \rightarrow X$.
 The bundle $E_0$ is the pullback of $\hat{F}$ to $T^3 \times X$, and we can canonically extend the map $\hat{\alpha}$ and the map $\hat{\beta}$ from \cref{proposition:X-conjugation-lift} to $E_0$ and find that they have the required properties.
\end{proof}

Because of \cref{proposition:bundle-lifts-of-alpha-beta}, the connection $\theta$ defines a connection on the orbifold $(T^3 \times \text{K3})/\Gamma$.
The holonomy of $\theta$ around the four $S^1 \times C \subset (T^3 \times X)/\Gamma$ fixed by $\alpha$ is trivial, and the holonomy around the four $S^1 \times S^2$ fixed by $\beta$ has order $2$.

\subsubsection{The Resulting Connection on the Resolution of $(T^3 \times \text{K3})/\Gamma$}

\begin{corollary}
\label{corollary:non-trivial-g2-instanton}
 Let $N_t$ denote the one parameter family of resolutions of $(T^3 \times X)/\Gamma$ from \cref{subsubsection:branched-double-cover}.
 Then, for $t$ small enough, there exists an irreducible $G_2$-instanton with structure group $\SO(3)$ over the resolution $N_t$.
\end{corollary}

\begin{proof}
 We make use of the $\alpha$-invariant and $\beta$-invariant connection $\theta$ from \cref{proposition:bundle-lifts-of-alpha-beta} over $(T^3 \times X)/\Gamma$.
 
 Next consider the product connection $A_0$ on the trivial $\SO(3)$-bundle over Eguchi-Hanson space ${\XEH}$.
 Like in \cref{subsection:examples-on-t7-mod-gamma}, we get a constant Fueter section on each connected component of $\fix (\alpha)=4 \cdot (S^1 \times C)$, i.e.
 \[
  S^1 \times C \rightarrow \Fr \times E_0|_{S^1 \times C}\times_{\U(2) \times G} M_0.
 \] 
 Likewise, let $A_{0,1}$ be the ASD instanton over ${\XEH}$ from \cref{proposition:gocho-asd-instanton}.
 As in \cref{subsection:examples-on-t7-mod-gamma}, we get a constant Fueter section on each connected component of $\fix (\beta)=4 \cdot (S^1 \times S^2)$, i.e.
 \begin{align*}
  S^1 \times S^2 & \rightarrow \Fr \times E_0|_{S^1 \times S^2}\times_{\U(2) \times G_{\rho_{0,1}}} M_{0,1}.
 \end{align*}
 By \cref{proposition:bundle-lifts-of-alpha-beta}, the connection $\theta$ and the eight Fueter sections satisfy the necessary compatibility condition from \cref{definition:approximate-solution}.
 Thus, \cref{theorem:instanton-existence} applies and gives a $G_2$-instanton $\tilde{A}_t$ on $N_t$.
 The connections $\tilde{A}_t$ converge to $\theta$ on compact subsets of $(T^3 \times X)/\Gamma \setminus \fix(\Gamma)$ as $t \rightarrow 0$.
 The connection $\theta$ has full holonomy $\SO(3)$, as otherwise the Fubini-Study metric on $\CP^2$ would need to have reduced holonomy.
 Thus, $\tilde{A}_t$ has full holonomy for small $t$ and is therefore irreducible.
\end{proof}

\appendix
\section{Appendix}

\subsection{The Isometry Group of Eguchi-Hanson Space}

The following result is well known, but we were unable to locate a proof in the literature, so we provide it here.

\begin{proposition}
\label{proposition:eguchi-hanson-isometry-group}
 The group of holomorphic isometries of $\XEH$ is isomorphic to $\U(2)/\{\pm 1\}$.
\end{proposition}

\begin{proof}
We use the notation from the description of $\XEH$ as a Hyperkähler reduction from before \cref{definition:eguchi-hanson-hyperkaehler}.
We view $\SU(2)$ embedded in $\mathbb{H}^{2 \times 2}$ as quaternion valued matrices with no $j$ or $k$ components.
Then $\SU(2)$ acts on $\mathcal{M}$ by right multiplication.
This action restricts to $\mu^{-1}(\zeta)$ and commutes with the action of $\U(1)$.
The action is not effective, as $-1 \in \SU(2)$ acts trivially, but the induced action of the quotient group $\SU(2)/\{ \pm 1\} \simeq \SO(3)$ is effective.
Next, let $\SO(2)$ act on $\mathcal{M}$ from the left via
\[
 q_a \mapsto e^{it} \cdot q_a , \quad t \in (0, 2 \pi].
\]
Again, the action restricts to $\mu^{-1}(\zeta)$ and commutes with the action of $\U(1)$, but is not effective as $-1 \in \SO(2)$ acts trivially.
The actions of $\SO(2)/\{\pm 1\}$ and $\SU(2)/\{ \pm 1\}$ commute, as the first group is acting from the left, the second is acting from the right.
We thus get that the group $\SO(2)/\{\pm 1\} \times \SU(2)/\{ \pm 1\}$ acts through isometries on $\XEH$.
Last, one readily confirms that the map
\begin{align*}
U(1)/\{\pm 1\} \times \SU(2)/\{ \pm 1\} &\rightarrow \U(2)/\{ \pm 1\}
\\
[\lambda],[A] & \mapsto [\lambda A]
\end{align*}
is a group isomorphism.
Its inverse is given by $[B] \mapsto ([\sqrt{\det B}],[B/\sqrt{\det B}])$ which is not well-defined as a map $\U(1) \times \SU(2) \rightarrow \U(2)$ but is well-defined after dividing out $\{\pm 1\}$.
\end{proof}

\begin{remark}
 One may also recover the full isometry group of the Eguchi-Hanson space by noticing that there is an additional isometry induced by the map on $\mathcal{M}$ that swaps coordinates, i.e. $\mathcal{M} \rightarrow \mathcal{M}$, $(q_1,q_2) \mapsto (q_2,q_1)$.
 The group of \emph{all} isometries (not necessarily holomorphic) of $\XEH$ is isomorphic to $\SO(3) \times \O(2)$.
 The group of triholomorphic isometries of $\XEH$ is isomorphic to $\SO(3)$.
\end{remark}

\subsection{Rigidity of Finite Subgroups}

Let $G$ be a compact connected Lie group and $\Gamma$ be a finite group.
In \cref{subsubsection:gauge-theory-on-ale} we took $\Gamma$ to be a finite subgroup of $\SU(2)$, thereby acting on $B^4$.
An orbifold $G$-bundle over $B^4/\Gamma$ is a $G$-bundle $P$ over $B^4$ together with a lift of the action of $\Gamma$ to $P$.
In \cref{equation:ALE-moduli-definitions} we extended elements of $G$ to elements of the orbifold gauge group $\mathscr{G}(P)$.
We could do this, because we assumed the lift of $\Gamma$ to act in a standard way on $P$, see \cref{equation:asymptotic-at-infinity-condition2} for the precise statement.
In other words:
we used that up to gauge equivalence, orbifold bundles over $B^4/\Gamma$ are determined by the homomorphism $\Gamma \rightarrow P_0 \simeq G$ induced by the lift of $\Gamma$ to $P$.
The proof of this fact was given in \cref{assumption:trivialisation-at-infinity}, but used that the homomorphism $\Gamma \rightarrow G$ is rigid, in some sense.
We make this rigidity precise here and prove that every finite group in a compact Lie group is rigid.
The proof is taken from \cite{Bader2021}, where also the generalisation to non-compact $G$ is explained.

\begin{definition}
 The set $\Hom(\Gamma,G) \subset G^{ |\Gamma| }$ endowed with the restriction of the product topology on $G^{ |\Gamma| }$ is called the \emph{representation variety}.
\end{definition}

\begin{definition}
 Let $E$ be a $\Gamma$-module.
 A map $b \in \Gamma \rightarrow E$ is called \emph{cocycle} if
 \[
  b(\gamma \delta)
  =
  b(\gamma)+\gamma \cdot b(\delta)
  \text{ for all }
  \gamma,\delta \in \Gamma.
 \]
 We denote the set of cocycles by $Z^1(\Gamma,E)$.
 A map $b \in \Gamma \rightarrow E$ is called \emph{coboundary} if there exists $v \in E$ such that
 \[
  b(\gamma)
  =
  v-\gamma \cdot v
  \text{ for all }
  \gamma \in \Gamma.
 \]
 We denote the set of coboundaries by $B^1(\Gamma,E) \subset Z^1(\Gamma,E)$.
 The \emph{first cohomology of $\Gamma$ with coefficients in $E$ is}
 \[
  H^1(\Gamma,E)
  =
  Z^1(\Gamma,E)/B^1(\Gamma,E).
 \]
\end{definition}

\begin{theorem}[Point 3 in \cite{Weil1964}]
\label{theorem:weils-local-rigidity}
 Fix a group homomorphism $r: \Gamma \rightarrow G$.
 The group $G$ is acting on $\mathfrak{g}$ through the adjoint representation, and together with $r$ this gives $\Gamma$ the structure of a $\Gamma$-module.
 If $H^1(\Gamma, \mathfrak{g})=0$, then there exists a neighbourhood $U \subset \Hom(\Gamma,G)$ of $r$ in which each element is conjugate to $r$, i.e. for all $s \in U$ there exists $g \in G$ such that
 \[
  s
  =
  l_g \circ r_{g^{-1}} \circ r.
 \]
 Here, $l_g,r_{g^{-1}}:G \rightarrow G$ denote left translation and right translation on $G$, respectively.
\end{theorem}

\begin{definition}
 Fix $\pi:\Gamma \rightarrow \Aut(E)$.
 An \emph{affine action} of $\Gamma$ on $E$ is a group homomorphism $\phi: \Gamma \rightarrow \Aff(E)$.
 We say that $\pi$ is the \emph{linear part} of the affine action $\phi$ if for all $\gamma \in \Gamma$ there exists $v_0 \in E$ such that
 \[
  \phi(\gamma)(v)
  =
  \pi(\gamma)(v)+v_0
  \text{ for all }
  v \in E.
 \]
\end{definition}

\begin{lemma}[Lemma 2.1 in \cite{Dymarz2016}]
\label{lemma:group-cohomology-vanishing-characterisation}
 The map $\pi: \Gamma \rightarrow \Aut(E)$ endows $\Gamma$ with an $E$-module structure.
 We have $H^1(\Gamma, E)=0$ with respect to this $E$-module structure if and only if every affine action with linear part $\pi$ has a fixed point. 
\end{lemma}

\begin{corollary}
\label{corollary:finite-group-has-zero-group-cohom}
 The finite group $\Gamma$ with any $E$-module structure satisfies $H^1(\Gamma,E)=0$.
\end{corollary}

\begin{proof}
 Let $\phi: \Gamma \rightarrow \Aff(E)$ be an affine action.
 Then the element
 \[
  X:=
  \sum_{\delta \in \Gamma}
  \phi(\delta)(0)
  \in E
 \]
 satisfies $\phi(\gamma)(X)=X$ for all $\gamma \in \Gamma$.
 By \cref{lemma:group-cohomology-vanishing-characterisation} this implies that $H^1(\Gamma, E)=0$.
\end{proof}

\begin{corollary}
\label{corollary:rep-variety-components-come-from-conjugation}
 The representation variety $\Hom(\Gamma,G)$ has finitely many connected components.
 For each connected component $C$ there exists $r \in \Hom(\Gamma,G)$ such that
 \[
  C
  =
  U_r
  :=
  \{
   l_g \circ r_{g^{-1}} \circ r:
   g \in G
  \}.
 \]
\end{corollary}

\begin{proof}
 Because $\Gamma$ is finite and $G$ is compact we have that $\Hom(\Gamma,G)$ is compact and therefore has finitely many connected components.
 Fix some $r \in \Hom(\Gamma,G)$.
 Then $U_r$ is compact because it is the image of $G$ under the conjugation map.
 Thus, $U_r$ is closed.
 On the other hand, $U_r$ is open by \cref{theorem:weils-local-rigidity} together with \cref{corollary:finite-group-has-zero-group-cohom}.
 Thus, each connected component of $\Hom(\Gamma,G)$ is of the form $U_r$ for some $r \in \Hom(\Gamma,G)$.
\end{proof}

\subsection{Removable Singularities}

In \cref{definition:psi-function-moduli-space-iso} we defined a map from the moduli space of ASD connections over the Eguchi-Hanson space ${\XEH}$ into the moduli space of ASD connections over the one point compactification of ${\XEH}$.
There, we used that every finite energy ASD connection that is defined over the complement of a point can be extended over this point.
This statement was proved for Yang-Mills connections, not just ASD connections, in \cite{Uhlenbeck1982}.
This is called the \emph{Removable Singularities Theorem}.
Because our map between moduli spaces should be a map between \emph{framed} moduli spaces, we need a version of the Removable Singularities Theorem that respects framings.
This is \cref{proposition:removable-singularities-uniqueness} and we then apply it to our special case of connections over ${\XEH}$ in \cref{corollary:removable-sing-over-ale-limit-at-infinity}.

\begin{theorem}[Theorem 4.1 in \cite{Uhlenbeck1982}, Theorem D.1 in \cite{Freed1991}]
\label{theorem:original-removable-singularities}
 Let $G$ be a compact Lie group and $A$ be a connection on the trivial $G$-bundle over $B^4 \setminus \{0\}$, $A \in \mathscr{A}((B^4 \setminus \{0\}) \times G)$, which is in $L^2_{1,\text{loc}}$ and anti-self-dual with respect to a smooth metric on $B^4$.
 If
 \[
  \int_{B^4 \setminus \{0\}}
  |F(A)|^2
  < \infty,
 \]
 then there exists an injective bundle homomorphism $\xi: (B^4 \setminus \{0\}) \times G \rightarrow B^4 \times G$ and a smooth connection $A' \in \mathscr{A}(B^4 \times G)$ such that $\xi^* A'=A$ over $B^4 \setminus \{0\}$.
\end{theorem}

\Cref{theorem:original-removable-singularities} asserts existence of an extension over $0$, and the following proposition asserts that this extension is essentially unique up to gauge:

\begin{proposition}
\label{proposition:removable-singularities-uniqueness}
 The data $\xi$ and $A'$ from \cref{theorem:original-removable-singularities} are unique in the following sense:
 if $\xi', \xi'': (B^4 \setminus \{0\}) \times G \rightarrow B^4 \times G$ and $A',A'' \in \mathscr{A}(B^4 \times G)$ are such that $(\xi')^* A'=(\xi'')^* A''=A$, then the map $\xi'' \circ (\xi')^{-1}: (B^4 \setminus \{0\}) \times G \rightarrow (B^4 \setminus \{0\}) \times G$ can be extended to a continuous map $B^4 \times G \rightarrow B^4 \times G$.
\end{proposition}

\begin{proof}
 We view the connections $A',A''$ on the trivial bundle $B^4 \times G$ as elements in $\Omega^1(B^4 , \mathfrak{g})$, and view the gauge transformation $\xi'' \circ (\xi')^{-1}$ as a map $B^4 \setminus \{0\} \rightarrow G$, denoted by $s$.
 Without loss of generality assume that $A'(0)=A''(0)=0$, which can be arranged by composing $\xi', \xi''$ with a suitable gauge transformation of $B^4 \times G$.
 Then $A''=s^*A'$ on $B^4 \setminus \{0\}$, thus
 \[
  0=
  A''(0)
  =
  \lim_{x \rightarrow 0}
  s^{-1}(x) \d s(x)
 \]
 and by taking norms we see that $\lim_{x \rightarrow 0} \d s(x)=0$.
 This implies that $\lim_{x \rightarrow 0} s(x)$ exists:
 if the limit does not exist, then we have two sequences $x_i, x_i' \rightarrow 0$ such that $\lim_{i \rightarrow \infty} s(x_i) \neq \lim_{i \rightarrow \infty} s(x_i')$.
 Without loss of generality assume that $x_i, x_i'$ can be joined by a line.
 The mean value theorem then gives a sequence $\theta_i \in B^4 \setminus \{0\}$ such that $| \d s(\theta_i)| \rightarrow \infty$, which is a contradiction.
 
 Therefore $\lim_{x \rightarrow 0} s(x)$ exists and defines a continuous map $\overline{s}:B^4 \rightarrow G$, which in turn extends $\xi'' \circ (\xi')^{-1}$.
\end{proof}

Viewing the map $\xi$ from \cref{theorem:original-removable-singularities} as a map $\xi:B^4 \setminus \{0\} \rightarrow G$, the limit $\lim_{x \rightarrow 0} \xi(x)$ does not exist in general.
But in important cases it does, according to the following proposition:

\begin{proposition}
\label{proposition:bounded-implies-xi-limit-exists}
 Under the conditions of \cref{theorem:original-removable-singularities}, assume that $A$ is bounded, viewed as an element in $\Omega^1(B^4 \setminus \{0\}, \mathfrak{g})$.
 Viewing $\xi$ as a map $\xi:B^4 \setminus \{0\} \rightarrow G$, we have that the limit 
 \[
 \lim_{x \rightarrow 0} \xi(x)
 \in G
 \]
exists.
\end{proposition}

\begin{proof}
 Without loss of generality assume that $A'(0)=0$.
 Then,
 \begin{align}
 \label{equation:removable-singularities-boundedness}
  \xi^*A'(x)=A(x)
  \text{ for all }
  x \in B^4 \setminus \{0\}.
 \end{align}
 Taking norms in \cref{equation:removable-singularities-boundedness} and using $\xi^*A'(x)=\xi^{-1}(x) \d \xi (x)+A'(x)$ we see that $\d \xi$ is bounded on $B^4 \setminus \{0\}$, and we can conclude the proof as in the proof of \cref{proposition:removable-singularities-uniqueness}.
\end{proof}

This can be applied to the case of ASD instantons on ALE manifolds:

\begin{corollary}
\label{corollary:removable-sing-over-ale-limit-at-infinity}
 Let $P$ be a $G$-bundle over ${\XEH}$ and denote by $\mathscr{A}^{\asd, -2}$ the set of ASD-connections on $P$ as in \cref{equation:ALE-moduli-definitions}.
 Let $A_0+a \in \mathscr{A}^{\asd, -2}$, then there exists an orbifold $G$-bundle $P'$ over ${\XEHh}$ together with a connection $A' \in \mathscr{A}(P')$ and an injective bundle homomorphism $\xi:P \rightarrow P'$ such that $\xi^*A'=A_0+a$.
 Denote by $f: B^4/ \Gamma \rightarrow V$ the chart of ${\XEHh}$ around $\infty$ from \cref{proposition:one-point-compactification-of-eguchi-hanson}.
 Fixing a trivialisation of $P$ over $V \setminus \{ \infty \}$ induces a trivialisation of $P'$ over $V$ and we can view $\xi$ as a map $V \setminus \{ \infty \} \rightarrow G$.
 Then the limit $\lim_{x\rightarrow \infty} \xi (x)$, where $\infty \in {\XEHh}$, exists.
\end{corollary}

\begin{proof}
 The assumption $A_0+a \in \mathscr{A}^{\asd, -2}$ means that $a=\mathcal{O}(r^{-2})$, measured in the ALE metric.
 By inspecting how the inversion $f$ acts on $1$-forms, we find that $a=\mathcal{O}(1)$, measured in the orbifold metric, and \cref{proposition:bounded-implies-xi-limit-exists} gives the claim.
\end{proof}

\addcontentsline{toc}{section}{References}

\bibliographystyle{alpha}
\bibliography{g2-inst-on-resolutions-library.bib}

\end{document}